\documentclass[a4paper,11pt,reqno]{article}
\usepackage[T1]{fontenc}
\usepackage[utf8]{inputenc}
\usepackage{mathrsfs}
\usepackage[hmargin=1in,vmargin=1.5in]{geometry}
\usepackage{ulem}
\usepackage{bbm}
\usepackage{textcomp}
\usepackage{amsmath}
\usepackage{graphicx,epsfig,color}
\usepackage{amssymb,amsthm}
\usepackage{esint}
\usepackage[colorlinks,citecolor=blue,urlcolor=blue]{hyperref}
\usepackage{tikz}

\newcounter{moncompteur}
\renewenvironment{enumerate}{\begin{list}{\textup{(}\textit{\roman{moncompteur}}\textup{)}}{\usecounter{moncompteur}}}{\end{list}}

\newcommand{\ddr}{\mathrm{d}}
\newcommand{\R}{\mathbb{R}}
\newcommand{\E}{\mathbb{E}}
\renewcommand{\P}{\mathbb{P}}
\newcommand{\N}{\mathbb{N}}

\newcommand{\egaldistr}{\overset{\mathcal{L}}{=}}
\newcommand{\dd}{\mathrm{d}}
\newcommand{\ind}[1]{\mathbf{1}_{\left\{#1\right\}}}
\newcommand{\floor}[1]{{\left\lfloor  #1 \right\rfloor}}
\newcommand{\crochet}[1]{{\left\langle #1 \right\rangle}}
\renewcommand{\bar}[1]{\overline{#1}}
\newcommand{\Coag}{\mathrm{Coag}}

\theoremstyle{plain}
\newtheorem{theorem}{Theorem}[section]
\newtheorem{theostar}{Theorem}[section]

\newtheorem{lemma}[theorem]{Lemma}

\newtheorem{corollary}[theorem]{Corollary}
\newtheorem{proposition}[theorem]{Proposition}

\theoremstyle{definition}
\newtheorem{definition}[theorem]{Definition}

\theoremstyle{remark}
\newtheorem{remark}[theorem]{Remark}

\date{\today}
\title{ \bf Coalescences in Continuous-State Branching Processes}
\author{  \bf Cl\'ement Foucart \thanks{ LAGA, Université Paris 13, France. \textrm{foucart@math.univ-paris13.fr}} \and \bf Chunhua Ma \thanks{School of Mathematical Sciences and LPMC, Nankai University, China. \textrm{mach@nankai.edu.cn}} \and \bf Bastien Mallein \thanks{LAGA, Université Paris 13, France. \textrm{mallein@math.univ-paris13.fr}}}

\renewcommand{\emph}[1]{\textit{#1}}
\numberwithin{equation}{section}

\begin{document}

\maketitle
\vspace{0.5cm}

\begin{abstract}
Consider a continuous-state branching population 
constructed as a flow of nested subordinators. Inverting the subordinators and reversing time give rise to a flow of coalescing Markov processes (with negative jumps) which correspond to the ancestral lineages of individuals in the current generation. The process of the ancestral lineage of a fixed individual is the Siegmund dual process of the continuous-state branching process. We study its semi-group, its long-term behavior and its generator. In order to follow the coalescences in the ancestral lineages and to describe the backward genealogy of the population, we define non-exchangeable Markovian coalescent processes obtained by sampling independent Poisson arrival times over the flow. These  coalescent processes are called consecutive coalescents, as only consecutive blocks can merge. They are characterized in law by finite measures on $\mathbb{N}$ which can be thought as the offspring distributions of some inhomogeneous immortal Galton-Watson processes forward in time.
\end{abstract}
\textbf{Keywords:} Branching processes, coalescent processes, continuous-state branching processes, flow of subordinators,  genealogy, duality.

\section*{Introduction}
\label{sec:introduction}

Random population models can be divided in two classes, those with constant finite size and those whose size is varying randomly. It is known since the 2000s that populations with constant finite size, evolving by resampling, have genealogies given by exchangeable coalescents. These 
processes, defined by M\"ohle and Sagitov \cite{MR1880231}, Pitman \cite{Pitman}, Sagitov \cite{MR1742154} and Schweinsberg \cite{MR1781024}, are generalisations of Kingman's coalescent for which multiple coalescences of ancestral lineages are allowed. They correspond to the genealogy backward in time of so-called generalized Fleming-Viot processes. Those processes, which can be seen as scaling limits of Moran models \cite{MR0127989}, were defined and studied by Donnelly and Kurtz \cite{DonnKurtz} (via a particle system called lookdown construction) and by Bertoin and Le Gall \cite{MR1990057} (via flows of exchangeable bridges). Both constructions are similar in many aspects and are summarized via the notion of flow of partitions, see Labbé \cite{MR3227064,MR3224288} and Foucart \cite{MR3069373}.  We refer to Bertoin's book \cite{MR2253162} for a comprehensive account on exchangeable coalescents.

The main objective of this work is to study coalescent processes induced by branching processes. We briefly explain how branching concepts have been   developed from the sixties to the beginning of the twenty-first century. Continuous-state branching processes (CSBPs for short) are positive Markov processes representing the size of a continuous population. They have been defined by Ji\v{r}ina \cite{MR0101554} and Lamperti \cite{MR0208685} and are known to be scaling limits of Galton-Watson Markov chains, see Grimvall \cite{MR0362529} and Lamperti \cite{MR0217893}.  The most famous CSBP is certainly the Feller's branching diffusion 
\[\ddr X_t=\sigma \sqrt{X_t}\ddr B_t+\beta X_t\ddr t\]
which is the rescaled limit of binary branching processes, see Feller \cite{MR0046022} and Ji\v{r}ina \cite{MR0247676}. Feller's CSBP is the only CSBP with continuous paths, other ones have positive jumps which represent macroscopic reproduction events in the population.

At about the same time as the rise of exchangeable coalescents, considerable research was devoted to  the study of the genealogy of branching processes forward in time. Galton-Watson processes have a natural lexicographical tree's genealogy. This representation leads Aldous \cite{MR1207226} and Duquesne and Le Gall \cite{MR1954248} to study scaling limits of discrete trees and establish remarkable convergences towards Brownian continuum tree in the Feller diffusion case and L\'evy continuum tree in the case of a general CSBP. Another natural genealogy for a branching population is provided by Bertoin and Le Gall in their precursor article \cite{MR1771663} in terms of flows of subordinators. At any fixed times $s<t$, the population between time $s$ and $t$ is represented by a subordinator (a Lévy process with non-decreasing paths). Individuals are ordered in such a way that ancestors from time $s$ are the jumps locations of the subordinator and each ancestor from time $s$ has a family at time $t$ whose size is the size of the jump. 

Both representations with trees and subordinators are future-oriented and less attention has been paid to the description of coalescences in ancestral lineages of continuous-state branching processes. 
We briefly review some methods that have been developed recently in order to study the genealogy backwards in time of branching processes.

When reproduction laws are stable, branching and resampling population models can be  related through renormalisation by the total size  and  random time-change. We refer to Berestycki et al \cite{MR2349577}, Birkner et al. \cite{MR2120246}, Foucart and Hénard \cite{MR3035751} and Schweinsberg \cite{MR1983046}. The connection between exchangeable coalescents and CSBPs is particular to stable laws and the study of the genealogy of a general branching process requires a different method. 

One approach consists of conditioning the process to be non-extinct at a given time, sampling two or more individuals uniformly in the population and study the time of coalescence of their ancestral lineages. This program is at the heart of the works of Athreya \cite{MR3012088}, Duquesne and Labb\'e \cite{MR3164759}, Harris et al. \cite{2017arXiv170300299H}, Johnston \cite{2017arXiv170908500J}, Lambert \cite{MR2014270} and Le \cite{MR3189452}. 

Starting from a different point of view, Bi and Delmas \cite{MR3531711} and Chen and Delmas \cite{MR3025710} have considered stationary subcritical branching population obtained as processes conditioned on the non-extinction. The genealogy is then studied via a Poisson representation of the population. We refer also to Evans and Ralph \cite{MR2582640} for a study in the same spirit. 

A third approach is to represent the backwards genealogy through point processes. Aldous and Popovic \cite{MR2193998} and Popovic \cite{MR2100386} have shown how to encode the genealogy of a critical Feller diffusion with a Poisson point process on $\mathbb{R}_+\times \mathbb{R}_+$ called Coalescent Point Process. Atoms of the coalescent point process represents times of coalescences between two ``consecutive'' individuals in the boundary of the Brownian tree. Such a description was further developed by Lambert and Popovic \cite{MR3059232} for a L\'evy continuum tree. In this general setting, multiple coalescences are possible and the authors build a point process with multiplicities, which records both the coalescence times and the number of involved mergers in the families of the current population. Their method requires in particular to work with the height process introduced by Le Gall and Le Jan in \cite{MR1617047}. 

In the present article, we choose a different route and seek for a dynamical description of the genealogy. We first observe that flows of subordinators provide a continuous branching population whose size is infinite at any time and whose ancestors are arbitrarily old. We then study the inverse flow which tracks backward in time the ancestral lineage of an individual in the current population. This process corresponds to the Siegmund dual of the CSBP. 

In a second time, we construct random partitions by sampling arrival times of an independent Poisson process (with a fixed intensity) on the flow. We then describe how partitions coagulate when time's arrow points to the past and define new elementary non-exchangeable Markovian coalescents. We call these processes \textit{consecutive coalescents} as only consecutive blocks will be allowed to merge. Our method follows closely that of Bertoin and Le Gall for exchangeable coalescents (\cite{MR1990057}, \cite{LGB2}, \cite{LGB3}). Heuristically, the exchangeable bridges are replaced by subordinators and the uniform random variables by arrival times of a Poisson process. 

Consecutive coalescents are simple dual objects of continuous-time Galton-Watson processes and allow one to simplify the description of the genealogy of general CSBPs given by the Coalescent Point Process as introduced in \cite[Section 4]{MR3059232}. We shall also answer an open question in \cite[Remark 6]{MR3059232} by showing how to define the complete genealogy of individuals standing in the current generation when the so-called Grey's condition is satisfied. 

In the case of Neveu's CSBP (which does not fulfill Grey's condition), Bertoin and Le Gall in \cite{MR1771663}  have shown that the genealogy of the CSBP, started from a fixed size (without renormalization nor time-change) is given by a Bolthausen-Sznitman coalescent. We will see that for this CSBP, the consecutive coalescents have simple explicit laws. This will enable us to recover results of M\"ohle \cite{MR3333734} and M\"ohle and Kukla \cite{MR3758343} about the number of blocks in a Bolthausen-Sznitman coalescent. 

\bigskip

We wish to mention that Grosjean and Huillet  in \cite{MR3581248} have studied a recursive balls--in--boxes model which can be seen as a consecutive coalescent in discrete time. Moreover, Johnston and Lambert \cite{JohnstonLambert} have independently considered Poissonization techniques for studying the coalescent structure in branching processes. 

\bigskip

The paper is organized as follows. In Section \ref{sec:csbp}, we recall the definition of a continuous-state branching process and how Bochner's subordination can be used to provide a representation of the genealogical structure associated to CSBPs. In Section \ref{sec:inverseFlow}, we investigate the inverse flow by characterizing its semi-group and studying its long-term behavior. In Section \ref{fellerflow}, we provide a complete study of the inverse flow in the case of the Feller diffusion. We recover with an elementary approach the Coalescent Point Process of Popovic \cite{MR2100386}. In Section \ref{sec:markovianCoalescent}, we study the coalescences in the inverse flow of a general CSBP by defining the consecutive coalescents. We describe the genealogy of the whole population standing at the current generation under the Grey's condition (recalled in Section \ref{sec:csbp}). In Section \ref{martingaleproblem}, we investigate the infinitesimal dynamics of the inverse flow. The process of the ancestral lineage of a fixed individual is characterized by its generator. In Section \ref{sec:examples}, we provide some examples for which calculations can be made explicitly.

\tableofcontents
\paragraph{Notation.}
In the rest of the article, $\egaldistr$ denotes equality in law between random variables. Condition $\int_0 f(x) \ddr x < \infty$ means there exists $\epsilon>0$ such that $\int_0^\epsilon f(x) \ddr x < \infty$, and similarly $\int^\infty f(x) \ddr x < \infty$ means there exists $A>0$ such that $\int_A^\infty f(x) \ddr x < \infty$. For any $n,m \in \mathbb{N}$ such that $n\leq m$, the integer interval between $n$ and $m$ is denoted by $[|n,m|]$.

\section{Generalities on continuous-state branching processes}
\label{sec:csbp}
This section is divided in two parts. In the first one, we introduce the continuous-state branching processes as well as their fundamental properties. In the second one, we show how continuous-state branching processes can be constructed as flows of subordinators. Their main properties are also stated.

\label{sec:preliminariesCsbp}

\subsection{Continuous-state branching processes}

We give here an overview of continuous-state branching processes and their fundamental properties. Most statements in this section can be found for instance in \cite[Chapter 3]{MR2760602} or \cite[Chapter 12]{MR3155252}.

\begin{definition}
A continuous-state branching process is a Feller process $(X_t,t\geq 0)$, taking values in $[0,\infty]$, with $0$ and $\infty$ being absorbing states, whose semi-group satisfies the so-called branching property:
\begin{equation}
  \label{eqn:branchProp}
  \forall x,y\geq 0, \forall t\geq 0, \  X_t(x+y)\overset{\mathcal{L}}{=}X_t(x)+\tilde{X}_t(y),
\end{equation}
where $(X_t(x),t\geq 0)$ and $(\tilde{X}_t(y),t\geq 0)$ are two independent processes with the same law as $(X_t,t\geq 0)$, started respectively from $x$ and $y$. 
\end{definition}

The branching and the Markov properties ensure that for all $t \geq 0$ there exists a map $\lambda\in (0,\infty)\mapsto v_{t}(\lambda)$, which satisfies for all $\lambda>0$, $x \geq 0$ and $t,s\geq 0$
\begin{equation}
  \label{cumulant}
  \mathbb{E}[e^{-\lambda X_{t}(x)}]=\exp(-xv_{t}(\lambda))  \text{ and } v_{s+t}(\lambda)=v_{s}\circ v_t(\lambda).
\end{equation}
Silverstein \cite{MR0226734} shows that $t\mapsto v_{t}(\lambda)$ is the unique solution to the integral equation
\begin{equation}
  \label{odev}
  \forall t\in [0,\infty), \forall \lambda \in (0,\infty)/\{\rho\}, \quad \int_{v_{t}(\lambda)}^{\lambda}\frac{\ddr z}{\Psi(z)}=t
\end{equation}
where $\rho:=\inf\{z>0; \Psi(z)\geq 0\}$ is the largest positive root of $\Psi$, a L\'evy-Khintchine function of the form
\begin{equation}\label{LKpsi}
  \Psi(q)=\frac{\sigma^{2}}{2}q^{2}-\beta q+\int_{0}^{\infty}\left(e^{-qx}-1+qx\mathbbm{1}_{x\leq 1}\right)\pi(\ddr x),
\end{equation}
with $\sigma^2 \geq 0$, $\beta \in \R$ and $\pi$ a measure on $(0,\infty)$ satisfying $\int (1 \wedge x^2)  \pi(\ddr x) < \infty$. The function $\Psi$ is called branching mechanism and characterizes the law of the CSBP. We shall say later CSBP$(\Psi)$ for a CSBP with branching mechanism $\Psi$. 
The extended generator of the CSBP$(\Psi)$ is as follows
\begin{equation}
  \mathcal{L}f(z)=z\frac{\sigma^2}{2}f''(z)+\beta zf'(z)+z\int_{0}^{\infty}\left(f(z+h)-f(z)-hf'(z)\mathbbm{1}_{h\leq 1}\right)\pi(\ddr h)
\end{equation}
for any $f\in C_0^2$ \footnote{The space of twice differentiable continuous functions over $(0,\infty)$ vanishing at $\infty$.}.  The CSBP has infinite variations if
\begin{equation}
  \label{eqn:driftless}
  \int_0^1 x \pi(\ddr x) = \infty \quad \text{or} \quad \sigma^2 > 0.
\end{equation}
An important family of branching mechanisms are those of the form \[\Psi(q)=\frac{\sigma^2}{2}q^2-\beta q+c_\alpha q^{\alpha}\] with $\sigma^{2}\geq 0$, $c_\alpha\geq 0$ and $\alpha \in (0,2)$. The L\'evy measure $\pi$ associated to such a mechanism $\Psi$ is
\[
  \pi(\ddr h)=c'_\alpha h^{-1-\alpha}\ddr h, \quad \text{ with }  c'_\alpha=\frac{\alpha(\alpha-1)}{\Gamma(2-\alpha)}c_\alpha.
\]

The CSBP($\Psi$) is said to be supercritical, critical or subcritical if respectively $\Psi'(0+)<0$, $\Psi'(0+)=0$ or $\Psi'(0+)>0$. In the subcritical and critical cases, the largest root $\rho$ is $0$. In the supercritical case $\rho\in (0,\infty]$. The following theorem due to Grey \cite{MR0408016} summarizes the possible behaviors at the boundaries of a CSBP$(\Psi)$. 

\begin{theostar}[Grey, \cite{MR0408016}]\label{Grey} 
Consider $(X_{t}(x), t\geq 0)$ a CSBP$(\Psi)$ started from $x$.
\begin{enumerate}
  \item For any $x\geq 0$, $$\mathbb{P}(\underset{t\rightarrow \infty}\lim X_{t}(x)=0)=1-\mathbb{P}(\underset{t\rightarrow \infty}\lim X_{t}(x)=\infty)=e^{-x\rho}.$$ 
  \item For any $t>0$, the limit $v_t(\infty):=\underset{\lambda \rightarrow \infty}{\lim}v_t(\lambda)$ exists in $(0,\infty)$ if and only if $\Psi(u)\geq 0$ for some $u\geq 0$ and
  \begin{equation}
    \label{Extinction}
    \int^{\infty}\frac{\ddr u}{\Psi(u)}<\infty \quad \text{(condition for extinction).}
  \end{equation}
If \eqref{Extinction} holds, then for any $t\geq 0$, $\mathbb{P}(X_t(x)=0)=e^{-xv_t(\infty)}>0$.
  \item Under condition \eqref{Extinction}, the following events coincide almost-surely
\[\left\{\underset{t\rightarrow \infty}{\lim} X_t(x)=0\right\}=\{\exists t\geq 0: X_t(x)=0\}.\]
  \item For any $t>0$, the limit $v_t(0):=\underset{\lambda \rightarrow 0}{\lim}v_t(\lambda)$ exists in $(0,\infty)$ if and only if $\Psi(u)<0$ for some $u\geq 0$ and 
  \begin{equation}
    \label{Explosion}
    \int_{0}\frac{\ddr u}{|\Psi(u)|}<\infty \quad \text{(condition for explosion)}.
  \end{equation}
If (\ref{Explosion}) holds, then for any $t\geq 0$, $\mathbb{P}(X_t(x)=\infty)=1-e^{-xv_t(0)}$.
  \item Under condition \eqref{Explosion}, the following events coincide
\[\left\{\underset{t\rightarrow \infty}{\lim} X_t(x)=\infty\right\}=\{\exists t\geq 0: X_t(x)=\infty\}.\]
\end{enumerate}
\end{theostar}

The event $\{X_t(x)=0 \text{ for some } t\geq 0\}$ is called extinction, and $\{X_t(x)=\infty \text{ for some } t\geq 0\}$ is called explosion. We refer to the integral conditions \eqref{Extinction} and \eqref{Explosion} as Grey's condition for extinction and explosion respectively. Lambert \cite{MR2299923} and Li \cite{MR1727226} have studied the quasi-stationary distribution of subcritical CSBPs conditioned on the non-extinction.

\begin{theostar}[Lambert \cite{MR2299923}, Li \cite{MR1727226}] In the subcritical case, under Grey's condition for extinction $\int^{\infty}\frac{\ddr u}{\Psi(u)}<\infty$, there exists a probability measure $\nu$ over $(0,\infty]$ such that for any Borelian set $A\subset [0,\infty]$
\[
  \nu(A):=\underset{t\rightarrow \infty}{\lim} \mathbb{P}(X_t(x)\in A|X_t(x)>0).
\]
The Laplace transform of $\nu$ is given by 
\begin{equation}
  \label{qsdl}
  \int_{0}^{\infty}e^{-uz}\nu(\ddr z)=1-e^{-\Psi'(0+)\int_{u}^{\infty}\frac{\ddr x}{\Psi(x)}} \text{ for any } u\geq 0.
\end{equation} 
\end{theostar}

\subsection{Flows of subordinators }

Observe that on the one hand, by the the branching property of CSBP, the random variable $X_t(x)$ is a positive infinitely divisible random variable, parametrized by $x$. Therefore, for all $t \geq 0$, the process $x \mapsto X_t(x)$ is a positive Lévy process, hence a subordinator. In particular, the map $\lambda\mapsto v_t(\lambda)$ is the Laplace exponent of this (possibly killed) subordinator, and can be written as
\begin{equation}
  \label{eqn:levykinDecompositionv_t}
  v_t(\lambda)=\kappa_t+d_t\lambda+\int_0^{\infty}(1-e^{-\lambda x})\ell_t(\ddr x)
\end{equation}
with $\kappa_t\geq 0$, $d_t\geq 0$ and $\ell_t$ a L\'evy measure on $\mathbb{R}_+$ such that $\int_{0}^{\infty}(1\wedge x)\ell_t(\ddr x)<\infty$. 

\begin{remark}
Note that the quantities $v_t(\infty)$ and $v_t(0)$ defined in Theorem~\ref{Grey} can be rewritten, with the formula in \eqref{eqn:levykinDecompositionv_t}
\[
  v_t(\infty) = \ell_t([0,\infty)) \quad \text{and} \quad v_t(0) = \kappa_t.
\]
Therefore \eqref{Extinction} is equivalent to the finiteness of the measure $\ell_t$ for all $t > 0$, and \eqref{Explosion} to the positivity of $\kappa_t$ for all $t > 0$.
\end{remark}

On the other hand, the semigroup property entails that for any $s,t\geq 0$,
\begin{equation}
  \label{compositionv}
  v_{t+s}=v_t\circ v_s.
\end{equation}
Bochner's subordination implies that if $Y^{(t)}$ is a subordinator with Laplace exponent $v_t$ and $Y^{(s)}$ is a subordinator with Laplace exponent $v_s$, then $Y^{(t)} \circ Y^{(s)}$ is a subordinator with Laplace exponent $v_t\circ v_s=v_{t+s}$. Therefore, writing $\tilde{X}$ an independent copy of the CSBP $X$, we have
\[
  \forall x \geq 0, \, X_{t+s}(x) \egaldistr \tilde{X}_t(X_s(x)).
\]
This last observation leads Bertoin and Le Gall \cite{MR1771663} to consider representing a CSBP as a flow of subordinators, which we now define.

\begin{definition}
\label{flowdef2}
A flow of subordinators is a family $(X_{s,t}(x), s \leq t, x \geq 0)$ satisfying the following properties:
\begin{enumerate}
  \item For every $s \leq t$, $x \mapsto X_{s,t}(x)$ is a càdlàg subordinator, with same law as $x \mapsto X_{0,t-s}(x)$.
  \item For every $t \in \R$, $(X_{r,s}, r \leq s \leq t)$ and $(X_{r,s}, t\leq r \leq s)$ are independent.
  \item For every $r \leq s \leq t$, $X_{r,t} = X_{s,t} \circ X_{r,s}$.
  \item For every $s \in \R$ and $x \geq 0$, we have $X_{s,s}(x) = x = \lim_{t \to s} X_{s,t}(x)$ in probability.
\end{enumerate}
\end{definition}

\begin{remark}
The convergence in (\textit{iv})  also holds uniformly on compact sets by second Dini's theorem.
\end{remark}

It was proved by Bertoin and Le Gall \cite{MR1771663} that any CSBP can be constructed as a flow of subordinators. For the sake of completeness, we prove here that CSBP and flow of subordinators are in one-to-one map.

\begin{lemma}
Let $(X_{s,t}(x), s \leq t, x \geq 0)$ be a flow of subordinators as in Definition~\ref{flowdef2}, there exists a branching mechanism $\Psi$ such that for all $s \in \R$ and $x \geq 0$, $(X_{s,s+t}(x), t \geq 0)$ is a CSBP($\Psi$) starting from $x$. Reciprocally, for each branching mechanism $\Psi$, there exists a flow of subordinators such that for all $s \in \R$ and $x \geq 0$, $(X_{s,s+t}(x), t \geq 0)$ is a CSBP($\Psi$) starting from $x$.
\end{lemma}

\begin{proof}
Let $(X_{s,t}(x), s \leq t, x \geq 0)$ be a flow of subordinators. By Definition~\ref{flowdef2}(ii) and (iii), we have that $t \mapsto X_{s,s+t}(x)$ is a Markov process for all $x \geq 0$ and $s \in \R$. Moreover, Definition~\ref{flowdef2}(iv) implies this Markov process to be continuous in probability, therefore Feller, and by Definition~\ref{flowdef2}(i), we conclude that this Markov process is homogeneous in time, and satisfies the branching property, as
\[
  X_{s,s+t}(x+y) = X_{s,s+t}(x) + \left(  X_{s,s+t}(x+y) -  X_{s,s+t}(x) \right),
\]
and $ X_{s,s+t}(x+y) -  X_{s,s+t}(x) $ is independent of $X_{s,s+t}(x)$ and has same law as $X_{s,s+t}(y)$. Reciprocally, by \cite[Proposition 1]{MR1771663}, given a branching mechanism $\Psi$, there exists a process $(S^{(s,t)}(a), s \leq t, a \geq 0)$ such that almost surely
\begin{enumerate}
  \item for all $s \leq t$, $a \mapsto S^{(s,t)}(a)$ is a càdlàg subordinator with Lévy-Khintchine exponent $\lambda \mapsto v_{t-s}(\lambda)$, defined in~\eqref{odev},
  \item for all $t \in \R$, $(S^{(r,s)}, r \leq s \leq t)$ and $(S^{(r,s)}, t \leq r \leq s)$ are independent,
  \item for all $r \leq s \leq t$, $S^{(s,t)} \circ S^{(r,s)} = S^{(r,t)}$,
  \item the finite dimensional distributions of $t \mapsto S^{(s,s+t)}(a)$ are the ones of a CSBP($\Psi$).
\end{enumerate}
One readily observe that points (i)--(iii) imply Definition~\ref{flowdef2}(i)--(iii). Moreover, by the fourth point, $S^{(s,s+t)}(a)$ has the law of a CSBP($\Psi$) $X_t$ starting from $X_0=a$. As $X$ is a Feller process, we have $\lim_{t \to 0} X_t=a$ in probability, thus (iv) yields $\lim_{t \to 0} S^{(s,s+t)}(a) =a$ in probability, completing the proof.
\end{proof}

A noteworthy consequence of the above lemma is that if $(X_{s,t}(x), s \leq t, x \geq 0)$ is a flow of subordinators associated to the branching mechanism $\Psi$, we have that for all $s \leq t$ and $x \geq 0$,
\begin{equation}
  \label{eqn:lapTransform}
 \forall \lambda \in (0,\infty), \  \E\left( \exp\left( - \lambda X_{s,t}(x) \right) \right) = \exp(-x v_{t-s}(\lambda)),
\end{equation}
where $v_{t-s}(\lambda)$ is the function defined in \eqref{odev}.
One can think of this flow of subordinators as a way to couple on the same probability space every Markov property and every branching property \eqref{eqn:branchProp}, for all values of $t,x,y$ simultaneously in one process.

The flow of subordinators provides a genuine continuous-space branching population model. More precisely, the interval $[0,X_{s,t}(x)]$ can be interpreted as the set of descendants at time $t$ of the population that was represented at time $s$ as the interval $[0,x]$.
With this interpretation, the genealogy forward in time of the population is defined as follows. If $X_{s,t}(y-)<X_{s,t}(y)$, we say that for all $z \in (X_{s,t}(y-),X_{s,t}(y)]$, the individual $z$ at time $t$ is a descendant of the individual $y$ living at time $s$. If $X_{s,t}(y-) = X_{s,t}(y)$ (i.e. $X_{s,t}$ is continuous at $y$), we then say that individual $z = X_{s,t}(y)$ at time $t$ is the descendant of the individual $y$ living at time $s$ if and only if $y = \inf\{ x > 0: X_{s,t}(x) = z\}$. One can observe that the cocycle property ensures that this construction indeed defines a genealogy. If $z$ at time $t$ is a descendant of $y$ at time $s$, which is a descendant of $x$ at time $r$, we have
\[
  X_{s,t}(y-) < z \leq X_{s,t}(y) \quad \text{and} \quad X_{r,s}(x-) < y \leq X_{r,s}(x).
\]
By the cocycle property ($X_{r,t}=X_{s,t}\circ X_{r,s}$) and as $X_{s,t}$ is non-decreasing then
\[
  X_{r,t}(x-)=X_{s,t}(X_{r,s}(x-))\leq X_{s,t}(y-)<z \quad \text{and} \quad X_{r,t}(x)=X_{s,t}(X_{r,s}(x))\geq X_{s,t}(y)\geq z,
\]
thus $z$ at time $t$ is a descendant of $x$ at time $r$. Similar computations can be written if $X_{s,t}$ is continuous at point $y$ and/or $X_{r,s}$ is continuous at point $x$.

Recall the condition \eqref{eqn:driftless} for the sample paths of the CSBP$(\Psi)$ to have infinite variations. This condition ensures the subordinator $X_{s,t}$ to be driftless, i.e. $d_{r} = 0$ for all $r \geq 0$ in \eqref{eqn:levykinDecompositionv_t}. As a result, under \eqref{eqn:driftless}, the range $X_{s,t}([0,\infty))$ of the subordinator has zero Lebesgue measure, ensuring that almost every individual $x$ at time $t$ belong to one of the infinite families of ancestors at time $s$. This, assumption \eqref{eqn:driftless} often simplifies the interpretation of results obtained in this article.
Under this assumption, we denote by $J_{s,t} = \{ x \geq 0 : X_{s,t}(x) \neq X_{s,t}(x-)\}$ the set of jumps of $X_{s,t}$. By definition of the genealogy, almost surely the population at time $t$, indexed by $\R_+$, can be partitioned according to their ancestor at time $s$ by $\left\{ (X_{s,t}(y-),X_{s,t}(y)], y \in J_{s,t} \right\}$.

Recall that according to Theorem~\ref{Grey}-(ii), Grey's condition $\int^{\infty}\frac{\ddr u}{\Psi(u)}<\infty$ entails that for any $t>0$, $\ell_t([0,\infty])<\infty$. Under this condition, the subordinators $X_{s,t}$ are therefore compound Poisson processes. In particular, the set $J_{s,t}$ is the set of arrival times of a Poisson process with intensity $v_t(\infty)$. Note that the partition $\left\{ (X_{s,t}(y-),X_{s,t}(y)], y \in J_{s,t} \right\}$ consists of a family of consecutive intervals. This justifies the introduction of consecutive coalescents on $\N$ in Section~\ref{sec:markovianCoalescent}.

\section{The inverse flow}
\label{sec:inverseFlow}

We start this section by a preliminary observation on the genealogy backward in time of a CSBP. Consider the Poisson point process on $\mathbb{R}_+\times (0,\infty)$
\begin{equation}
  \label{eqn:ancestorRecording}
  \mathcal{E}_t = \sum_{\substack{x \geq 0}} \delta_{\left(a_tx,\Delta X_{-t,0}(x))\right)}
\end{equation}
with some renormalisation constant $a_t>0$ for all $t>0$. Recall $\rho$ the largest positive root of $\Psi$ and $\nu$ the quasi-stationary distribution \eqref{qsd} of a subcritical CSBP conditioned on the non-extinction.

\begin{proposition}\label{warmup}
Assume $\int^{\infty}\frac{\ddr u}{\Psi(u)}<\infty$ and set $a_t=1$ if $\Psi'(0+)<0$, $a_t=v_t(\infty)$ if $\Psi'(0+)\geq 0$. Then
\begin{equation}
  \label{eqn:originOfSpecies}
  \lim_{t \to \infty} \mathcal{E}_t = \mathcal{E}_\infty \quad \text{ in law, for the topology of weak convergence}
\end{equation}
where $\mathcal{E}_\infty$ is a Poisson point process with intensity respectively 
$\rho\ddr x \otimes \delta_\infty(\ddr z)$ when $\Psi'(0+)<0$, $\ddr x \otimes \delta_\infty(\ddr z)$ when $\Psi'(0+)=0$, and $\ddr x\otimes \nu(\ddr z)$ when $\Psi'(0+)>0$.
\end{proposition}
\begin{remark}
In the supercritical case, flows of CSBPs can be renormalized to converge almost-surely. We refer to Duquesne and Labbé \cite{MR3164759}, Grey \cite{MR0408016},  and Foucart and Ma \cite{2016arXiv161106178F}.  Since for any time $t$, $X_{-t,0}$ and $X_{0,t}$ have the same law, we could therefore renormalize in law the size of the descendants at time $0$ of $x$ from time $-t$. Typically, $\Delta X_{-t,0}(x)$ is of order exponential in the finite mean case ($|\Psi'(0+)|<\infty$), and double exponential in the infinite mean case ($|\Psi'(0+)|=\infty$). 
\end{remark}
\begin{proof}
Under Grey's condition, $v_t(\infty)=\ell_t(]0,\infty])<\infty$, and $x \mapsto X_{-t,0}(x)$ is a compound Poisson process with no drift. Therefore, the point process $\mathcal{E}_t$ is a Poisson point process with intensity $\frac{\ell_t(\ddr x)}{a_t}$. Observe additionally that for any $q\geq 0$, 
\begin{equation}
  \label{eqn:laplTransform}
  \int_{0}^{\infty}e^{-qx}\frac{\ell_t(\ddr x)}{a_t}=1-\frac{v_t(q)}{a_t}.
\end{equation}

In the supercritical case ($\Psi'(0+)<0$), we have $\lim_{t \to \infty} v_t(q) = \rho$ for all $q > 0$ (while $v_t(0)=0$), and $a_t=1$ for all $t>0$. Therefore \eqref{eqn:laplTransform} shows that $\ell_t(\ddr x)$ converges weakly toward $\rho \delta_\infty(\ddr x)+(1-\rho) \delta_0(\ddr x)$. As a result, we conclude that $\mathcal{E}_t$ converges in law toward a Poisson point process on $(0,\infty] \times (0,\infty]$ with intensity $\rho \ddr x \otimes \delta_\infty(\dd z)$.

In the subcritical and critical cases, we have $\lim_{t \to \infty} v_t(\infty) = 0$, and we set $a_t = v_t(\infty)$. By \eqref{odev} and \eqref{cumulant}, we have $\frac{\ddr }{\ddr u}v_t(u)=\frac{\Psi(v_t(u))}{\Psi(u)}$. Therefore
\[
  \frac{v_t(q)}{v_t(\infty)}=\exp\left(-\int_{q}^{\infty}\frac{\ddr}{\ddr u}\log(v_t(u))\ddr u\right)=\exp\left(-\int_{q}^{\infty}\frac{\Psi(v_t(u))}{v_t(u)}\frac{\ddr u}{\Psi(u)}\right).
\]
One has $\lim_{t \to \infty} \frac{\Psi(v_t(u))}{v_t(u)}= \Psi'(0+)$, thus we obtain that $\lim_{t \to \infty} \frac{v_t(q)}{v_t(\infty)} = e^{-\Psi'(0+)\int_{q}^{\infty}\frac{\ddr u}{\Psi(u)}}$ by monotone convergence. This limit is $1$ in the critical case ($\Psi'(0+)=0$), which by \eqref{eqn:laplTransform} and thus $\frac{\ell_t(\ddr x)}{a_t}$ converges weakly towards $\delta_\infty$. In the subcritical case ($\Psi'(0+)>0$), we see that $\frac{\ell_t(\ddr x)}{a_t}$ converges weakly towards the probability measure $\nu$ with Laplace transform \eqref{qsdl}. We conclude the convergence of $\mathcal{E}_t$ to the stated limits.
\end{proof}

Let us describe in details the meaning of the above convergence, for supercritical, critical and subcritical CSBPs. Observe that $\mathcal{E}_t$ encodes information on the individuals at time $-t$ having a large family of descendants at time $0$. Thus, \eqref{eqn:originOfSpecies} gives information on the origin of the earliest ancestors of the population at time $0$. Depending on the sign of $\Psi'(0+)$, we have three different behaviours:
\begin{enumerate}
\item If $\Psi'(0+)<0$, a unique ancestor from time $-\infty$, located at an exponential random variable with parameter $\rho$, which generates all individuals at time $0$. This individual is the ancestor of the process.
\item If $\Psi'(0+)=0$, then $a_t:=v_t(\infty)\underset{t\rightarrow \infty}{\longrightarrow} 0$ and the whole population at time $0$ has a common ancestor, but the backward lineage of this ancestor converges in law as $t \to \infty$ towards $\infty$.
  \item If $\Psi'(0+)>0$, then the population at time $0$ is split into distinct families, each of which coming down from a different ancestor at time $-\infty$.
\end{enumerate}
In the (sub)critical case, individuals from generation $-t$ with descendance at time $0$ are located at distance $O(1/v_t(\infty))$ from $0$. Proposition~\ref{warmup} motivates a more complete study of the ancestral lineages of individuals alive in the population at time $0$.  Our main aim is to provide an almost-sure description of how the $(\mathcal{E}_t,t\geq 0)$ evolves and to get precise information on the sizes of the families.

We now introduce the inverse flow of the flow of subordinators $(X_{s,t}, s \leq t)$ and study some of its properties. We first define, for $s \leq t$ and $y \geq 0$
\[
  X^{-1}_{s,t}(y):=\inf\{x: X_{s,t}(x)>y\}.
\]
The process $X^{-1}_{s,t}$ is the right-continuous inverse of the càdlàg process $X_{s,t}$. Note that the individual $X^{-1}_{s,t}(y)$ is the ancestor alive at time $s$ of the individual $y$ considered at time $t\geq s$. It is therefore a natural process to introduce in order to study the genealogy of a CSBP backwards in time.
We call inverse flow the process $(\hat{X}_{s,t}(y), s \leq t, y  \geq 0)$ defined for all $s \leq t, y \geq 0$
as follows
\begin{equation}
  \label{eqn:defFlotInverse}
 \hat{X}_{s,t}(y) = X^{-1}_{-t,-s}(y).
\end{equation}
We first list some straightforward properties of inverse flows.

\begin{lemma}
\label{flowA}
The following properties hold:
\begin{enumerate}
  \item Almost surely, for every $s\leq t$ and $x,y>0$, we have $\{X_{s,t}(x)>y\}=\{\hat{X}_{-t,-s}(y)< x\}$.
  \item For every $t \geq 0$, $(\hat{X}_{r,s}, r \leq s \leq t)$ and $(\hat{X}_{r,s}, t \leq r \leq s)$ are independent.
  \item Almost surely, for every $s\leq t\leq u$, $\hat{X}_{s,u}=\hat{X}_{t,u}\circ \hat{X}_{s,t}$.
  \item For all $x \geq 0$, $\hat{X}_{0,0}(x) = x = \lim_{t \to 0} \hat{X}_{0,t}(x)$ in probability.
\end{enumerate}
\end{lemma}

\begin{remark}
The convergence in (\textit{iv})  also holds uniformly on compact sets.
\end{remark}

\begin{proof}
These results are an immediate consequence of Lemma~\ref{lem:factsInverse}, which describes well-known properties of right-continuous inverses, and the definition of flow of subordinators. More precisely, the first point is a consequence of Lemma~\ref{lem:factsInverse}(ii), the third one of Lemma~\ref{lem:factsInverse}(iii) and the fourth one follows from  Lemma~\ref{lem:factsInverse}(iv) and Definition~\ref{flowdef2}(iv).

Finally, the second point follows simply from the fact that for all $a \leq b \leq t$, $\hat{X}_{a,b}$ is measurable with respect to $(X_{r,s}, -t \leq r \leq s )$. Hence, by Definition~\ref{flowdef2}(ii), we conclude that (ii) holds.
\end{proof}

We shall denote $(\hat{X}_{t}(y), y\geq 0, t \geq 0)$ the flow of inverse subordinators $(\hat{X}_{0,t}(y), y\geq 0, t \geq 0)$. 
As noted above, it tracks backwards in time the ancestral lineages of the population at time $0$. Since individuals are ordered, $\hat{X}_{t}(y)$ can also be interpreted as the random size of the population at time $-t$ whose descendance at time $0$ has size $y$. Observe that by Lemma~\ref{flowA}(i) and Definition~\ref{flowdef2}(i), we have
\begin{equation}
  \label{duality}
  \forall s \leq t,\ \forall x,y \geq 0,\ \mathbb{P}(X_{s,t}(x)>y)=\mathbb{P}(\hat{X}_{s,t}(y)< x).
\end{equation}
The relation \eqref{duality} is known as Siegmund duality. We refer the reader for instance to Siegmund \cite{MR0431386} and Clifford and Sudbury \cite{MR781422}. 

\begin{theorem}
\label{flowinverse}
Fix $y>0$. The process $(\hat{X}_{t}(y),t\geq 0)$ is a Markov process in $(0,\infty)$. Its semigroup $(Q_t,t\geq 0)$ satisfies for any bounded measurable function $f$  and any $t\geq 0$
\begin{equation}
  \label{eqn:flowinverse}
  \mathbb{E}[Q_tf(\mathbbm{e}_q)]=\mathbb{E}[f(\mathbbm{e}_{v_t(q)})] \ \text{ for all } q>0
\end{equation}
where $\mathbbm{e}_q$  and $\mathbbm{e}_{v_t(q)}$ are exponential random variable with parameter $q$ and $v_t(q)$ and $\mathbbm{e}_q$ is independent of $(\hat{X}_t(y),t\geq 0,y\geq 0)$.
\end{theorem}

Observe that \eqref{eqn:flowinverse} characterizes the semigroup $Q_t$, by identification of the Laplace transforms, as it can be rewritten as: for all $q \geq 0$,
\[
  \int_0^\infty q e^{- qy} Q_t f(y) \ddr y = \int v_t(q) e^{-v_t(q)y} f(y) \ddr y,
\]
therefore $Q_tf$ is the inverse Laplace transform of $q \mapsto \frac{v_t(q)}{q}
\int e^{-v_t(q)y} f(y) \ddr y$.

\begin{proof}
We observe that the cocycle property and the independence, obtained in points (ii) and (iii) of Proposition~\ref{flowA} readily entail that $t \mapsto \hat{X}_t(y)$ has the Markov property. Moreover, if $\hat{X}_{0,t}(y) = 0$ then $X_{-t,0}(0) = y>0$, which is impossible, as $X_{-t,0}(0)$ is the value at time $t$ of a CSBP starting from mass $0$, and $0$ is an absorbing point for a CSBP. Similarly, $\hat{X}_{0,t}(y) = \infty$ yields that $\lim_{z \to \infty} X_{-t,0}(z) \leq y$, which is impossible as well as $X_{-t,0}(z)$ is a non-null subordinator.

Finally, we now turn to the computation of the semigroup of $\hat{X}(y)$, which is obtained through the Siegmund duality. Let $\mathbbm{e}_q$ be an independent exponential random variable with parameter $q$, we have
\[\mathbb{P}(\hat{X}_{t}(\mathbbm{e}_q)>x)=\mathbb{P}(X_{-t,0}(x)<\mathbbm{e}_q)=\mathbb{E}[e^{-qX_{-t,0}(x)}]=e^{-xv_{t}(q)},\]
which implies that \eqref{eqn:flowinverse} holds.
\end{proof}
The above theorem shows that the semigroup of $(\hat{X}_t)$ can be expressed in simple terms when applied to exponential distributions. This will motivate later on the study of the action of the flow $\hat{X}$ on Poisson point processes. 

We now observe that the Markov process $t \mapsto \hat{X}_{0,t}(y)$ can be straightforwardly extended as a Markov process on $[0,\infty]$.

\begin{proposition}[Boundaries and Feller property]
\label{boundaries}
Let $y > 0$ fixed, we denote by $(\hat{X}_t, t \geq 0)$ the Markov process $(\hat{X}_{t}(y), t \geq 0)$. 
\begin{enumerate}
  \item The boundary $0$ is an entrance boundary of $(\hat{X}_t,t\geq 0)$ if and only if $\int^{\infty}\frac{\ddr u}{\Psi(u)}<\infty$. In that case, $(Q_t,t\geq 0)$ is extended to $[0,\infty)$ by
\[Q_tf(0)=\int_{0}^{\infty}f(u)v_t(\infty)e^{-uv_t(\infty)}\ddr u.\]
Otherwise, we set $Q_tf(0) = f(0)$.
  \item  The boundary $\infty$ is an entrance boundary of $(\hat{X}_t,t\geq 0)$ if and only if $\int_{0}\frac{\ddr u}{|\Psi(u)|}<\infty$. In that case, $(Q_t,t\geq 0)$ is defined over $]0,\infty]$ with
\[Q_tf(\infty)=\int_{0}^{\infty}f(u)v_t(0)e^{-uv_t(0)}\ddr u.\] 
Otherwise, we set $Q_tf(\infty) = f(\infty)$.
  \item The semigroup $(Q_t,t\geq 0)$ defined over $[0,\infty]$ is Feller.
\end{enumerate}
\end{proposition}

\begin{remark}\label{inverseflowfrom0}
The Markov processes $(\hat{X}_{0,t}(0))$ and $(\hat{X}_{0,t}(\infty))$ have the following interpretations, in terms of the CSBP
\begin{enumerate} 
  \item The process $(\hat{X}_{0,t}(0),t\geq 0)$, starting from $0$ at time $0$, represents the smallest individual at generation $-t$ to have descendants at time $0$.  If $\int^\infty \frac{\ddr u}{|\Psi(u)|} < \infty$, there is extinction in finite time for the CSBP $X$ (i.e. with positive probability, $X_{-t,0}(x) = 0$). In that case $\hat{X}_{0,t}(0)$ is a non-trivial Markov process. If $\int^{\infty}\frac{\ddr u}{|\Psi(u)|}=\infty$, there is no extinction in finite time for the CSBP, thus all individuals at time $t$ have descendants at time $0$, $(\hat{X}_t(0),t\geq 0)\equiv 0$.
  \item The process $(\hat{X}_t(\infty),t\geq 0)$, starting from $\infty$, represents the smallest individual at generation $t$ with an infinite progeny at time $0$. If $\int_0 \frac{\ddr u}{|\Psi(u)|} < \infty$, there is explosion in finite time for the CSBP $X$ (i.e. with positive probability, $X_{-t,0}(x) = \infty$). In that case, $\hat{X}_{0,t}(\infty)$ is a non-trivial Markov process. If $\int_0 \frac{\ddr u}{|\Psi(u)|} = \infty$, there is no explosion in finite time and all individuals at time $t$ have finitely many descendants at time $0$. Thus $(\hat{X}_t(\infty),t\geq 0)\equiv \infty$ and $Q_tf(\infty):=f(\infty)$.
\end{enumerate}
\end{remark}

\begin{proof}
For any fixed time $t$, $(\hat{X}_t(x),x\in (0,\infty))$ is non-decreasing in $x$. Thus $\underset{x\rightarrow \infty}{\lim} \hat{X}_t(x)=\hat{X}_t(\infty)$ and
$\underset{x\rightarrow 0}{\lim} \hat{X}_t(x)=\hat{X}_t(0)$ exist almost-surely in $[0,\infty]$. Recall the duality relation \eqref{duality}
\[\mathbb{P}(\hat{X}_t(y)<x)=\mathbb{P}(y<X_t(x)).\]
The first point for the boundary $0$
is obtained as follows. By  the duality relation, we see that  
\[\mathbb{P}(\hat{X}_t(y)\geq x)=\mathbb{P}(y\geq X_t(x)).\]
By letting $y$ to $0$, we have 
\[\mathbb{P}(\hat{X}_t(0)\geq x)=\mathbb{P}(X_t(x)=0)=e^{-xv_t(\infty)}.\]
According to Theorem~\ref{Grey}-(ii), $\int^{\infty}\frac{\ddr u}{\Psi(u)}<\infty$ is a necessary and sufficient condition  for $v_t(\infty) < \infty$. It remains to justify the formula for $Q_tf(0)$. By using Theorem~\ref{flowinverse} and 
the facts that in probability, $\lim_{q \to \infty} \mathbbm{e}_q = 0$ and $\lim_{q \to \infty} v_t(q) = v_t(\infty)\in (0,\infty]$, we have for any continuous bounded function $f$ on $[0,\infty)$, 
\[ Q_tf(0)=\lim_{q \to \infty} \E(Q_tf(\mathbbm{e}_{q}))=\lim_{q \to \infty} \E(f(\mathbbm{e}_{v_t(q)}))=\E\left(f(\mathbbm{e}_{v_t(\infty)}) \right)\]
by dominated convergence. We deduce the formula for $Q_tf(0)$. We now prove that the semigroup property holds at $0$. By definition of $Q_tf(0)$, we have that
\[
  Q_{t+s}f(0) = \E\left( f(\mathbbm{e}_{v_{t+s}(\infty)}) \right) \quad \text{and} \quad Q_t(Q_sf)(0) = \E\left( Q_sf(\mathbbm{e}_{v_t(\infty)}) \right) = \E\left( f (\mathbbm{e}_{v_{s} \circ v_t(\infty)}) \right).
\]
Therefore, as $v_{t+s} = v_t \circ v_s$, we complete the proof of (\textit{i}).

The proof of (\textit{ii}) follows very similar lines to the proof of (\textit{i}), and is based on the fact that $\int_0 \frac{\ddr u}{|\Psi(u)|}<\infty$ is a necessary and sufficient condition for $v_t(0) >0$. The expression of $Q_tf(\infty)$ is found using that $\lim_{q \to 0} \mathbbm{e}_q = \infty$ in probability.
Finally, to prove that the semigroup $Q_t$ extended to $[0,\infty]$ is Feller, we observe that the random map $y\mapsto \hat{X}_{t}(y)$ jumps only on constant stretch of $X_{-t,0}$ (being its right-continuous inverse). There is no fixed value in $(0,\infty)$ at which $X_{-t,0}$ is constant and therefore $y\in (0,\infty)\mapsto \hat{X}_{t}(y)$ has no fixed discontinuities. This entails that for any continuous function $f$ over $[0,\infty]$, $Q_tf$ is continuous at any point $y\in (0,\infty)$.  By definition $Q_tf(x)\underset{x\rightarrow \infty}{\longrightarrow} Q_tf(\infty)$ and  $Q_tf(x)\underset{x\rightarrow 0}{\longrightarrow} Q_tf(0)$. The semigroup maps $C([0,\infty])$ in $C([0,\infty])$ and one only needs to show the pointwise continuity at $0$ of $Q_tf$, which follows from Proposition~\ref{flowA}(iv).
\end{proof}

We study now the long term behaviour of $(\hat{X}_t,t\geq 0)$ in the critical and subcritical case. By transience, we mean that $\hat{X}_t(x)\underset{t\rightarrow\infty}{\longrightarrow} \infty$ a.s. for any $x\in (0,\infty)$.
\begin{proposition}\label{transiencerecurrence}
Let $\Psi$ be a branching mechanism. We observe that
\begin{enumerate}
  \item if $\Psi$ is supercritical, then $\hat{X}$ is positive recurrent with stationary law $\mathbbm{e}_\rho$;
  \item if $\Psi$ is subcritical, then $\hat{X}$ is transient;
  \item if $\Psi$ is critical, then $\hat{X}$ is transient if and only if $\int_0 \frac{u}{\Psi(u)}\ddr u < \infty$, otherwise it is null recurrent.
\end{enumerate}
\end{proposition}
\begin{remark} Intuitively, in the subcritical case, for any fixed $a>0$, individuals below level $a$ living at arbitrarily large time in the past will have no progeny at time $0$. Therefore the ancestral lineage of an individual $x$ living at time $0$, goes above any fixed level $a$ as time goes to $\infty$. This explains the transience. In the critical case, large oscillations can occur when $\int_{0}\frac{x}{\Psi(x)}\ddr x=\infty$. This latter condition is known see Duhalde et al. \cite{MR3264444} to entails that first entrance times of the CSBP have infinite mean, in such case the process $(\hat{X}_t,t\geq 0)$ is null recurrent. Note that if $\Psi(q)=cq^{\alpha}$ with $1\leq \alpha \leq 2$ then $(\hat{X}_t,t\geq 0)$ is null recurrent if $\alpha=2$ and transient if $\alpha<2$.
\end{remark}

\begin{proof}
We first prove (\textit{i}). Let $y\in (0,\infty)$. By duality \eqref{duality} and Theorem~\ref{Grey}-(i) \[\mathbb{P}(\hat{X}_t(y)<x)=\mathbb{P}(X_t(x)>y)\underset{t\rightarrow \infty}{\longrightarrow} \mathbb{P}(\text{non-extinction})=1-e^{-\rho x}.\]
Assume now $\Psi$ subcritical or critical. For any Borelian set $B$ and any $p>0$, set $$U_p(y,B):=\int_{0}^{\infty}e^{-pt}\mathbb{P}(\hat{X}_t(y)\in B)\ddr t.$$
Fix $q>0$, recall $\int_{v_t(q)}^{q}\frac{\ddr x}{\Psi(x)}=t$ and $v_\infty(q)=0$. One has
\begin{align*}
\mathbb{E}[U_p(\mathbbm{e}_q,B)]=\int_{0}^{\infty}U_p(y,B)qe^{-qy}\ddr y&=\int_{0}^{\infty}\int_{0}^{\infty}v_t(q)e^{-pt}e^{-uv_t(q)} \mathbbm{1}_B(u)\ddr u\ddr t\\
&=\int_{0}^{\infty}\mathbbm{1}_B(u)\ddr u\int_{0}^{q}e^{-p\int_{x}^{q}\frac{\ddr v}{\Psi(v)}}e^{-ux}\frac{x}{\Psi(x)}\ddr x.
\end{align*}
By monotone convergence
\begin{equation}\label{potential}
\underset{p\rightarrow 0}{\lim} \uparrow \int_{0}^{\infty}U_p(y,B)qe^{-qy}\ddr y=\int_{0}^{\infty}U_0(x,B)qe^{-qx}\ddr x
=\int_{0}^{\infty}\mathbbm{1}_B(u)\ddr u\int_{0}^{q}e^{-ux}\frac{x}{\Psi(x)}\ddr x.
\end{equation} 
Set $B=]0,a[$ for $a>0$, then
$$\int_{0}^{\infty}U_0(x,B)qe^{-qx}\ddr x=\int_{0}^{a}\ddr u\int_{0}^{q}e^{-ux}\frac{x}{\Psi(x)}\ddr x=\int_{0}^{q}\frac{1-e^{-ax}}{\Psi(x)}\ddr x.$$
In the subcritical case $\Psi'(0+)>0$, therefore $\int_{0}^{q}\frac{x}{\Psi(x)}\ddr x<\infty$ and for almost every $x\in ]0,\infty[$, one has $0<U_0(x,B)<\infty$. Since for any $x\leq y$, $\hat{X}_t(x)\leq \hat{X}_t(y)$ then $$\mathbb{P}(\hat{X}_t(x)<a)\geq \mathbb{P}(\hat{X}_t(y)<a)$$
therefore $U_0(x,B)\geq U_0(y,B)$ and then $0<U_0(x,B)<\infty$ for all $x$. We may now invoke Proposition 2.2-(iv') in Getoor \cite{MR580144}, by taking the increasing sequence $B_n:=]0,n[$. This  entails that the process $(\hat{X}_t,t\geq 0)$ is transient. In the critical case, if $\int_{0}^{q}\frac{x}{\Psi(x)}\ddr x<\infty$ then
the process is transient. If now $\int_{0}^{q}\frac{x}{\Psi(x)}\ddr x=\infty$ then by (\ref{potential}) for any set $B$ with positive Lebesgue measure, $U_0(x,B)=\infty$ for all $x$. By Proposition 2.4-(i) in \cite{MR580144}, we conclude that $(\hat{X}_t,t\geq 0)$ is recurrent.
\end{proof}

\section{The Feller flow}
\label{fellerflow}

In this section, we investigate the genealogy backwards in time of Feller CSBPs. These are continuous CSBPs with quadratic branching mechanisms of the form $\Psi : q \mapsto \frac{\sigma^{2}}{2}q^{2}-\beta q$, with $\beta \in \mathbb{R}$ and $\sigma^2 \geq 0$.  For any fixed $x$, the Feller CSBP $(X_t(x),t\geq 0)$ with mechanism $\Psi$  can be constructed as the solution of the stochastic differential equation
\[\ddr X_t(x)=\sigma\sqrt{X_t(x)}\ddr B_t+\beta X_t(x)\ddr t,\qquad X_0(x)=x\]
where $(B_t,t\geq 0)$ is a Brownian motion. We study here in detail the flow $(X_{s,t}(x),t\geq s, x\geq 0)$ of CSBPs with branching mechanism $\Psi$ and the inverse flow $(\hat{X}_{s,t}(x),t\geq s, x\geq 0)$. Many calculations can be made explicit in this setting, see for instance Pardoux \cite{MR2507503} for a study of the flow $(X_{s,t}(x),t\geq s, x\geq 0)$.

Note that $\Psi$ is subcritical if $\beta < 0$, critical if $\beta = 0$ and supercritical if $\beta > 0$. Moreover, in the latter case we have $\rho = \frac{2\beta}{\sigma^2}$. Observe also that the differential equation \eqref{odev} can be rewritten
\[
  \frac{\ddr v_t(\lambda)}{\ddr t} = v_t(\lambda) \left( \frac{\sigma^2}{2} v_t(\lambda) - \beta \right), \quad \text{with } v_0(\lambda) = \lambda,
\]
and it is a simple exercise to solve it into
\[
  v_t(\lambda) =
  \begin{cases}
    \frac{\lambda \beta e^{\beta t}}{\beta+\frac{\lambda \sigma^{2}}{2}\left(e^{\beta t}-1\right)} & \text{ if } \beta \neq 0\\
    \frac{\lambda}{1+\frac{\sigma^{2}\lambda t}{2}} & \text{  if } \beta = 0.
  \end{cases}
\]
Therefore, one can write $v_t(\lambda) = \int_{0}^{\infty} (1 - e^{-\lambda r}) \ell_t(\ddr r)$, by setting
\[
  \ell_t(\ddr r) =  v_t(\infty)^2 e^{-\beta t} e^{-v_t(\infty) e^{-\beta t} r } \ddr r 
\]
where by definition, $v_t(\infty) = \frac{2\beta }{\sigma^{2}(1-e^{-\beta t})}>0$ for $\beta \neq 0$ and $v_t(\infty) = \frac{2}{t\sigma^2}$ if $\beta =0$. Observe that in both cases, $\frac{\ell_t}{\ell_t([0,\infty])}$ is an exponential law with parameter $\hat{\beta}_t = v_t(\infty)e^{-\beta t}$, which can be rewritten as
\[
  \hat{\beta}_t =\begin{cases}
     \frac{2\beta }{\sigma^{2}(e^{\beta t}-1)} & \text{ if } \beta \neq 0\\
    \frac{2}{t \sigma^2} & \text{ if } \beta = 0.
   \end{cases}
\]

\begin{remark}
Observe that above, we often make a distinction between $\beta \neq 0$ and $\beta = 0$, but the functions $v_t$, $\ell_t$ or $\hat{\beta}_t$ that we defined are continuous at $\beta=0$.
\end{remark}

We now study the law of the inverse Feller flow $(\hat{X}_{s,t}(y), s \leq t, y \geq 0)$, in particular characterizing its marginal distributions as a process in the variable $t$ or $y$.
\begin{theorem}
\label{dualdiffusionflow}
The inverse flow $(\hat{X}_t(x), x\geq 0, t\geq 0)$  is characterized as follows. Setting
\[
  \forall t \geq 0, \lambda \geq 0, \hat{v}_t(\lambda)=\frac{\lambda\hat{\beta}_t}{\lambda+\hat{\beta}_te^{\beta t}},
\]
we have
\begin{enumerate}
  \item for any fixed $y\geq 0$, $(\hat{X}_t(y),t\geq 0)$ is a Markov process with semigroup given by
\[\mathbb{E}[e^{-\lambda \hat{X}_{t}(y)}]=e^{-y \hat{v}_t(\lambda)-\frac{\sigma^{2}}{2}\int_{0}^{t}\hat{v}_s(\lambda)\ddr s}.\]
  \item For any fixed $t$, $(\hat{X}_t(y),y\geq 0)$ is a subordinator with Laplace exponent $\hat{v}_t$ started from the positive random variable $\hat{X}_t(0)$ whose Laplace transform is  $\mathbb{E}[e^{-\lambda \hat{X}_{t}(0)}]=e^{-\frac{\sigma^{2}}{2}\int_{0}^{t}\hat{v}_s(\lambda)\ddr s}$.
\end{enumerate}
\end{theorem}
\begin{remark} The map $(\hat{v}_t(\lambda),t\geq 0)$ is solution to 
\eqref{odev} with function $\hat{\Psi}(q):=\frac{\sigma^2}{2}q^{2}+\beta q$. The two-parameter process $(\hat{X}_t(x),t\geq 0, x\geq 0)$ is a flow of continuous-state branching processes with immigration with mechanisms $\hat{\Psi}$ and linear immigration $\hat{\Phi}(q):=\frac{\sigma^{2}}{2}q$. In particular, $(\hat{X}_t(x)-\hat{X}_t(0),t\geq 0)$ is a Feller CSBP with branching mechanism $\hat{\Psi}$.
\end{remark}
\begin{proof}
As $x \mapsto X_{-t,0}(x)$ is a subordinator with Lévy-Khintchine exponent
\[
  v_t(\lambda) = \hat{\beta}_t e^{\beta t} \int_0^\infty \left(1 - e^{-\lambda r} \right) \hat{\beta}_t e^{-\hat{\beta}_t r } \ddr r,
\]
we obtain that this is in fact a compound Poisson process, with jump rate $\hat{\beta}_t e^{\beta t}$ and exponential jump distribution with parameter $\hat{\beta}_t$. Therefore, writing $(N_x^{(t)}, t \geq 0)$ an homogeneous Poisson process with intensity $\hat{\beta}_te^{-\beta t}$ and $(x^{(t)}_i, i \geq 1)$ i.i.d. exponential random variables with parameter $\hat{\beta}_t$, one can rewrite $X_{-t,0}$ as
\begin{equation}
  \label{eqn:fellerRepresentation}
  \forall x \geq 0, X_{-t,0}(x) = \sum_{j=1}^{N^{(t)}_x} x^{(t)}_j.
\end{equation}

We set $(\tau^{(t)}_j, j \geq 1)$ the sequence of inter-arrival times of $(N^{(t)}_x, x \geq 0)$ which are i.i.d. exponential random variables, and reciprocally we write $M_y^{(t)} = \sup\{ n \geq 1: \sum_{i=1}^n x^{(t)}_i \leq y \}$ for all $y \geq 0$, which is the Poisson process with inter-arrival times $(x^{(t)}_j, j \geq 1)$. We observe that by \eqref{eqn:fellerRepresentation}, $\hat{X}_{0,t}$ being the right-continuous inverse of $X_{-t,0}$, we have
\begin{equation}
  \label{eqn:backwardFellerRepresentation}
  \hat{X}_{0,t}(y) = \sum_{j=1}^{M^{(t)}_y+1} \tau^{(t)}_j.
\end{equation}
Note that we have $\hat{X}_{0,t}(0) > 0$, contrarily to $X_{0,t}(0)=0$, but that $\hat{X}_{0,t}$ is also a compound Poisson process with exponential jump rate. The construction of $X_{-t,0}$ and $\hat{X}_{0,t}$ are represented on Figure \ref{renewal}.

\begin{figure}[!ht]
\centering \noindent
\includegraphics[height=.24 \textheight]{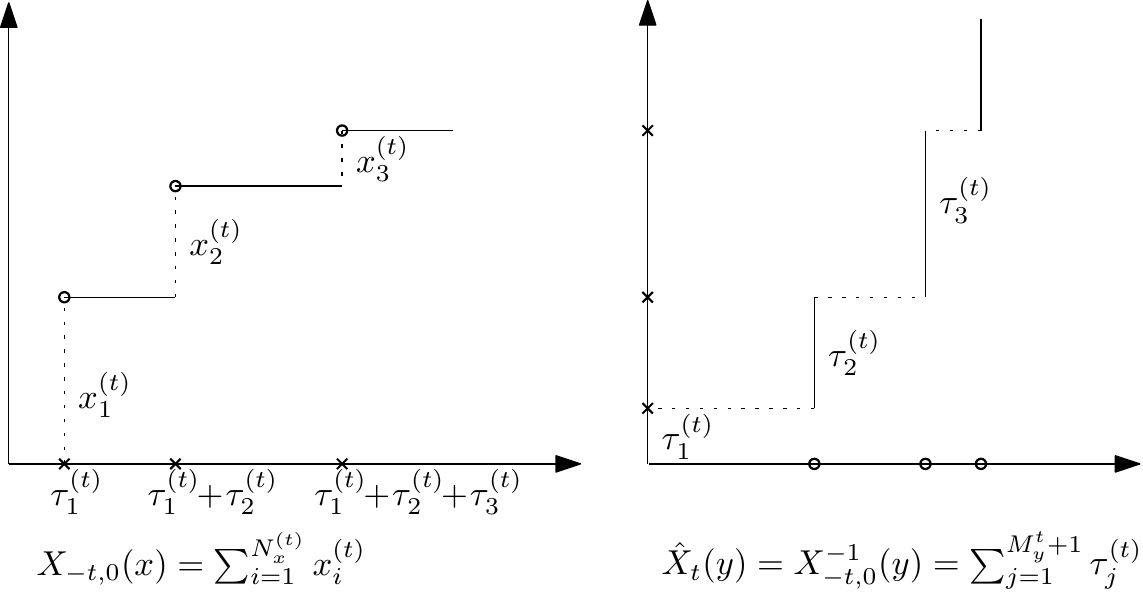}
\label{renewal}
\caption{Inverse of compound Poisson process}
\end{figure}

Note that by \eqref{eqn:backwardFellerRepresentation}, we have that
\[
  \E\left( e^{-\lambda \hat{X}_{0,t}(0)} \right) = \E\left( e^{-\lambda \tau^{(t)}_1} \right) = \frac{\hat{\beta}_t}{\hat{\beta_t} + e^{\beta t}\lambda},
\]
and moreover, for all $y \geq 0$
\[
  \mathbb{E}[e^{-\lambda (\hat{X}_t(y)-\hat{X}_t(0))}] = \exp\left( - y \frac{\lambda \hat{\beta}_t}{\hat{\beta}_t e^{\beta t} + \lambda} \right) = e^{-y \hat{v}_t(\lambda)}
\]
by straightforward Poisson computations. It remains to verify that
\[\mathbb{E}[e^{-\lambda \hat{X}_t(0)}]=e^{-\int_{0}^{t}\hat{v}_s(\lambda)\ddr s}.\]
We observe that $e^{-\int_{0}^{t}\hat{v}_s(\lambda)\ddr s}=\frac{\frac{\sigma^{2}}{2}\hat{v}_t(\lambda)+\beta}{\frac{\sigma^{2}}{2}\lambda+\beta}$, thus deduce that $\frac{\hat{\beta}_t}{\hat{\beta}_t+e^{\beta t}\lambda}=e^{-\int_{0}^{t}\hat{v}_s(\lambda)\ddr s}$  from the identity $\hat{v}_t(\lambda)=\frac{\lambda \hat{\beta}_t}{\lambda e^{\beta t}+\hat{\beta}_t}$,which concludes the proof.
\end{proof}

The inverse Feller flow being itself a flow of subordinators with explicit law, many quantities can be computed explicitly, such as the most recent common ancestor of a population. Picking two individuals $x \leq y$ at time $0$, the age $T_{x,y}$ of the most recent common ancestor of $x$ and $y$ is the first time $t$ such that there exists an individual $z$ at generation $-t$ that gave birth to both $x$ and $y$, or more precisely
\begin{equation}
  \label{eqn:tCoal}
  T_{x,y} = \inf\{ t \geq 0 : \hat{X}_{0,t}(x) = \hat{X}_{0,t}(y) \}.
\end{equation}
This definition of most recent common ancestor can naturally be generalized as follows: given $A$ a subset of $\R_+$, we set
\[
  T_A = \inf\{ t \geq 0 : \#\{\hat{X}_{0,t}(A)\} = 1\}.
\]
However, as  the partition of $\R_+$, $\cup_{z \geq 0} \left(\hat{X}_{0,t}\right)^{-1}(\{z\})$ is a partition in intervals, we have
\[
  T_A = T_{\inf A, \sup A} \quad \text{a.s.}
\]
Therefore, obtaining the law of $T_{x,y}$ will be enough to study the genealogy of the Feller flow.

\begin{proposition}
\label{prop:coalescentFeller}
For any $0 \leq x \leq y \leq z$, we have
\[
 \forall t \geq 0, \P\left( T_{x,y} \leq t\right) = e^{-\hat{\beta}_t(y-x)},
\]
and $T_{x,y}$ and $T_{y,z}$ are independent. In particular, we have
\[
  \P\left( T_{x,y} = \infty \right) =
  \begin{cases}
    e^{\frac{2\beta}{\sigma^2} (y-x)} & \text{if } \beta < 0\\
    0 &\text{if } \beta \geq 0.  
  \end{cases}
\]
\end{proposition}

Among other things, this proposition proves that the population comes down from a single ancestor in critical or supercritical cases ($\beta \geq 0$), while in the subcritical case, for $\beta < 0$, the population at time $0$ can be separated into families with different ancestors at time $-\infty$.

\begin{proof}
This result is a consequence of the inverse flow representation of Theorem \ref{dualdiffusionflow}. Indeed, for all $x \leq y$ and $\lambda \geq 0$, we have
\[
  \mathbb{E}[e^{-\lambda (\hat{X}_t(y)-\hat{X}_t(x))}]=e^{-(y-x)\hat{v}_t(\lambda)},
\]
thus, letting $\lambda \to \infty$ we obtain $\mathbb{P}(T_{x,y}\leq t)=e^{-(y-x)\hat{v}_t(\infty)}$. Moreover, we observe that $\hat{v}_t(\infty) = \hat{\beta}_t$, proving the first equation.

By \eqref{eqn:tCoal}, and given that $(\hat{X}_{0,t}(y)-\hat{X}_{0,t}(x), t \geq 0)$ and $(\hat{X}_{0,t}(z)-\hat{X}_{0,t}(y), t \geq 0)$ are independent Feller CSBP with mechanism $\hat{\Psi}$, starting from $y-x$ and $z-y$ respectively, we conclude that $T_{x,y}$ and $T_{y,z}$ are independent.

To obtain $\P(T_{x,y} = \infty)$ we compute
\[
  \lim_{t \to \infty} \P(T_{x,y} \geq t) = \exp\left( - (y-x) \lim_{t \to \infty} \hat{\beta}_t \right) = \begin{cases} 0 &\text{if } \beta \geq 0 \\ e^{-(y-x)\frac{-2\beta}{\sigma^2} } & \text{if } \beta < 0, \end{cases}
\]
concluding the proof.
\end{proof}

\begin{remark}
A straightforward consequence of the above coalescent is that for any choice of $\{x_1,\ldots x_n\}$ of individuals at generation $0$, the coalescent tree of this family of individuals will only consist in binary merging. Indeed, for every pair $(x_i,x_{i+1})$ of consecutive individuals, their time of coalescence is independent from the time of coalescence of any other pair of consecutive individuals in the population, and has density with respect to the Lebesgue measure. Therefore, almost surely the first coalescing time will consists in the merging of only two neighbours.
\end{remark}

Proposition \ref{prop:coalescentFeller} readily entails the representation of the genealogical tree of the population at time $0$ as a functional of a Poisson point process. The following construction is reminiscent of the comb representation  by Lambert and Uribe Bravo \cite{LambertUribe}.

\begin{proposition}
\label{lem:comb}
There exists a Poisson point process  $N$ with intensity $\ddr x\otimes \mu(\ddr t)$ on $\R_+ \times (\R_+ \cup \{\infty\})$ where 
$$\mu(\ddr t)=\begin{cases}\frac{2\beta^{2}}{\sigma^{2}}\frac{e^{\beta t}}{(e^{\beta t}-1)^{2}}\ddr t &\text{ if } \beta\neq 0\\
\frac{2}{\sigma^{2}t^{2}}\ddr t &\text{ if } \beta=0,\\
\frac{2\beta^{2}}{\sigma^{2}}\frac{e^{\beta t}}{(1-e^{\beta t})^{2}}\ddr t + \frac{2|\beta|}{\sigma^2}\delta_{\infty} &\text{ if } \beta< 0
\end{cases}$$
such that almost surely, for any $0 \leq x \leq y$ and $t \geq 0$, we have
\[
  T_{x,y} <t \iff N([x,y] \times [t,\infty]) = 0.
\]
\end{proposition}

In other words, the coalescent time of $x$ and $y$ is given by the position of the largest atom in the point process $N([x,y] \times \cdot)$. In particular, in the critical case ($\beta = 0$), this result recovers the Brownian coalescent point process of Popovic \cite[Lemma~4 and Theorem~5]{MR2100386}.  In the subcritical case ($\beta < 0$) when two individuals have no common ancestor, there are separated by an infinite atom of the point process $N([x,y] \times \cdot)$.
\begin{proof}
We observe from Proposition \ref{prop:coalescentFeller} that for all $x \leq y \leq z$, $T_{x,y}$ and $T_{y,z}$ are independent and $T_{x,z} \egaldistr \max(T_{x,y},T_{y,z})$. Moreover, note by definition that $T_{x,z} \geq \max(T_{x,y},T_{y,z})$ a.s. This yields that for all $x \leq y \leq z$, 
\begin{equation}
  \label{eqn:equalityDefiningthePP}
  T_{x,z} = \max(T_{x,y}, T_{y,z}) \quad \text{a.s.}
\end{equation}
We consider the event of probability one for which the above equation is true simultaneously for all $x,y,z \in \mathbb{Q}_+$. As a result, the field $T_{x,y}$ is decreasing in $x$ and increasing in $y$. Therefore, there exists a càdlàg modification of the field satisfying \eqref{eqn:equalityDefiningthePP} simultaneously for all $x,y,z \in \R_+$.

As a result, we can construct a simple point process $G$ on $\R_+ \times (\R_+ \cup \{\infty\})$ satisfying
\[
  T_{x,y} < t \iff G([x,y] \times [t,\infty]) = 0,
\]
via the construction $G = \sum_{z \geq 0} \ind{\exists \epsilon > 0: T_{z-\epsilon,z-} < T_{z-\epsilon,z}} \delta_{z,T_{z-\epsilon,z}}$. The point process $G$ is simple (i.e. each atom in the point process has mass one). Moreover, note that
\begin{align*}
  \P(G([x,y] \times [t,\infty] = 0) &= \P(T_{x,y}<t) = \exp\left( - \int_t^\infty \int_x^y \hat{\beta}_s' \ddr z \ddr s - \int_x^y \hat{\beta}_\infty \ddr s  \right)\\
  &= \P(N ([x,y] \times [t,\infty])= 0),
\end{align*}
where $N$ is a Poisson point process with intensity $\ddr x \otimes (\hat{\beta}_t' \ddr t + \hat{\beta}_\infty \delta_{\infty}(\ddr t))$. Hence, by monotone classes theorem, for all measurable relatively compact set $B \subset \R_+ \times (\R_+ \cup \{\infty\})$, we have
\[
  \P(G(B) = 0) = \P(N(B) = 0).
\]
As a result, by \cite[Theorem 10.9]{Kallenberg}, we have $N \egaldistr G$, which concludes the proof.
\end{proof}

Pitman and Yor \cite[Sections 3 and 4]{MR656509}, see also \cite{MR3164759} for a more general setting, have shown that any flow of Feller's branching diffusions can be represented through a Poisson point process on $]0,\infty[\times \mathrm{C}$, where $\mathrm{C}$ denotes the space of continuous paths on $\mathbb{R}_{+}$. In our setting, the flow $(\hat{X}_t(x)-\hat{X}_t(0),t\geq 0)$ can be represented as follows: for all $t>0$,
\[\hat{X}_t(x)=\hat{X}_t(0)+\sum_{\substack{x_i\leq x \\ i\in I}}\hat{X}^{i}_t\]
where $\mathcal{N}=\sum_{i\in I}\delta_{(x_i,\hat{X}^{i})}$ is a PPP with intensity $\ddr x\otimes \mathbf{n}(\ddr X)$ and $\mathbf{n}$ is the so-called cluster measure (see \cite{MR3164759}). The atoms $(\hat{X}^i, i\in I)$ can be interpreted as the ancestral lineages of the initial individuals $(x_i,i\in I)$. They are independent Feller diffusions with mechanism $\hat{\Psi}$ starting from infinitesimal masses. For any $i\in I$, denote by $\zeta_i:=\inf\{t\geq 0; \hat{X}^{i}_t=0\}$. The time $\zeta_i$ represents a binary coalescence time between two ``consecutive'' individuals. By definition of $\mathbf{n}$, for any $t>0$, $\mathbf{n}(\zeta>t)=\hat{v}_t(\infty)$ and therefore $\sum_{i\in I}\delta_{(x_i, \zeta_i)}$ has the same law as $N$. We represent the ancestral lineages and their coalescences in Figure \ref{renewalfig}. Recall also from Remark \ref{inverseflowfrom0} that $\hat{X}_t(0)$ is the first individual from generation $-t$ to have descendants at time $0$.

\begin{figure}[!ht]
\centering \noindent
\includegraphics[height=.20 \textheight]{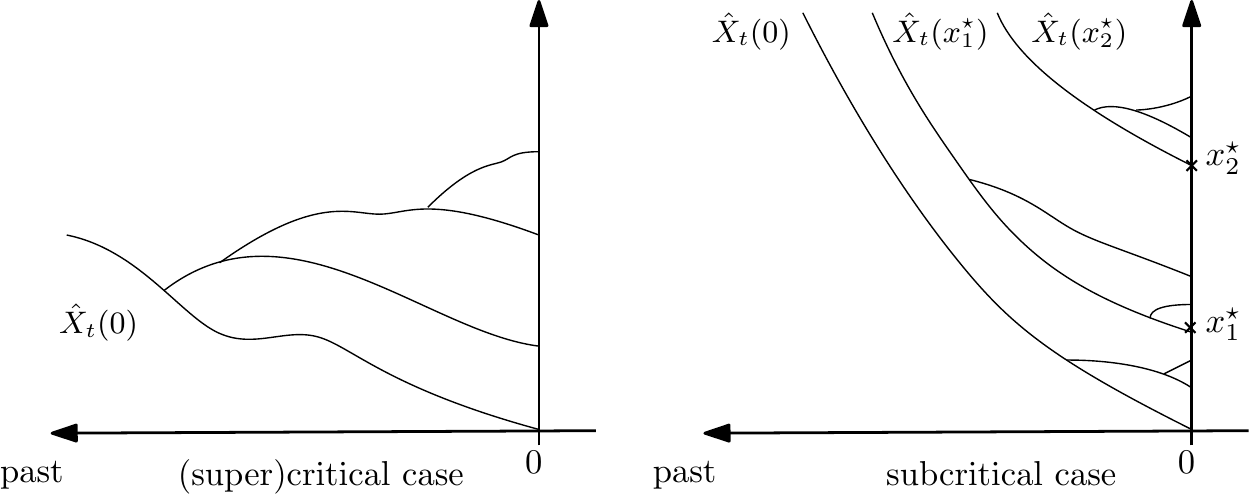}
\caption{Symbolic representation of ancestral lineages and their binary coalescences}
\label{renewalfig}
\end{figure}

In the subcritical case, $(\beta<0)$, $(\hat{X}_t(x)-\hat{X}_t(0),x\geq 0,t\geq 0)$ is a flow of supercritical CSBPs. Following Bertoin et al. \cite{MR2455180} (see also Pardoux \cite[Section 7]{MR2507503}), one can define the random sequence $(x^{\star}_n,n\geq 1)$ recursively as follows:
\[x_1^{\star}:=\inf\{x\geq 0; \hat{X}_t(x)-\hat{X}_t(0)\underset{t\rightarrow \infty}{\longrightarrow} \infty\}\text{ and }x_{n+1}^{\star}:=\inf\{x\geq x_n^{\star}; \hat{X}_t(x)-\hat{X}_t(0)\underset{t\rightarrow \infty}{\longrightarrow} \infty\}.\]
The random sequence $(x^{\star}_n,n\geq 1)$ is known as the initial prolific individuals of the flow of supercritical CSBPs  $(\hat{X}_t(x)-\hat{X}_t(0),x\geq 0,t\geq 0)$ and corresponds to the jumps times of a Poisson process with intensity $-\frac{2\beta}{\sigma^2}$. Within the framework of inverse flow, the random partition of $\mathbb{R}_+$: $([0,x_1^{\star}[,[x_1^{\star},x_{2}^{\star}[,...)$ corresponds to current families with distinct common ancestors. Note that the sequence $(x_n^\star, n \geq 1)$ is also the sequence of atoms of the point process $N(\cdot \times\{\infty\})$, defined in Lemma \ref{lem:comb}.

\bigskip

We observed in this section that the law of the flow $\hat{X}$ is explicit when $X$ is a Feller flow. When the branching mechanism $\Psi$ is not of the quadratic form, multiple births occur in the population. Thus, when time runs backward, coalescences of multiple lineages should arise. The law of the inverse flow $\hat{X}$ becomes then more involved. In the next section, we construct a simple class of Markovian coalescents which will allow us to encode easily multiple coalescences in lineages backward in time. The law of the lineage's location $(\hat{X}_t(x),t\geq 0)$ for a fixed individual $x\geq 0$ is studied further in Section \ref{martingaleproblem}.

\section{Consecutive coalescents}
\label{sec:markovianCoalescent}

In this section we study the genealogy of branching processes both forward and backward in time, using random partitions of consecutive integers. We shall see how to define a coalescent process in this framework and that the associated coalescent theory is elementary. In a second time, 
we apply these results to the genealogy of a population in a continuous-state branching process sampled according to a Poisson point process with intensity $\lambda$. In a third time, by making the parameter $\lambda$ increase to $\infty$, we obtain a full description of the genealogical tree of individuals in a CSBP under Grey's condition. 

\subsection{Consecutive coalescents in continuous-time Galton-Watson processes}

In this section, we construct a class of simple Markovian coalescents arising when studying the genealogy backward in time of continuous-time Galton-Watson processes. We begin by introducing the classical notation for coalescent processes on the space of partitions. For a more precise description of that framework, in the context of exchangeable coalescents, we refer to Bertoin's book \cite[Chapter 4]{MR2253162}, from which we borrow our definitions and notation.

Let $n \in \N \cup \{\infty\}$, we denote by $[n] = \{j \in \N : j \leq n\}$ the set of integer smaller or equal to $n$. We call \emph{consecutive partition} of $[n]$ a collection $C$ of disjoint subsets $\{C_1,C_2,\ldots\}$ with consecutive integers (i.e. intervals of $[n]$), such that $\bigcup C_i = [n]$. Without loss of generality, we will always assume that the subsets of the collection $C$ are ranked in the increasing order of their elements. We denote by $\mathcal{C}_n$ the set of consecutive partitions of $[n]$. Note that any $C \in \mathcal{C}_n$ is characterized by the ranked sequence of its blocks sizes $(\#C_1,\#C_2,\ldots)$, as
\[
   \forall j \in \N, \quad C_j = \left\{k \in \N : \#C_1 + \cdots + \#C_{j-1} < k \leq \#C_1 + \cdots + \#C_{j} \right\}.
\]
Clearly, a consecutive partition has at most one block with infinite size. For any $i,j\in [n]$, we write $i\overset{C}{\sim} j$ if and only if $i$ and $j$ belongs to the same block of $C$. For any $n\in \mathbb{N}\cup \{\infty\}$, we set $0_{[n]}=(\{1\},\{2\},...,\{n\})$ and $1_{[n]}=\{[n]\}$. We introduce some classical operations on $\mathcal{C}_n$. For each $k \leq n$ and $C \in \mathcal{C}_n$, we denote by
\[
  C_{|[k]}  = \left\{ C_j \cap [k], j \in \N \right\},
\]
the restriction of $C$ to $[k]$ and
\[ \#C_{|[k]}:=\#\{j \in \N : C_j\cap [k] \neq \emptyset\}, \]
the number of blocks of $C_{|[k]}$. Note that for any $m\leq k$, $(C_{|[k]})_{|[m]}=C_{|[m]}\in \mathcal{C}_m$.

We define a distance on $\mathcal{C}_\infty$ the set of consecutive partitions of $\N$ by setting
\[
  d(C,C') = \sup\{ n \in \N : C_{|[n]} = C'_{|[n]}\}^{-1}.
\]
Note that the metric space $(\mathcal{C}_\infty,d)$ is compact. We next introduce the coagulation operation. For any $C \in \mathcal{C}_n$ and $C' \in \mathcal{C}_{n'}$ such that $\#C\leq n'$, we define the partition $\Coag(C,C')$ by
\[
\Coag(C,C')_j = \bigcup_{i \in C'_j} C_i \quad \text{ for any } j \in \N.
\]
It is straightforward that $\Coag(C,C') \in \mathcal{C}_n$, as each of its blocks are the union of a consecutive sequence of consecutive blocks. Thus, $\Coag$ defines an internal composition law on $\mathcal{C}_\infty$. Moreover for any $C$, $C'$ such that $\#C\leq \#C'$ and $n\geq 1$
\[\Coag(C,C')_{|[n]}=\Coag(C_{|[n]}, C'_{|[n]})=\Coag(C_{|[n]}, C').\]
The operator $\Coag$ is therefore Lipschitz continuous with respect to $d$ and we easily see that it is associative. For any partition $C\in \mathcal{C}_n$, $\Coag(C,0_{[n]})=C$ and $\Coag(C,1_{[n]})=1_{[n]}$.

We are interested in random consecutive partitions such that blocks sizes $(\#C_j, j \geq 1)$ are i.i.d. random variables in $\N \cup \{\infty\}$. We observe that if $C$ and $C'$ are two independent random consecutive partitions with i.i.d. block sizes, then
\[
  \# \Coag(C,C')_j = \sum_{i \in C'_j} \#C_i \overset{\mathcal{L}}{=} \sum_{i \in C'_1} \#C_i,
\]
hence $\Coag(C,C')$ is a random consecutive partition with blocks coarser than those of $C$, and with i.i.d. sizes. In view of the very particular form of a consecutive partition, it is legitimate to question whether the framework of partitions is needed. However, the use of the operator $\Coag$ enables us to encode easily multiple coalescences and to follow closely the theory of exchangeable coalescents and its terminology. This encoding simplifies the main formulas we obtain when studying the genealogy of a continuous-state branching population.
 
\begin{definition}
\label{consecutivecoal}
A Markov process $(C(t),t\geq 0)$ with values in $\mathcal{C}_\mathbb{N}$ is called consecutive coalescent if its semigroup is given as follows: the conditional law of $C(t+s)$ given $C(t)=C$ is the law of $\Coag(C,C')$ where $C'$ is some random consecutive partition with i.i.d blocks sizes and whose law may depend on $t$ and $s$. A consecutive coalescent is said to be homogeneous if $C'$ depends only on $s$ and standard if $C(0)=0_{[\infty]}$.
\end{definition}

We now recall further well-known material on continuous-time Galton-Watson processes. We refer to Athreya and Ney \cite[Chapter III]{MR2047480} for more details on these processes. Consider a finite measure $\mu$ on $\mathbb{Z}_+$ such that $\mu(1)=0$. A continuous-time Galton-Watson process $(Z_t(n),t\geq 0)$ with reproduction measure $\mu$, is a Markov process counting the number of individuals in a random population with $n$ ancestors where all individuals behave independently, and each individual has an exponential lifetime $\zeta$ with parameter $\mu(\mathbb{Z}_+)$ and begets at its death a random number of children with probability distribution $\mu/\mu(\mathbb{Z}_+)$. The process $(Z_t(n),t\geq 0)$ is characterized in law by $\mu$ and thus by the function
\[
  \psi(x)=-\sum_{k=0}^{\infty}(x^{k}-x)\mu(k),\qquad x\in [0,1].
\]
The process $(Z_t(n),t\geq 0)$ satisfies the branching property
\begin{equation}
  \label{eqn:discretebranchProp}
  \forall n,m\geq 0, \forall t\geq 0, \  Z_t(n+m)\egaldistr Z_t(n)+\tilde{Z}_t(m) \qquad\text{ for any } n,m\in \mathbb{Z}_+
\end{equation}
where $(\tilde{Z}_t(m),t\geq 0)$ is a continuous-time Galton-Watson process independent of $(Z_t(n),t\geq 0)$, and with the same law as $(Z_t(m),t\geq 0)$. This entails that the generating function of $Z_t(n)$ for any $t\geq 0$ has the form 
\[
  \mathbb{E}[x^{Z_t(n)}]=u_t(x)^{n}, \qquad x\in [0,1], n\in \mathbb{Z}_+
\]
where for all $t \geq 0$,  $u_{t}(s)$ is the solution of $\int_{u_{t}(s)}^{s}\frac{\ddr z}{\psi(z)}=t$ for any $t\geq 0$. When $\mu$ has no mass at $0$, the process 
is called \textit{immortal}. Each individual has at least two children and $(Z_t(n),t\geq 0)$ is non-decreasing in time.

With the same procedure as in Definition~\ref{flowdef2}, we can represent the family of continuous-time branching processes by considering a flow of random walks $(Z_{s,t}(n),t\geq s, n\geq 1)$ satisfying the following properties:
\begin{enumerate}
  \item  for any $s\leq t$, $(Z_{s,t}(n),n\geq 0)$ is a continuous-time random walk whose jump law has support included in $\mathbb{N}$ and generating function $u_{t-s}$.
  \item  For every $t_1<t_2<...<t_p$, the random walks $(Z_{t_i,t_{i+1}}, i < p)$ are independent and satisfy
\[\forall n \geq 0, \ Z_{t_1,t_p}(n)=Z_{t_{p-1},t_p}\circ ...\circ Z_{t_{1},t_2}(n).\]
  \item For any $n \geq 1$ and $s \in \R$, $(Z_{s,t+s}(n),t\geq 0)$ is an immortal homogeneous continuous-time Galton-Watson process started from $n$ individuals.
\end{enumerate}

We now construct a flow of partitions describing the genealogy of an immortal continuous-time Galton-Watson process constructed via this flow of random walks. For any $s\leq t$, set
\[\overset{\rightarrow}{C}(s,t):=\left([|Z_{s,t}(i-1)+1,Z_{s,t}(i)|],i\geq 1\right),\]
with $Z_{s,t}(0)=0$. Then by (ii), for any $r < s < t$,
\[\overset{\rightarrow}{C}(r,t)=\Coag(\overset{\rightarrow}{C}(s,t),\overset{\rightarrow}{C}(r,s))\]
i.e. $\overset{\rightarrow}{C}_j(r,t)= \bigcup_{i\in \overset{\rightarrow}{C}_j(r,s)}\overset{\rightarrow}{C}_i(s,t)$
and by definition $\#\overset{\rightarrow}{C}_i(s,t)=Z_{s,t}(i)-Z_{s,t}(i-1)$ for any $s \leq t$ and any $i\geq 1$. We introduce the time-reversed flow of partition by defining for all $s\leq t$,
\[C(s,t)=\overset{\rightarrow}{C}(-t,-s).\]
We sum up the main straightforward properties of $C$ in the following proposition.

\begin{proposition}\label{flowpartitions}
The stochastic flow of consecutive partitions  $(C(s,t), -\infty \leq s\leq t\leq \infty)$ satisfies: 
\begin{enumerate}
  \item For any $s\leq u\leq t$ 
\[C(s,t)=\Coag(C(s,u),C(u,t)) \text{ a.s}\]
  \item If $s_1<s_2<...<s_n$, the partitions $C(s_1,s_2)$,..., $C(s_{n-1},s_n)$ are independent.
  \item The random variables $(\#C_i(s,t), i\geq 1)$ are valued in $\mathbb{N}$ and i.i.d.  
  \item $C(0,0)=0_{[\infty]}$nd $C(s,t)\rightarrow 0_{[\infty]}$ when $t-s\rightarrow 0$.
  \item  The random variable $C(s,t)$ has the same law as $C(0,t-s)$.
\end{enumerate}
\end{proposition}

The Markov process $(C(t), t\geq 0)$ defined by $C(t):=C(0,t)$ for any $t\geq 0$ is an homogeneous standard consecutive coalescent in the sense of Definition \ref{consecutivecoal}. Note that by (i), for any $s,t\geq 0$, $$C(t+s)=\Coag(C(t),C(t,t+s))$$
namely for any $j\geq 1$,
\begin{equation}
  \label{reversedbranching}
  C_j(t+s)=\bigcup_{i\in C_j(t,t+s)}C_i(t),
\end{equation}
so that consecutive blocks are merging as time runs. 

\begin{remark}
One can readily check from \eqref{reversedbranching} that for any $i\geq 1$, $s\geq 0$ and $t\geq 0$
\begin{equation}
  \label{box}
  \#C_i(t+s)=\sum_{m=1}^{\#C_i(t,t+s)}\#C_{m+\sum_{k=1}^{i-1}\#C_k(t,t+s)}(t).
\end{equation}
Processes satisfying \eqref{box} have been studied  in discrete time by Grosjean and Huillet \cite{MR3581248}. 
\end{remark}

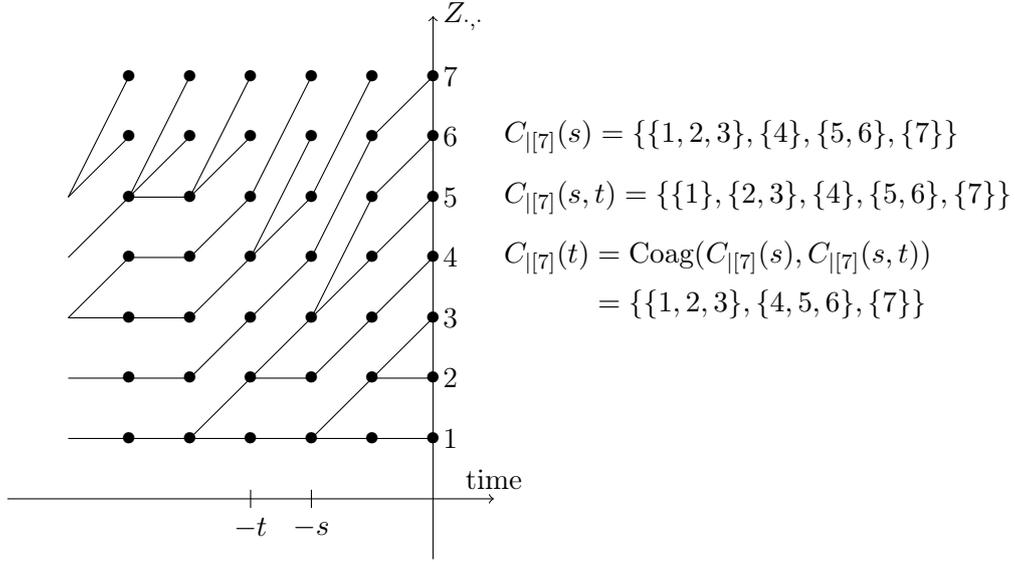
\begin{figure}
\centering
\begin{tikzpicture}[scale=0.8]
\draw [->] (-7,0) -- (1,0) node[above] {time};
\draw [->] (0,-1) -- (0,8) node[right] {$Z_{\cdot,\cdot}$};
\draw (0,1) node[right] {1};
\draw (0,2) node[right] {2};
\draw (0,3) node[right] {3};
\draw (0,4) node[right] {4};
\draw (0,5) node[right] {5};
\draw (0,6) node[right] {6};
\draw (0,7) node[right] {7};

\draw (-2,0.15) -- (-2,-0.15) node[below] {$-s$};
\draw (-3,0.15) -- (-3,-0.15) node[below] {$-t$};
\draw (1,6) node[right] {$C_{|[7]}(s) =\{ \{1,2,3\},\{4\},\{5,6\},\{7\} \}$};
\draw (1,5) node[right] {$C_{|[7]}(s,t) =\{ \{1\},\{2,3\},\{4\},\{5,6\},\{7\} \}$};
\draw (1,4) node[right] {$C_{|[7]}(t) = \Coag(C_{|[7]}(s),C_{|[7]}(s,t))$};
\draw (1,3.2) node[right] {$\phantom{C_{|[7]}(t) } = \{ \{1,2,3\},\{4,5,6\},\{7\} \}$};

\draw (-1,1) -- (0,1) node {$\bullet$};
\draw (-1,2) -- (0,2) node {$\bullet$};
\draw (-1,2) -- (0,3) node {$\bullet$};
\draw (-1,3) -- (0,4) node {$\bullet$};
\draw (-1,4) -- (0,5) node {$\bullet$};
\draw (-1,5) -- (0,6) node {$\bullet$};
\draw (-1,6) -- (0,7) node {$\bullet$};
\draw (-2,1) -- (-1,1) node {$\bullet$};
\draw (-2,1) -- (-1,2) node {$\bullet$};
\draw (-2,2) -- (-1,3) node {$\bullet$};
\draw (-2,3) -- (-1,4) node {$\bullet$};
\draw (-2,3) -- (-1,5) node {$\bullet$};
\draw (-2,4) -- (-1,6) node {$\bullet$};
\draw (-2,5) -- (-1,7) node {$\bullet$};
\draw (-3,1) -- (-2,1) node {$\bullet$};
\draw (-3,2) -- (-2,2) node {$\bullet$};
\draw (-3,2) -- (-2,3) node {$\bullet$};
\draw (-3,3) -- (-2,4) node {$\bullet$};
\draw (-3,4) -- (-2,5) node {$\bullet$};
\draw (-3,4) -- (-2,6) node {$\bullet$};
\draw (-3,5) -- (-2,7) node {$\bullet$};
\draw (-4,1) -- (-3,1) node {$\bullet$};
\draw (-4,1) -- (-3,2) node {$\bullet$};
\draw (-4,2) -- (-3,3) node {$\bullet$};
\draw (-4,3) -- (-3,4) node {$\bullet$};
\draw (-4,4) -- (-3,5) node {$\bullet$};
\draw (-4,5) -- (-3,6) node {$\bullet$};
\draw (-4,5) -- (-3,7) node {$\bullet$};
\draw (-5,1) -- (-4,1) node {$\bullet$};
\draw (-5,2) -- (-4,2) node {$\bullet$};
\draw (-5,3) -- (-4,3) node {$\bullet$};
\draw (-5,4) -- (-4,4) node {$\bullet$};
\draw (-5,5) -- (-4,5) node {$\bullet$};
\draw (-5,5) -- (-4,6) node {$\bullet$};
\draw (-5,5) -- (-4,7) node {$\bullet$};
\draw (-6,1) -- (-5,1) node {$\bullet$};
\draw (-6,2) -- (-5,2) node {$\bullet$};
\draw (-6,3) -- (-5,3) node {$\bullet$};
\draw (-6,3) -- (-5,4) node {$\bullet$};
\draw (-6,4) -- (-5,5) node {$\bullet$};
\draw (-6,5) -- (-5,6) node {$\bullet$};
\draw (-6,5) -- (-5,7) node {$\bullet$};
\end{tikzpicture}
\caption{Monotone labelling of an immortal Galton-Watson forest and its consecutive coalescent}
\end{figure}

By stationarity, for any $t\geq 0$, $C(t)\overset{\mathcal{L}}{=}\overset{\rightarrow}{C}(t)$. Therefore the coalescent process $(C(t),t\geq 0)$ is characterized by the reproduction measure $\mu$ of the associated continuous-time Galton-Watson process $(\#\overset{\rightarrow}{C}_1(t),t\geq 0)$. 
Moreover, note that by construction, for any $m\leq n$ 
\[(C_{|[m]}(t),t\geq 0)=((C_{|[n]}(t))_{|[m]},t\geq 0).\]
This consistency property ensures that the family of jump rates of $(C_{|[n]}(t),t\geq 0)$ characterizes the law of $(C(t),t\geq 0)$. In the next lemma, the coagulation rate of a consecutive coalescent restricted to $[n]$ is provided.
\begin{lemma}[Law of the $n$-coalescent]
\label{ratesofrestriction}
Let $n\in \N$ and $C\in \mathcal{C}_n$. Set $\#C=m$ and assume $C_{|[n]}(0)=C$. For any $j\in [m-1]$,
consider the consecutive partitions of $[m]$ 
\begin{itemize}
  \item[-] $C^{j,k}_{\mathrm{in}}:=(\{1\},...,\{j,...,j+k-1\},...,\{m\})$ for any $2\leq k\leq m-j$, and attach to each $C^{j,k}_{\text{in}}$ an independent exponential clock with parameter $\mu(k)$,
  \item[-] $C^{j}_{\mathrm{out}}:=(\{1\},...,\{j,...,m\})$, and attach to each $C^{j}_{\text{out}}$ an independent exponential clock with parameter $\bar{\mu}(m-j+1)$
\end{itemize}
where $\bar{\mu}(k):=\sum_{j=k}^{\infty}\mu(j)$ for any $k\in \mathbb{N}$. Then the process jumps from the partition $C_{|[n]}(t-)$ to $\text{Coag}(C_{|[n]}(t-),D)$
with $D$ the partition in $\{C^{j,k}_{\mathrm{in}}, C^j_\mathrm{out}\}$ associated to the first random clock that rings.
\end{lemma}

\begin{proof}
For any $n\geq 1$, any $t\geq 0$,
\[\Coag(C(t),C(t,t+s))_{|[n]}=\Coag(C_{|[n]}(t),C(t,t+s)).\]
Moreover, by associativity of the operator Coag, the restricted process $(C_{|[n]}(t),t\geq 0)$ starting from $C$, whose number of blocks is $m$, has the same law as the process $(\Coag(C,C_{|[m]}(t)),t\geq 0)$ where $(C_{|[m]}(t),t\geq 0)$ is the restriction at $[m]$ of the standard process started from $0_{[\infty]}$. Therefore, we only need to focus on the jump rates of the standard coalescent. 

For any $j$, the rate at which the process jumps from $0_{[m]}$ to $C_\mathrm{out}^{j}:=(\{1\},...,\{j,...,m\})$  is therefore given by 
\[\underset{s\rightarrow 0+}{\lim} \frac{1}{s}\mathbb{P}(\#C_j(t,t+s)\geq m-j+1).\]
Since $\#C_j(t,t+s)$ has the same law as the random variable $Z_s(1)$ where $(Z_t(1),t\geq 0)$ is a continuous-time Galton-Watson process with reproduction measure $\mu$, then the latter limit is $\bar{\mu}(m-j+1)$.

Similarly, for any $k\leq m-j-1$, the rate at which the process jumps from $0_{[m]}$ to $C_{\mathrm{in}}^{j,k}$ is
\[\underset{s\rightarrow 0+}{\lim} \frac{1}{s}\mathbb{P}(\#C_j(t,t+s)=k)=\mu(k).\]
There is no simultaneous births forward in time and therefore no simultaneous coalescences. 
\end{proof}

By letting $n$ and $m$ to $\infty$ in Lemma \ref{ratesofrestriction}, we see that the coalescences in the consecutive coalescent process $C$ valued in $\mathcal{C}_\infty$ can be described in the following way: to each block $j$ of $C$ is associated a family $(\mathbbm{e}_{j,k}, k\geq 2)$ of exponential clocks, that ring at rate $(\mu(k),k\geq 2)$. Each time a clock $\mathbbm{e}_{j,k}$ rings, the consecutive blocks $j$, $j+1$, ... , $j+k-1$ coalesce into one. Note that these clocks could also be used to construct the immortal Galton--Watson process forward in time: each time the clock $\mathbbm{e}_{j,k}$ rings, the $j$th individual produces $k$ children.

We now take interest in the number of blocks of a consecutive coalescent. Similarly to the continuous-state space, the dual process $\hat{Z}$ defined for any $n\in \mathbb{N}$ and any $t\geq 0$ by
\[\hat{Z}_t(n):=\min\{k\in \mathbb{N}: Z_{-t,0}(k)\geq n\}.\]
The process $\hat{Z}$ is a Markov process, and for all $n \in \N$, $\hat{Z}_t(n)$ is the ancestor at time $-t$ of the individual~$n$ considered at time $0$.

\begin{proposition}
\label{numberofblocks}
For any $t\geq 0$, and any $n,m\in \mathbb{N}$, $\#C_{|[n]}(t)=\hat{Z}_t(n)$ and
\begin{equation}
  \label{equivalencedCC}
  n\overset{C(t)}{\sim} m \Longleftrightarrow \hat{Z}_t(n)=\hat{Z}_t(m).
\end{equation}
For any $\ell \in [|2,n|]$, the process $(\#C_{|[n]}(t),t\geq 0)$ jumps from $\ell$ to $\ell-k+1$ at rate $(\ell-k)\mu(k)+\bar{\mu}(k)$ and is absorbed at $1$.
\end{proposition}

\begin{proof}
By definition
\[
  C_{|[n]}(t)=([|1,Z_{-t}(1)|],[|Z_{-t}(1)+1,Z_{-t}(2)|],...,[|Z_{-t}(a-1)+1,n|])
\]
with $a=\#C_{|[n]}(t)=\min\{k\in \mathbb{N}: Z_{-t}(k)\geq n\}=:\hat{Z}_t(n)$. Consider now an integer $m\leq n$. If $\hat{Z}_t(m)=a$ then $m\in [|Z_{-t}(a-1)+1,n|]=C_{a}(t)\cap [n]$ and $m\overset{C(t)}{\sim} n$. The rates of jumps in $(\#C_{|[n]}(t),t\geq 0)$ are readily obtained by Lemma \ref{ratesofrestriction}.
\end{proof}

\begin{remark} Consecutive coalescents can be defined for a measure $\mu$ with a mass at $0$ from the relation \eqref{equivalencedCC}. However the process $(C(t),t\geq 0)$ in this case is inhomogeneous in time. We mention that the process $(\hat{Z}_t(n),t\geq 0)$ is studied by Li et al. in \cite{MR2409319}.
\end{remark}

\subsection{Consecutive coalescents in CSBPs through Poisson sampling}

We now explain how consecutive coalescents arise in the study of the backward genealogy of CSBPs. Loosely speaking, exchangeable bridges in the theory of exchangeable coalescents, \cite{MR1990057}, are replaced by subordinators and the sequence of uniform random variables by the arrival times of a Poisson process with intensity $\lambda$. The following typical random consecutive partitions will play a similar role as paintboxes for exchangeable coalescents.
\begin{definition}
\label{Poisson box}
We call $(\lambda,\phi)$-Poisson box a random consecutive partition $C$ obtained by setting
\[i\overset{C}{\sim} j \Longleftrightarrow X^{-1}(J_i)=X^{-1}(J_i),\]
where $X$ is a subordinator with Laplace exponent $\phi$ and $(J_j, j \geq 1)$ are the ranked atoms of an independent Poisson process with intensity $\lambda$.
\end{definition}

\begin{figure}[ht]
\centering
\begin{tikzpicture}
  \draw[->] (0,-0.5) -- (0,5) node[left] {$X(x)$};
  \draw[->] (-0.5,0) -- (6.5,0) node[below] {$x$};
  \draw (0,0) -- (2,0.5);
  \draw [densely dashed] (2,0.5) -- (2,1.25);
  \draw (2,1.25) -- (3,1.5);
  \draw [densely dashed] (3,1.5) -- (3,3);
  \draw (3,3) -- (4.5,3.375);
  \draw [densely dashed] (4.5,3.375) -- (4.5,3.875);
  \draw (4.5,3.875) -- (6.5,4.375);
  
  \draw [->] (7.3,0.25) node[right] {$J_1$} -- (7,0.25);
  \draw [->] (1,-0.3) node[below] {$J'_1$} -- (1,0);
  \draw [dotted] (7,0.25) -- (1,0.25) -- (1,0);

  \draw [->] (7.3,0.75) node[right] {$J_2$} -- (7,0.75);
  \draw [->] (2,-0.3) node[below] {$J'_2$} -- (2,0);
  \draw [dotted] (7,0.75) -- (2,0.75) -- (2,0);

  \draw [->] (7.3,1.7) node[right] {$J_3$} -- (7,1.7);
  \draw [->] (7.3,2) node[right] {$J_4$} -- (7,2);
  \draw [->] (7.3,2.8) node[right] {$J_5$} -- (7,2.8);
  \draw [->] (3,-0.3) node[below] {$J'_3$} -- (3,0);
  \draw [dotted] (7,1.7) -- (3,1.7);
  \draw [dotted] (7,2) -- (3,2);
  \draw [dotted] (7,2.8) -- (3,2.8) -- (3,0);

  \draw [->] (7.3,3.8) node[right] {$J_6$} -- (7,3.8);
  \draw [->] (4.5,-0.3) node[below] {$J'_4$} -- (4.5,0);
  \draw [dotted] (7,3.8) -- (4.5,3.8) -- (4.5,0);

  \draw [->] (7.3,4.125) node[right] {$J_7$} -- (7,4.125);
  \draw [->] (5.5,-0.3) node[below] {$J'_5$} -- (5.5,0);
  \draw [dotted] (7,4.125) -- (5.5,4.125) -- (5.5,0);
  \end{tikzpicture}
\caption{Construction of a random Poisson-box partition $C$, with subordinator $X$ and arrival times $(J_j, j \geq 1)$, satisfying $C_{|[7]}  = \{ \{1\},\{2\}, \{3,4,5\}, \{6\}, \{7\}\}$}
\end{figure}
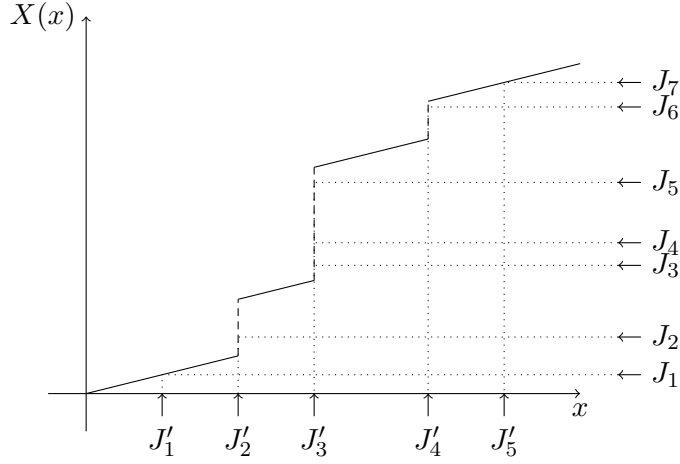
The $(\lambda,\phi)$-Poisson boxes will occur as typical random partitions in genealogical trees of CSBPs. More precisely, in the coalescent process describing the genealogy of individuals sampled according to a Poisson point process, the partitions will be distributed as $(\lambda,\phi)$-Poisson boxes. The following Lemma is proved in Appendix~\ref{sec:keyPPA}, and can be thought of as a revisiting of Pitman's discretization of subordinators \cite{MR1466546}.
\begin{lemma}\label{Poissonboxlaw}
Consider a subordinator $X$ with Laplace exponent 
\[
  \phi : \mu \mapsto d \mu + \int_{0}^{\infty}\left(1 - e^{-\mu x}\right) \ell(\ddr x)
\]
and $(J_k,k\geq 1)$ the arrival times of an independent Poisson process with intensity $\lambda$. Let $C$ be the $(\lambda,\phi)$-Poisson box constructed with $X$ and $(J_k,k\geq 1)$ and set 
for any $i\geq 1$, $J'_i:=X^{-1}(J_k)$ for $k\in C_i$. Then
\begin{enumerate}
  \item $C$ is a random consecutive partition with i.i.d blocks sizes and for any $k\geq 1$
\[
  \P(\#C_1=k)=\frac{1}{\phi(\lambda)}\int_{0}^{\infty}\frac{(\lambda x)^{k}}{k!}e^{-\lambda x} \ell(\ddr x) + d \ind{k=1}=(-1)^{k-1}\frac{\lambda^k}{k!}\frac{\phi^{(k)}(\lambda)}{\phi(\lambda)},
\]
i.e. $\E(s^{\#C_1}) = 1 - \frac{\phi(\lambda(1-s))}{\phi(\lambda)}$ for all $s \in [0,1]$.
  \item  The sequence $(J'_i,i\geq 1)$ are the arrival times of a Poisson process of intensity $\phi(\lambda)$. 
  \item $(J'_i,i\geq 1)$ and $C$ are independent.
\end{enumerate}
\end{lemma}
\begin{remark} We shall also consider killed subordinators with a Laplace exponent that satisfies $\phi(0)=\kappa>0$, or equivalently $\ell(\{\infty\})=\kappa$. The above Lemma can be extended to this case. See Corollary \ref{corkilled}. The associated $(\lambda, \phi)$-Poisson box has finitely many blocks. Formulas in (i) hold true and additionally  each block has probability $\phi(0)/\phi(\lambda)$ to be infinite. The sequence $(J'_i, 1\leq i\leq \#C)$ forms the first arrival times of a Poisson process with intensity $\phi(\lambda)$.
\end{remark}

We now construct consecutive coalescent processes related to the genealogy of the flow of subordinators $(X_{s,t}(x), s \leq t, x \geq 0)$. Denote by $(J^{\lambda}_i,i\geq 1)$ the sequence of arrival times of an independent Poisson process with intensity $\lambda$. For any $t\geq 0$, we define $C^{\lambda}(t)$ as
\begin{equation}\label{defconsec}
  i\overset{C^{\lambda}(t)}{\sim} j \text{ if and only if } \hat{X}_{t}(J^{\lambda}_i)=\hat{X}_{t}(J^{\lambda}_j). 
\end{equation}
The next theorem describes the law of the  process $(C^{\lambda}(t),t\geq 0)$.
\begin{theorem}
\label{markovcoal}
For any $\lambda>0$, the partition-valued $(C^{\lambda}(t),t\geq 0)$ is a consecutive coalescent. Its semigroup is
Feller and its one-dimensional marginal law is characterized by 
\begin{equation}\label{generatingfunctionCC}
  \mathbb{E}[z^{\#C^{\lambda}_1(t)}]=1-\frac{v_t(\lambda(1-z))}{v_t(\lambda)}\text{ for any } z\in [0,1].
\end{equation}
There is no simultaneous coalescences and for any $k\geq 2$, the rate at time $t$ at which $k$ given consecutive blocks coalesce is
\begin{equation}
  \label{coagrate}
  \mu^{\lambda}_t(k):=\frac{\sigma^2}{2}v_t(\lambda)1_{\{k=2\}}+v_t(\lambda)^{k-1}\int_{]0,\infty[}\frac{x^{k}}{k!}e^{-v_t(\lambda)x}\pi(\ddr x).
\end{equation}
\end{theorem}
In the supercritical case, by choosing for intensity $\lambda=\rho$, the process $(C^{\rho}(t),t\geq 0)$ becomes time-homogeneous. Corollary \ref{supercritical} is obtained by a direct application of Theorem \ref{markovcoal} since $v_t(\rho)=\rho$ for any $t\geq 0$.
\begin{corollary}\label{supercritical} Assume $\Psi$ supercritical and take $\lambda=\rho$. The coalescent process $(C^{\rho}(t),t\geq 0)$ is homogeneous in time and the coagulation rate of $k$ given consecutive blocks is
$$\mu^{\rho}(k):=\frac{\sigma^2}{2}\rho 1_{\{k=2\}}+\rho^{k-1}\int_{]0,\infty[}\frac{x^{k}}{k!}e^{-\rho x}\pi(\ddr x).$$
\end{corollary}
\begin{remark}  
Bertoin et al. \cite{MR2455180} have shown that in any flow of supercritical CSBPs one can embedd an immortal continuous-time Galton-Watson process counting the so-called prolific individuals, whose lines of descent are infinite. The prolific individuals are located in $\mathbb{R}_+$ as the arrival times of a Poisson process with intensity $\rho$ at any time. Moreover, this continuous-time Galton-Watson process has reproduction measure $\mu^{\rho}$. The consecutive coalescent $(C^{\rho}(t),t\geq 0)$ represents its genealogy backward in time.
\end{remark}
We prove Theorem \ref{markovcoal}. We stress that by definition, from \eqref{defconsec},  $C^{\lambda}(t)$ is a $(\lambda,v_t)$-Poisson box and $C^{\lambda}(0)=0_{[\infty]}$ since $\hat{X}_0=\mathrm{Id}$.
Our first lemma  proves that the partition-valued process $(C^{\lambda}(t),t\geq 0)$ is Markovian in its own filtration and is a consecutive coalescent (possibly inhomogeneous in time) in the sense of Definition \ref{consecutivecoal}. 
\begin{lemma}
\label{semigroup}
For any $s,t\geq 0$
\begin{equation}
  \label{coag}
  C^{\lambda}(t+s)=\Coag(C^{\lambda}(t),C^{\lambda}(t,t+s))
\end{equation}
where $C^{\lambda}(t,t+s)$ is a $(v_t(\lambda),v_s)$-Poisson box which is independent of $C^{\lambda}(t)$. 
\end{lemma}

\begin{proof}
For any $s,t\geq 0$ and all $l\geq 1$, set $J^{\lambda}_l(t):=\hat{X}_{t}(J^{\lambda}_i)$ for all $i \in C^{\lambda}_l(t)$. Let $C^{\lambda}(t,t+s)$ the random consecutive partition defined by \begin{center} $l\overset{C^{\lambda}(t,t+s)}{\sim} k$ if and only if $\hat{X}_{t,t+s}(J^{\lambda}_l(t))=\hat{X}_{t,t+s}(J^{\lambda}_k(t))$. \end{center}
Then by the key lemma \ref{Poissonboxlaw}-(ii), $C^{\lambda}(t,t+s)$ is a $(v_t(\lambda),v_s)$-Poisson box which is independent of $C^{\lambda}(t)$. Recall the cocycle property $\hat{X}_{t+s}=\hat{X}_{t,t+s}\circ \hat{X}_{t}$ (Theorem \ref{flowinverse}-i)). Let $i,j\in \mathbb{N}$. Set $k$ and $l$  such that $i\in C^{\lambda}_k(t)$ and $j\in C^{\lambda}_l(t)$. 
By the cocycle property, $\hat{X}_{t,t+s}(J^{\lambda}_k(t))= \hat{X}_{t,t+s}(J^{\lambda}_l(t))$ holds if and only if $i\overset{C^{\lambda}(t+s)}{\sim} j$ and (\ref{coag}) holds by definition of the operator $\Coag$, see \eqref{reversedbranching}.
\end{proof}
The generating function of the block's size at time $t$, given in \eqref{generatingfunctionCC}, is obtained by a direct application of the Key lemma \ref{Poissonboxlaw} since $C^{\lambda}(t)$ is a $(\lambda,v_t)$-Poisson box. We now show that the semigroup satisfies the Feller property.
\begin{lemma}\label{Fellerproperty}
The process $(C^{\lambda}(t),t\geq 0)$ is Feller and for any $t\geq 0$, $C^{\lambda}(t,t+s)\underset{s\rightarrow 0}{\longrightarrow} 0_{[\infty]}$ in probability.
\end{lemma}
\begin{proof}
The Feller property corresponds to the continuity of the map
\[C\in \mathcal{C}_{\infty} \mapsto P^{\lambda}_t\varphi(C):=\mathbb{E}[\varphi(\Coag(C, C^{\lambda}(t)))]\]
for any  continuous function $\varphi$ from $\mathcal{C}_{\infty}$ to $\mathbb{R}_+$. This is clear since $\Coag$ is Lipschitz continuous. We now show the weak continuity of the semigroup. By definition $J_i^{\lambda}(t)=\hat{X}_t(J_k)$ for any $k\in C^{\lambda}(t)$ and for any $i\neq j$, $J^{\lambda}_i(t)\neq   J^{\lambda}_j(t)$. By Lemma \ref{flowA}-(ii) and independence between $(J_{i},i\geq 1)$ and $\hat{X}$, we see that  $(J^{\lambda}_i(t),i\geq 1)$ is independent of $\hat{X}_{t,t+s}$. By Lemma \ref{flowA}-(iv), since $\hat{X}_{t,t+s}(x)\underset{s\rightarrow 0}{\longrightarrow} x$ uniformly on compact sets, in probability, then for any $n$,
\[\mathbb{P}(\forall i\neq j \in [n],\ \hat{X}_{t,t+s}(J_i^{\lambda}(t))\neq \hat{X}_{t,t+s}(J_j^{\lambda}(t)))\underset{s\rightarrow 0}{\longrightarrow} 1.\]
Therefore $\mathbb{P}(d(C^{\lambda}(t,t+s),0_{[\infty]})\leq 1/n)\underset{s\rightarrow 0}{\longrightarrow} 1.$
\end{proof}
We now seek for the coagulation rate \eqref{coagrate}. 

\begin{lemma}
\label{coagratecomputation}
For any $z\in (0,1)$, $$\frac{1}{s}\mathbb{E}[z^{\#C^{\lambda}_1(t,t+s)}-z]\underset{s\rightarrow 0}{\longrightarrow} \varphi^{\lambda}_t(z):=\frac{\Psi(v_t(\lambda)(1-z))-(1-z)\Psi(v_t(\lambda))}{v_{t}(\lambda)}.$$
\end{lemma}

\begin{proof}
Let $z \in (0,1)$, since by Lemma \ref{Poissonboxlaw}, the random variables $(J^{\lambda}_l(t),l\geq 1)$ are the arrival times of an independent  Poisson process with intensity $v_t(\lambda)$, then
\begin{equation}
  \label{generatingfunction}
  \mathbb{E}(z^{\#C^{\lambda}_1(t,t+s)})=1-\frac{v_{s}(v_t(\lambda)(1-z))}{v_{s}(v_{t}(\lambda))}.
\end{equation}
Thus 
\begin{align*}
  \frac{1}{s}\mathbb{E}[z^{\#C^{\lambda}_1(t,t+s)}-z]&=\frac{1}{s}\left[\frac{(1-z)v_{t+s}(\lambda)-v_s(v_t(\lambda)(1-z))}{v_{t+s}(\lambda)}\right]\\
  &=\frac{1}{s}\frac{(1-z)(v_{t+s}(\lambda)-v_t(\lambda))+v_t(\lambda)(1-z)-v_s(v_t(\lambda)(1-z))}{v_{t+s}(\lambda)}\\
  &\underset{s\rightarrow 0}{\longrightarrow}  \frac{\Psi(v_t(\lambda)(1-z))-(1-z)\Psi(v_t(\lambda))}{v_{t}(\lambda)}=: \varphi^{\lambda}(t).
\end{align*}
The latter convergence holds since $(v_t(\lambda),t\geq 0)$ solves \eqref{odev} and $v_{t+s}=v_{s}\circ v_t$.
\end{proof} 
By letting $\theta=v_t(\lambda)$ in the next technical lemma, we see that  for any $t\geq 0$, the measures $\mu^{\lambda}_t$ on $\mathbb{N}$ defined in \eqref{coagrate} have generating function $\varphi^{\lambda}_t$.
\begin{lemma}
\label{findingtherate}
Recall $\Psi$ in \eqref{LKpsi}. For any $z\in (0,1)$ and any $\theta\geq 0$, $$\frac{\Psi(\theta(1-z))-(1-z)\Psi(\theta)}{\theta}=\sum_{k=2}^{\infty}(z^{k}-z)p_\theta(k)$$ with $p_\theta(k)=\frac{\sigma^{2}}{2}\theta 1_{\{k=2\}}+\int_{0}^{\infty}\frac{\theta^{k-1}x^{k}}{k!}e^{-\theta x}\pi(\ddr x)$.
\end{lemma}
\begin{proof}
It is easy to see that $\Psi(\theta(1-z))=\Psi_\theta(-\theta z)+\Psi(\theta)$, where 
\[\Psi_\theta(u):=\Psi(u+\theta)-\Psi(\theta)=\Psi'(\theta)u+\frac{\sigma^2}{2}u^2+\int_0^\infty (e^{-ux}-1+ux)e^{-\theta x}\pi(\ddr x).\]
Then we have that 
\begin{align*}
\Psi_\theta(-\theta z)=&-\Psi'(\theta)\theta z +\frac{\sigma^2}{2}(\theta z)^2+\int_0^\infty (e^{\theta zx}-1-\theta zx)e^{-\theta x}\pi(\ddr x)\nonumber \\
=& -\Psi'(\theta)\theta z +\frac{\sigma^2}{2}(\theta z)^2+z\int_0^\infty \sum_{k=2}^\infty \frac{\theta^k x^k}{k!}e^{-\theta x}\pi(\ddr x)\\
&\qquad \qquad \qquad \qquad\qquad \qquad\qquad \quad+\theta\int_0^\infty \sum_{k=2}^\infty\frac{\theta^{k-1}x^k}{k!}(z^k-z)e^{-\theta x}\pi(\ddr x)\\
=& -\Psi'(\theta)\theta z +\frac{\sigma^2}{2}\theta^2z+z\int_0^\infty(e^{\theta x}-1-\theta x)e^{-\theta x}\pi(\ddr x)+\;\theta\sum_{k=2}^\infty (z^k-z)p_\theta(k)\\
=& -z\Psi(\theta)+\theta\sum_{k=2}^\infty (z^k-z)p_\theta(k).
\end{align*}
The last equality follows from the fact that
\[\Psi'(\theta)=-\beta+\sigma^2\theta+\int_0^1x(1-e^{-\theta x})\pi(\ddr x)-\int_1^\infty xe^{-\theta x}\pi(\ddr x).\]
Thus we have $\Psi(\theta(1-z))=(1-z)\Psi(\theta)+\theta\sum_{k=2}^\infty (z^k-z)p_\theta(k)$ for any $z\in (0,1)$.
\end{proof}
We now explain how coalescences take place in the process $(C^{\lambda}(t),t\geq 0)$. By construction, the laws of $(C^{\lambda}_{|[n]}(t),t\geq 0)$ for $n\geq 1$ are consistent and  as in Lemma \ref{ratesofrestriction} the family of jump rates  of $(C^{\lambda}_{|[n]}(t),t\geq 0)$ characterizes the law of $(C^{\lambda}(t),t\geq 0)$. The following lemma is obtained along the same lines as Lemma \ref{ratesofrestriction} but in an inhomogeneous time setting. 
\begin{lemma}
\label{inhomogeneousrate}
Let $n\geq 1$, the $n$-coalescent process $(C^{\lambda}_{|[n]}(t),t\geq 0)$ has jump rates characterized by $\mu^{\lambda}_t$ and the coalescence events are as follows. For all $t>0$, conditionally on $\#C_{|[n]}(t-)=m$, for any $j\leq m-1$,
consider the consecutive partitions of $[m]$ 
\begin{itemize}
  \item[-] $C^{j,k}_{\text{in}}:=(\{1\},...,\{j,...,j+k-1\},...,\{m\})$ for any $2\leq k\leq m-j$ and attach to each $C^{j}_{\text{in}}$ a random clock $\zeta^{j,k}_{\text{in}}$  with law
\[\mathbb{P}(\zeta^{j,k}_{\text{in}}>s)=\exp\left(-\int_{s}^{\infty}\mu^{\lambda}_r(k)\ddr r\right).\]
  \item[-] $C^{j}_{\text{out}}:=(\{1\},...,\{j,...,m\})$ and attach to each $C^{j}_{\text{out}}$ a random clock $\zeta^{j,k}_{\text{out}}$  with law
\[\mathbb{P}(\zeta^{j,k}_{\text{out}}>s)=\exp\left(-\int_{s}^{\infty}\bar{\mu}^{\lambda}_r(m-j+1)\ddr r\right).\]
\end{itemize}
Then the process jumps from the partition $C_{|[n]}(t-)$ to $\text{Coag}(C_{|[n]}(t-),D)$
with $D$ the partition in $\{C^{j,k}_{\mathrm{in}}, C^j_\mathrm{out}\}$ associated to the first random clock that rings.
\end{lemma}
\begin{proof}[Proof of Theorem \ref{discretizedcc}]
It follows directly by combination of Lemmas \ref{semigroup}--\ref{inhomogeneousrate}.
\end{proof}

We provide now some basic properties of the consecutive coalescent $(C^{\lambda}(t),t\geq 0)$. 
\begin{proposition}\label{limitbehavior} Fix $\lambda>0$. If $\Psi$ is critical or supercritical then $(C^{\lambda}(t),t\geq 0)$ converges almost-surely towards the partition $1_{\mathbb{N}}$. If $\Psi$ is subcritical, then the process $(C^{\lambda}(t),t\geq 0)$ converges almost-surely towards a partition $C^{\lambda}(\infty)$, whose law is characterized by $$\mathbb{E}[z^{\#C^{\lambda}_1(\infty)}]=1-e^{-\Psi'(0+)\int_{\lambda(1-z)}^{\lambda}\frac{\ddr u}{\Psi(u)}} \text{ for any } z\in (0,1).$$
In this case, individuals $(J^{\lambda}_1,J^{\lambda}_2,...)$ belong to families with i.i.d sizes distributed as $\#C^{\lambda}_1(\infty)$.
\end{proposition}
\begin{proof} Recall that for any $t\geq 0$, $\frac{\ddr }{\ddr u}v_t(u)=\frac{\Psi(v_t(u))}{\Psi(u)}$. Therefore
$$ \frac{v_t(\lambda(1-z)}{v_t(\lambda)}=\exp\left(\int_{\lambda}^{\lambda(1-z)}\frac{\ddr}{\ddr u}\log(v_t(u))\ddr u\right)=\exp\left(\int_{\lambda}^{\lambda(1-z)}\frac{\Psi(v_t(u))}{v_t(u)}\frac{\ddr u}{\Psi(u)}\right).$$
If $\Psi'(0+)<0$, then $\frac{\Psi(v_t(u))}{v_t(u)}\underset{t\rightarrow \infty}\longrightarrow \Psi(\rho)/\rho=0$ and by monotone convergence $\frac{v_t(\lambda(1-z))}{v_t(\lambda)}\underset{t\rightarrow \infty}{\longrightarrow} 1$. By Theorem \ref{markovcoal}, we have $$\mathbb{E}[z^{\#C^{\lambda}_1(t)}]=1-\frac{v_t(\lambda(1-z))}{v_t(\lambda)}\underset{t\rightarrow \infty}{\longrightarrow}0.$$ 
The process $(\#C^{\lambda}_1(t),t\geq 0)$ is non-decreasing and thus converges almost-surely in $\bar{\mathbb{N}}$.  Therefore $\#C^{\lambda}_1(t) \underset{t\rightarrow \infty}{\longrightarrow} \infty$ a.s. Recall that $(C^{\lambda}(t),t\geq 0)$ converges almost surely towards $\mathbbm{1}_{\mathbb{N}}$ if and only if $C^{\lambda}_{|[n]}(t)=\mathbbm{1}_{[n]}$ for large enough $t$. Since $\#C_1(t)\underset{t\rightarrow\infty}{\longrightarrow} \infty$, then $\mathbb{P}(\#C_1(t)\geq n)=\mathbb{P}(T_{1,n}\leq t)\longrightarrow 1$ with $T_{1,n}$ the coalescence time of the ancestral lineages of $J^{\lambda}_1$ and $J^{\lambda}_n$. If $\Psi'(0+)\geq 0$, then $\frac{\Psi(v_t(u))}{v_t(u)}\underset{t\rightarrow \infty}\longrightarrow \Psi'(0+)$ and by monotone convergence $\frac{v_t(\lambda(1-z)}{v_t(\lambda)}\underset{t\rightarrow \infty}{\longrightarrow} e^{-\Psi'(0)\int_{\lambda(1-z)}^{\lambda}\frac{\ddr u}{\Psi(u)}}$. By Theorem \ref{markovcoal}, we have for any $i\geq 1$, $$\mathbb{E}[z^{\#C^{\lambda}_i(t)}]=1-\frac{v_t(\lambda(1-z))}{v_t(\lambda)}\underset{t\rightarrow \infty}{\longrightarrow}1-e^{-\Psi'(0+)\int_{\lambda(1-z)}^{\lambda}\frac{\ddr u}{\Psi(u)}}.$$
By monotonicity, $\#C^{\lambda}_1(t) \underset{t\rightarrow \infty}{\longrightarrow} \#C^{\lambda}_1(\infty)$ a.s. Therefore for large enough time $t_1$, for $t\geq t_1$, $C^{\lambda}_1(t)=C^{\lambda}_1(\infty)$. Since there is no coalescence between  blocks $C^{\lambda}_1$ and $C^{\lambda}_2$ after time $t_1$, the process $(\#C^{\lambda}_2(t),t\geq t_1)$ is non-decreasing and converges almost-surely towards $\#C^{\lambda}_2(\infty)$. By induction, for any $n_0$, there exists $t_{n_0}$ such that for any $t\geq t_{n_0}$, $\#C^{\lambda}_i(t)=\#C^{\lambda}_i(\infty)$ for all $i\leq n_0$. Thus, for any $t\geq t_{n_0}$, $C^{\lambda}_i(t)\cap [n_0]=C^{\lambda}_i(\infty)\cap [n_0]$, and then $d(C^{\lambda}(t),C^{\lambda}(\infty))\leq \frac{1}{n_0}$.
\end{proof}

The following proposition is a direct consequence of the strong law of large numbers.
\begin{proposition}[Singletons] \label{singleton}
For any $t\geq 0$, and any $\lambda>0$ there are infinitely many singleton blocks at time $t$ and 
$$\frac{\#\{i\in [n]; \#C^{\lambda}_i(t)=1\}}{n}\underset{n\rightarrow \infty}{\longrightarrow} D_t^{\lambda} \text{ a.s.}$$
with $D_{t}^{\lambda}=\frac{\lambda}{\Psi(\lambda)}\frac{\Psi(v_t(\lambda))}{v_t(\lambda)}$. This represents the proportion of ancestors that have not been involved in coalescences by time $t$.
\end{proposition}
We have seen in Theorem \ref{Grey}-(iv) and Proposition \ref{boundaries} that when $\int_{0}\frac{\ddr x}{|\Psi(x)|}<\infty$, the CSBP explodes and $\infty$ is an entrance boundary of $(\hat{X}_t,t\geq 0)$. Recall that for any $t>0$, $\hat{X}_t(\infty)$ is the first individual from generation $t$ to have an infinite progeny at time $0$. 
\begin{proposition}[Coming down from infinity]\label{CDIlambda} For any $t>0$, $\#C^{\lambda}(t)<\infty$ a.s if and only if $\int_{0}\frac{\ddr x}{|\Psi(x)|}<\infty$. Moreover
\[v_t(0)\#C^{\lambda}(t)\underset{t\rightarrow 0}{\longrightarrow} \mathbbm{e}_{1/\lambda} \text{ in law}\]
where $\mathbbm{e}_{1/\lambda}$ is an exponential random variable with parameter $1/\lambda$.
\end{proposition}
\begin{proof} Recall that $v_t(0)>0$ if and only if $\int_{0}\frac{\ddr x}{|\Psi(x)|}<\infty$. By Theorem \ref{markovcoal}, for any $i\geq 1$, $\mathbb{P}(\#C^{\lambda}_i(t)=\infty)=\frac{v_t(0)}{v_t(\lambda)}$ and therefore the number of blocks $\#C^{\lambda}(t)$ is a geometric random variable with parameter $\frac{v_t(0)}{v_t(\lambda)}$. For any fixed $x>0$, one has
\[\mathbb{P}(v_t(0)\#C^{\lambda}(t)>x)=\left(1-\frac{v_t(0)}{v_t(\lambda)}\right)^{\floor{\frac{x}{v_t(0)}}}\underset{t\rightarrow 0}{\longrightarrow} e^{-\frac{x}{\lambda }}. \qedhere\]
\end{proof}

\subsection{Backward genealogy of the whole population}

In the previous section, we have defined some coalescent processes arising from sampling initial individuals along a Poisson process with an arbitrary intensity $\lambda$. The consecutive coalescents obtained by this procedure are only approximating the backward genealogy. They give the genealogy of a random sample of the population. The objective of this subsection is to observe that when the Grey condition holds, one can define a consecutive coalescent matching with the complete genealogy of the population from any positive time. In all this section, assume the Grey's condition
\[
  \int^{\infty}\frac{\ddr x}{\Psi(x)}<\infty.
\]

Heuristically, we make $\lambda \to \infty$ in Theorem \ref{markovcoal}, to study the genealogy of the whole population. The limiting process would indeed characterize the genealogy of the CSBP as in this case, an everywhere dense sub-population would be sampled and its genealogy given,  which is enough to deduce the genealogical relationship between any pair of individuals. However, this method cannot work directly as one would have jump rates that may explode.

Fix a time $s>0$. The subordinator $(X_{-s,0}(x),x\geq 0)$ is a compound Poisson process with L\'evy measure $\ell_{s}(\ddr x)$ independent of $(X_{-t,-s}(x),x\geq 0, t\geq s)$. Let $(J^{v_{s}(\infty)}_i,i\geq 1)$ be the jump times of $(X_{-s,0}(x),x\geq 0)$. They are arrival times of a Poisson process with intensity $v_{s}(\infty)=\ell_{s}((0,\infty))$, independent of $(\hat{X}_{s,t},t\geq s)$. Consider $(C(s,t),t\geq s)$ the partition-valued process defined by $$i\overset{C(s,t)}{\sim} j \text{ iff }  \hat{X}_{s, t}(J^{v_{s}(\infty)}_i)=\hat{X}_{s, t}(J^{v_{s}(\infty)}_j).$$
The process $(C(s,t),t>s)$  provides a dynamical description of the genealogy of initial individuals whose most recent common ancestors are found at time $s>0$. The following theorem is a direct application of Theorem \ref{markovcoal}.

\begin{theorem}
\label{discretizedcc}
For any $s>0$, the flow of random consecutive partitions $(C(u,t),t\geq u\geq s)$ satisfies for any $t\geq u\geq s$, 
\begin{equation}\label{flowCst}
C(s,t)=\Coag(C(s,u),C(u,t)) \text{ a.s}
\end{equation}
where $C(u,t)$ is a Poisson box with parameters $(v_{u}(\infty),v_{t-u})$ independent of $C(s,u)$. Moreover for any $i\geq 1$, $t\geq s$ and $z\in (0,1)$, 
$$\mathbb{E}[z^{\#C_i(s,t)}]=1-\frac{v_{t-s}(v_s(\infty)(1-z))}{v_{t}(\infty)}$$
and the consecutive coalescent $(C(s,t),t>s)$ has coagulation rates $(\mu^{\infty}_t, t>s)$ with
\begin{equation}
  \label{coagrateinfty}
  \mu^{\infty}_t(k):=\frac{\sigma^2}{2}v_t(\infty)1_{\{k=2\}}+v_t(\infty)^{k-1}\int_{]0,\infty[}\frac{x^{k}}{k!}e^{-v_t(\infty)x}\pi(\ddr x)\ \text{ for any }k\geq 2.
\end{equation}
\end{theorem}

We see in the next corollary that by reversing time in any block of the consecutive coalescent $(C(s,t), 0<s\leq t)$, one obtains an inhomogeneous continuous-time Galton-Watson process.  Fix an horizon time $T>0$ and consider the consecutive partitions $C(T-t,T)$ for any $t\in [0,T[$.

\begin{corollary} The processes $(Z_i^{T}(t),0\leq t<T):=(\#C_i(T-t,T),0\leq t<T)$ are i.i.d inhomogeneous continuous-time Galton-Watson processes. For any $z\in [0,1]$, and any $t\in [0,T[$
\begin{equation}\label{reducedtreelaw}
\mathbb{E}[z^{Z_i^{T}(t)}]=1-\frac{v_{t}(v_{T-t}(\infty)(1-z))}{v_{T}(\infty)}.
\end{equation}
Moreover, denoting by $\gamma^{T}_i$, the time of its first jump, one has for any $t\in [0,T[$
\[\mathbb{P}(\gamma^{T}_i>t)=\frac{\Psi(v_{T}(\infty))}{v_{T}(\infty)}\frac{v_{T-t}(\infty)}{\Psi(V_{T-t}(\infty))}.\]
\end{corollary}

\begin{proof}
The law of $Z^{T}_i(t)$ for fixed $t$ is obtained by a direct application of Theorem  \ref{discretizedcc}. Only remains to show the branching property. By \eqref{flowCst} for any $s$ and $t$ such that $0\leq t+s<T$
$$C(T-(t+s),T)=\Coag(C(T-(t+s),T-t),C(T-t,T))$$
which provides $C_i(T-(t+s),T)=\bigcup_{j\in C_i(T-t,T)}C_{j}(T-(t+s),T-t)$ and
the branching property.
\end{proof}

\begin{remark} The process $(Z_1^{T}(t), 0\leq t <T)$ corresponds to the \textit{reduced} Galton-Watson process obtained by Duquesne and Le Gall \cite[Theorem 2.7.1]{MR1954248} in the (sub)critical case. We refer also to Fekete et al. \cite{2017arXiv170203533F} for an approach with stochastic differential equations. In the supercritical case, since for any $t\geq 0$, $v_{T-t}(\infty)\underset{T\rightarrow \infty}{\longrightarrow} \rho$, 
we see in \eqref{reducedtreelaw} that $(Z^{T}_1(t), t\geq 0)$ converges, as $T$ goes to infinity, in the finite-dimensional sense, towards a Markov process $(Z^{\infty}(t),t\geq 0)$ whose semigroup satisfies
for any $z\in (0,1)$
\[\mathbb{E}[z^{Z^{\infty}(t)}]=1-\frac{v_{t}(\rho(1-z))}{\rho}.\]
Namely, $(Z^{\infty}(t),t\geq 0)$ is a continuous-time Galton-Watson process, homogeneous in time, with reproduction measure $\mu^{\rho}$. Heuristically, individuals from time $-t$ with descendants at time $T$ will correspond at the limit with prolific individuals. 
\end{remark}
The coalescent process $(C(s,t),t\geq s)$ only describes coalescence in families from time $s>0$. We define now a coalescent process from time $0$ by using the flow of subordinators. Denote by $\mathscr{C}_{\mathbb{R}_+}$ the space of partitions of $]0,\infty[$ into consecutive half-closed intervals. That is to say, partitions of the form $\mathscr{C}=(]0,x_1], ]x_1,x_2],...)$ for some non-decreasing sequence of positive real numbers $(x_i, i\geq 1)$. The space of consecutive partitions of $\N$, $(\mathcal{C}_\infty, \Coag)$, acts as follows on $\mathscr{C}_{\mathbb{R}_+}$:  for any $\mathscr{C}\in \mathscr{C}_{\mathbb{R}_+}$ and $C\in \mathcal{C}_\infty$, for any $i\geq 1$
\[\Coag(\mathscr{C},C)_i=\bigcup_{j\in C_i}\mathscr{C}_j\]
where $\mathscr{C}_j=]x_{j-1},x_{j}]$ and $x_0=0$.
The following theorem achieves one of our goals and has to be compared with our preliminary observation in Proposition \ref{warmup}. It describes completely the genealogy backwards in time as well as the sizes of asymptotic families.

\begin{theorem}\label{completeCC} 
Define the process $(\mathscr{C}(t),t>0)$ valued in $\mathcal{C}_{\mathbb{R}^{+}}$ as follows:
$$\mathscr{C}(t)=\left\{(X_{-t,0}(x-),X_{-t,0}(x)],x\in J_{-t}\right\}.$$
The process $(\mathscr{C}(t),t>0)$ is a time-inhomogeneous Markov process such that for any $t\geq s>0$, $$\mathscr{C}(t)=\Coag(\mathscr{C}(s), C(s,t)) \text{ a.s.}$$
In the critical or supercritical case, $\mathscr{C}(t)\underset{t\rightarrow \infty}{\longrightarrow} \mathbbm{1}_{[0,\infty]}$ a.s. In the subcritical case, $\mathscr{C}(t)\underset{t\rightarrow \infty}{\longrightarrow} \mathscr{C}(\infty)$ a.s and the length of a typical interval at the limit has for law the quasi-stationary distribution of the CSBP conditioned on the non-extinction:
$$\mathbb{E}[e^{-u|\mathscr{C}_1(\infty)|}]=1-\exp\left(-\Psi'(0+)\int_{u}^{\infty}\frac{\ddr v}{\Psi(v)}\right).$$
\end{theorem}

\begin{proof} Recall that $X_{-t,0}=X_{-s,0}\circ X_{-t,-s}$ and $\hat{X}_{t}=\hat{X}_{s,t}\circ \hat{X}_s$. This entails that for any $x\in J_{-t}=\{J_j^{v_t(\infty)},j\geq 1\}$,
\begin{equation}\label{compositionbytheright} (X_{-t,0}(x-),X_{-t,0}(x)]=\bigcup_{\substack{y\in (0,\infty)\\ \hat{X}_{s,t}(y)=x}}(X_{-s}(y-),X_{-s}(y)].
\end{equation}
For any $t>0$ and any $j\geq 1$, set $\mathscr{C}_j(t)=(x_{j-1}(t),x_{j}(t)]=(X_t(J_{j}^{v_{t}(\infty)}-),X_t(J_{j}^{v_{t}(\infty)})]$ with $x_{0}(t)=0$.  By definition of $C(s,t)$ and \eqref{compositionbytheright}, we have
\[\mathscr{C}_i(t)=\bigcup_{j\in \mathcal{C}_i(s,t)}\mathscr{C}_j(s).\]
Proposition \ref{limitbehavior} ensures that $(C(s,t),t\geq s)$ converges almost-surely as $t$ goes to $\infty$. This entails the almost-sure convergence of $\mathscr{C}(t),t>0)$. In the supercritical or critical case, $C(s,t)\underset{t\rightarrow \infty}{\longrightarrow} 1_{\mathbb{N}}$ and then $\mathscr{C}(t)\underset{t\rightarrow \infty}{\longrightarrow} \mathbbm{1}_{]0,\infty[}$, where $\mathbbm{1}_{]0,\infty[}$ denotes the partition of $]0,\infty[$ with only one block. In the subcritical case, Proposition \ref{warmup} entails that for all $i\geq 1$ and $u\geq 0$, \[\mathbb{E}[e^{-u|\mathscr{C}_i(\infty)|}]=1-\exp\left(-\Psi'(0+)\int_{u}^{\infty}\frac{\ddr v}{\Psi(v)}\right).\] Note moreover that
 $\mathscr{C}(\infty)=\Coag(\mathscr{C}(s),C(s,\infty))$ for any $s>0$.
\end{proof}
\begin{figure}[!ht]
\centering \noindent
\includegraphics[height=.22 \textheight]{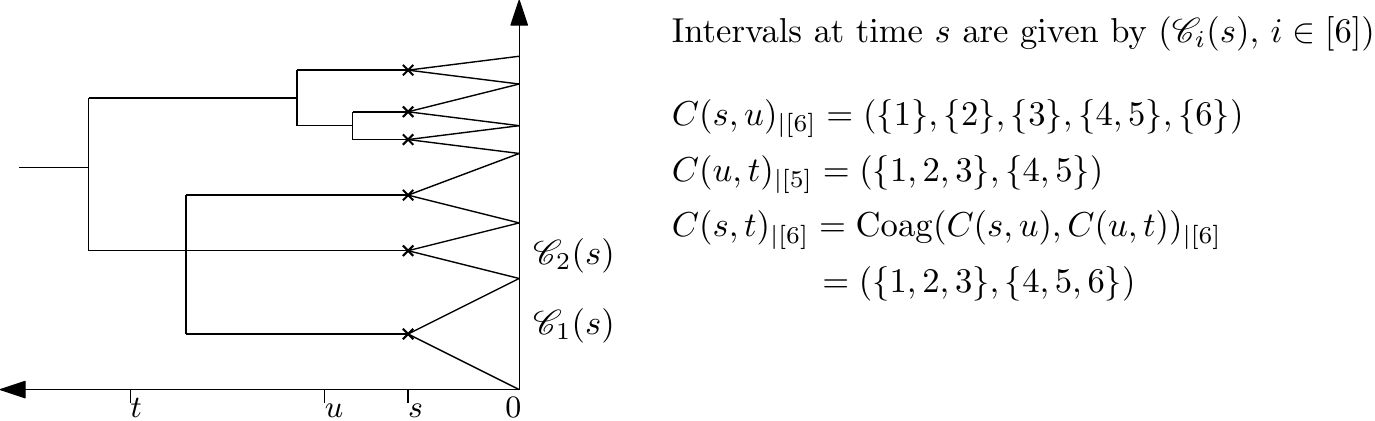}
\label{scoalescent}
\caption{Symbolic representation of the genealogy}
\end{figure}

\begin{remark}
Bertoin and Le Gall in \cite[Proposition 3]{MR2247827} have shown that in the critical case the L\'evy measures $(\ell_t,t>0)$ solve the following Smoluchowski equation
\begin{equation}\label{smolucho} \frac{\partial}{\partial t}<f,\ell_t>=v_t(\infty)\sum_{k=2}^{\infty}\mu^{\infty}_t(k)\int_{]0,\infty[^k}\big(f(x_1+...+x_k)-(f(x_1)+...+f(x_k))\big)\ell_t(\ddr x_1)...\ell_t(\ddr x_k)
\end{equation}
where $f$ is continuous function on $]0,\infty[$ with compact support and $<f,\ell_t>=\int_{]0,\infty[}f(x)\ell_t(\ddr x)$. The process $(\mathscr{C}(t),t>0)$ sheds some light on this deterministic equation since $\mu_t^{\infty}(k)$ is the rate in $(\mathscr{C}(t),t>0)$ at which $k$ given intervals coagulate and by the strong law of large numbers, \[\frac{1}{n}\sum_{i=1}^{n}\delta_{|\mathscr{C}_i(t)|}\underset{n\rightarrow \infty}{\longrightarrow} \frac{\ell_t(\ddr x)}{v_t(\infty)} \text{ a.s for any } t>0,\] where $|\mathscr{C}_i(t)|$ denotes the length of the $i^{\text{th}}$ interval in $\mathscr{C}(t)$. We refer the reader to Iyer et al. \cite{MR3313753}, \cite{iyer_leger_pego_2018} for recent works on Equation \ref{smolucho}.
\end{remark}
When the CSBP explodes, the individuals in the current generation have finitely many ancestors. The following proposition is the analogue of Proposition \ref{CDIlambda}.
\begin{proposition} Assume $\int_0\frac{\ddr u}{|\Psi(u)|}<\infty$, then the consecutive coalescent $(\mathscr{C}(t),t>0)$ comes down from infinity and \[\frac{v_t(0)}{v_t(\infty)}\#\mathscr{C}(t)\underset{t\rightarrow 0}{\longrightarrow} \mathbbm{e} \text{ in law }\]
where $\mathbbm{e}$ is a standard exponential random variable.
\end{proposition}
\begin{proof}
For any $t>0$, the lengths of the intervals in $\mathscr{C}(t)$  are i.i.d random variable with law $\frac{\ell_t(\ddr x)}{v_t(\infty)}$. Under the assumption, $\int_0\frac{\ddr u}{|\Psi(u)|}<\infty$, $\ell_t(\{\infty \})=v_t(0)>0$ and therefore the number of intervals in $\mathscr{C}(t)$ has a geometric law with parameter $\frac{v_t(0)}{v_t(\infty)}$. The convergence in law is proved by a similar calculation as in Proposition \ref{CDIlambda}.\end{proof}

We saw in Proposition~\ref{numberofblocks} that the number of blocks in $(C_{|[n]}(t),t\geq 0)$, the coalescent process associated to a continuous-time Galton-Watson process, corresponds to the inverse flow of random walks $(\hat{Z}_t(n),t\geq 0)$ at a fixed level $n$. Recall that in continuous-state space the process $(\hat{X}_t(x),t\geq 0)$  can then be interpreted as the size of the ancestral population whose descendants at time $0$ form a family of size $x$. The study is more involved than in the discrete setting and is the aim of the next section.

\section{A martingale problem for the inverse flow}
\label{martingaleproblem}

We investigate the infinitesimal dynamics of $(\hat{X}_t,t\geq 0)$ through its extended generator $\hat{\mathcal{L}}$. Recall that we write $\mathcal{L}$ the generator of the CSBP with mechanism $\Psi$. As we consider the flow of subordinators over $[0,\infty]$, it is natural to express  $\mathcal{L}$ as follows for all $\mathcal{C}^2$ bounded function $G$:
\begin{equation}
  \label{eqn:generatorCSBP}
  \mathcal{L}G(x)=\frac{\sigma^{2}}{2}xG''(x)+\beta xG'(x)+\int_{0}^{\infty}\pi(\ddr h)\int_{0}^{\infty}\ddr u \left(G(\Delta_{h,u}(x))-G(x)-h\mathbbm{1}_{\{u\leq x\}}G'(x)\mathbbm{1}_{\{h\leq 1\}}\right)
\end{equation}
with $\Delta_{h,u}(x):=x+h\mathbbm{1}_{\{x \geq u\}}$.

\begin{theorem}\label{extendedgen}
For any function $F$ in $C^{2}_b$ \footnote{the space of twice differentiable bounded functions over $(0,\infty)$ with bounded continuous derivatives}, set
\begin{multline*}
  \hat{\mathcal{L}}F(z)=\frac{\sigma^{2}}{2}zF''(z)+\left(\frac{\sigma^2}{2}-\beta z\right)F'(z) \\
  + \int_{0}^{\infty} \!\!\pi(\ddr h)\int_{0}^{\infty} \!\!\ddr u\left[F(\psi_{h,u}(z))-F(z)+h\mathbbm{1}_{\{h\leq 1\}}F'(z)\mathbbm{1}_{\{z>u\}}\right]
\end{multline*}
with $$\psi_{h,u}(z):=z\mathbbm{1}_{[0,u]}(z)+u\mathbbm{1}_{[u,u+h]}(z)+(z-h)\mathbbm{1}_{[u+h,\infty[}(z).$$
Then for any $y>0$, $(\hat{X}_t(y), t\geq 0)$ solves the following well-posed martingale problem
\[\displaystyle \mathrm{(MP)} \qquad \left( F(\hat{X}_t(y))-\int_{0}^{t}\hat{\mathcal{L}}F(\hat{X}_{s}(y))\ddr s, t \geq 0\right)\] is a martingale for any function $F$ in $\mathscr{D}:=\{F\in C^{2}_b;\ xF'(x), xF''(x)\underset{x\rightarrow \infty}{\longrightarrow} 0\}$. 
\end{theorem}
\begin{remark}
\label{rem:42}
Note that $\psi_{h,u} = \Delta_{h,u}^{-1}$ is the right-continuous inverse function of $\Delta_{h,u}$. For any $y \geq 0$, if individual $u$ makes at time $t$ a progeny of size $h$, then $\psi_{h,u}(y)$ at time $t-$ is the infinitesimal parent of individual $y$ at time $t$: if $y<u$, then $y$ has no parent but himself, if $y\in [u,u+h]$, the parent of $y$ is $\psi_{h,u}(y)=u$, if $y>u+h$  then its parent is $\psi_{h,u}(y)=y-h$. If $y_1\neq y_2$ then $\psi_{h,u}(y_1)=\psi_{h,u}(y_2)$ if and only if $y_1,y_2\in [u,u+h]$.
\end{remark}
\begin{figure}[!ht]
\centering
\includegraphics[height=.16 \textheight]{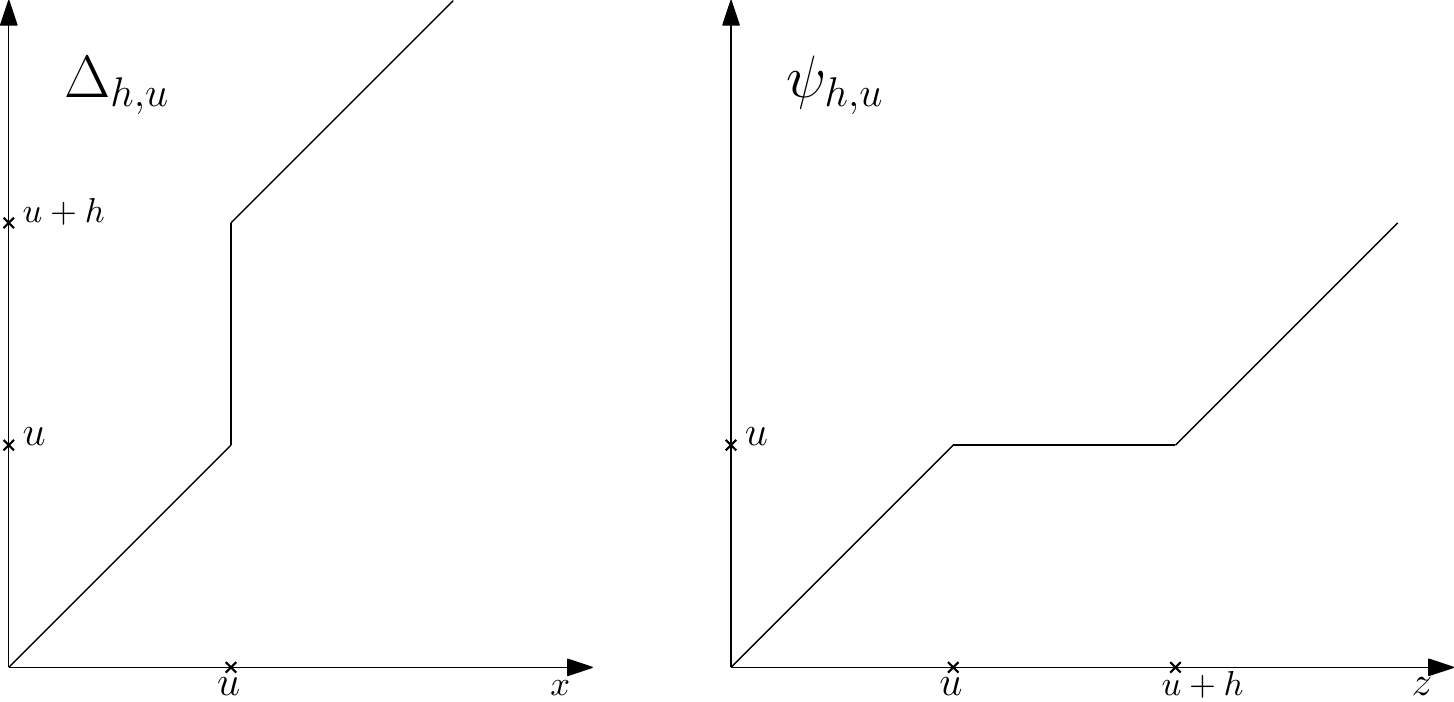}
\end{figure}

The proof of Theorem \ref{extendedgen} is divided in four lemmas.
\begin{lemma}\label{welldefined} The operator $\hat{\mathcal{L}}$ is well-defined on $C^{2}_b$.
\end{lemma}
\begin{proof}
Let $F\in C^2_b$, for any $y>0$ 
\begin{multline*}
  \left \lvert \int_{1}^{\infty}\pi(\ddr h)\int_{0}^{\infty}\ddr u (F(\psi_{h,u}(y))-F(y))\right\lvert\\
  =\left\lvert \int_{1}^{\infty}\pi(\ddr h)\int_{0}^{y}\ddr u (F(\psi_{h,u}(y)-F(y))\right \lvert \leq y\bar{\pi}(1)2||F||_\infty
\end{multline*}
where $\bar{\pi}(x):=\int_{x}^{\infty}\pi(\ddr u)$ for any $x>0$. For any $y>0$, $u>0$ and $h>0$, we have
$$\psi_{h,u}(y)-y=(u-y)\mathbbm{1}_{[u,u+h]}(y)-h\mathbbm{1}_{[u+h,\infty[}(y)=(u-y)\mathbbm{1}_{[y-h,y]}(u)-h\mathbbm{1}_{[0,y-h]}(u).$$
Therefore, if $h<1$, we have that
\begin{align*}
&\left\lvert F(\psi_{h,u}(y))-F(y)+hF'(y)\mathbbm{1}_{\{y>u\}}\right\lvert\\
\leq &\left\lvert F(\psi_{h,u}(y))-F(y) - (\psi_{h,u}(y)-y)F'(y) \right\lvert + \left\lvert (\psi_{h,u}(y)-y+h\mathbbm{1}_{\{y>u\}})F'(y)\right\lvert\\
\leq& \frac{(\psi_{h,u}(y)-y)^2}{2}||F''||_\infty+\left|\psi_{h,u}(y)-y+h\mathbbm{1}_{\{y>u\}}\right| ||F'||_\infty.
\end{align*}
On the one hand,
\begin{align*}
\psi_{h,u}(y)-y+h\mathbbm{1}_{\{y>u\}}&=(u-y)\mathbbm{1}_{[y-h,y]}(u)-h\mathbbm{1}_{[0,y-h]}(u)+h\mathbbm{1}_{[0,y]}(u)\\
&=(u+h-y)\mathbbm{1}_{[y-h,y]}(u)\geq 0
\end{align*}
and
$$\int_{0}^{\infty}(u+h-y)\mathbbm{1}_{[y-h,y]}(u) \ddr u =\left[\frac{u^{2}}{2}\right]_{y-h}^{y}\!+(h-y)h=\frac{h^2}{2}.$$
On the other hand
$$\frac{(\psi_{h,u}(y)-y)^2}{2}=\frac{1}{2}((u-y)^2\mathbbm{1}_{[y-h,y]}(u)+h^{2}\mathbbm{1}_{[0,y-h]}(u))$$
and
$$\int_{0}^{\infty}  \frac{(\psi_{h,u}(y)-y)^2}{2} \ddr u = \int_{0}^{\infty}\frac{1}{2}((u-y)^2\mathbbm{1}_{[y-h,y]}(u)+h^{2}\mathbbm{1}_{[0,y-h]}(u))\leq \frac{h^{3}}{6}+\frac{h^2}{2}y.$$
Thus, for any $h\geq 0$
\begin{equation}\label{bound} \int^{\infty}_0 \lvert F(\psi_{h,u}(y))-F(y)+hF'(y)\mathbbm{1}_{\{h<1\}} \mathbbm{1}_{\{y>u\}}\lvert \ddr u\leq C(h\wedge 1)^2||F''||_\infty+\frac{(h\wedge 1)^2}{2}||F'||_\infty y
\end{equation}
with a certain constant $C$ independent of $y$ and $h$. The integral with respect to $\pi(\ddr h)$ in $\hat{\mathcal{L}}F$ is therefore convergent and $\hat{\mathcal{L}}$ well-defined.
\end{proof}
We now follow the same method as Bertoin and Le Gall in \cite[Theorem 5]{LGB2} to show that $\hat{\mathcal{L}}$ is an extended generator, i.e. that $\left(F(\hat{X}_t(u))-\int_{0}^{t}\hat{\mathcal{L}}F(\hat{X}_s(u))\ddr s, t \geq 0\right)$ is a martingale for all $F$ in $\mathscr{D}$. Let $g$ be a continuous function over $[0,\infty[$ and $f$ a function in $C^{2}_0$. Set $G(t)=\int_{0}^{t}g(u)\ddr u $ and $F(t)=\int_{t}^{\infty} f(x)\ddr x$. Note that
$$\int_{0}^{\infty}\int_{0}^{\infty}g(u)f(x)\mathbbm{1}_{\{x\geq u\}}\ddr u\ddr x=\int_{0}^{\infty}g(u)F(u)\ddr u =\int_{0}^{\infty}f(x)G(x)\ddr x.$$
Moreover, one classically has that
\begin{equation*}
  \int_{0}^{\infty}f(x)\mathbb{P}(\hat{X}_s(u)< x)\ddr x=\mathbb{E}[F(\hat{X}_s(u))] \quad \text{ and } \quad \int_{0}^{\infty}g(u)\mathbb{P}(X_s(x)> u)\ddr u=\mathbb{E}[G(X_s(x))].
\end{equation*}
Recall that by (\ref{duality}), we have $\mathbb{P}(\hat{X}_s(u)< x)=\mathbb{P}(X_s(x)> u)$ for all $x, u \geq 0$. Then, integrating this equality with respect to $f(x)g(u)\ddr x \ddr u$
provides 
\begin{equation}\label{eq1}
\int_{0}^{\infty}\ddr u g(u)\mathbb{E}[F(\hat{X}_s(u))-F(u)]=\int_{0}^{\infty}\ddr x f(x)\mathbb{E}[G(X_s(x))-G(x)].
\end{equation}
Therefore, a first step in the search for $\hat{\mathcal{L}}$ is computing the right-hand side of \eqref{eq1}.
\begin{lemma}\label{computation} Let $\lambda>0$ and $g(x)=e^{-\lambda x}$ for any $x\in \mathbb{R}_+$ then for any $F\in \mathscr{D}$
\begin{equation} \int_{0}^{\infty}\! \!\ddr x f(x)\mathbb{E}[G(X_s(x))-G(x)]\\
=\int_{0}^{\infty}g(u)\ddr u \int_{0}^{s}\ddr t \mathbb{E}\left[\hat{\mathcal{L}}^{d}F(\hat{X}_t(u))+\hat{\mathcal{L}}^{c}F(\hat{X}_t(u))\right].
\end{equation}
\end{lemma}
\begin{proof} 
Assume that $g(v)=e^{-\lambda v}$ for some $\lambda>0$, the function $x\mapsto G(x)=\frac{1-e^{-\lambda x}}{\lambda}$ is in the domain of the generator $\mathcal{L}$ of the CSBP $(X_t,t\geq 0)$ (defined in \eqref{eqn:generatorCSBP}) and therefore
\begin{align*}
\mathbb{E}[G(X_s(x))-G(x)]&=\int_{0}^{s}\ddr t \mathcal{L}P_tG(x)=\int_{0}^{s}\ddr t \mathcal{L}^cP_tG(x)+\int_{0}^{s}\ddr t \mathcal{L}^d P_tG(x)
\end{align*}
where we write for all twice derivable function $H$,
\begin{align*}
  \mathcal{L}^c H(x) &= \frac{\sigma^{2}}{2}xH''(x)+\beta xH'(x)\\
  \mathcal{L}^d H(x) &= \int_{0}^{\infty}\pi(\ddr h)\int_{0}^{\infty}\ddr u \left(H(\Delta_{h,u}(x))-H(x)-h\mathbbm{1}_{\{u\leq x\}}H'(x)\mathbbm{1}_{\{h\leq 1\}}\right)
\end{align*}
which denote respectively the continuous and discontinuous parts of the generator $\mathcal{L}$. We start by studying the discontinuous part. For any $s\geq 0$, we can rewrite
\begin{multline}
  \label{eqn:intermediate}
  \int_{0}^{s}\ddr t \mathcal{L}^{d}P_tG(x)\\
  =\int_{0}^{s}\ddr t\int_{0}^{\infty}\pi(\ddr h)\int_{0}^{\infty}\ddr u\left[P_tG(x+h\mathbbm{1}_{\{u<x\}})-P_tG(x)-h(P_tG)'(x)\mathbbm{1}_{\{h<1\}}\mathbbm{1}_{\{u\leq x\}}\right].
\end{multline}

We first compute
\begin{align*}P_tG(x+h\mathbbm{1}_{\{u<x\}})-P_tG(x)&=\mathbb{E}[G(X_t(x+h\mathbbm{1}_{\{u<x\}})-G(X_t(x))]\\
&=\int_{0}^{\infty}g(v)\ddr v\left(\mathbb{P}(v\leq X_t(x+h\mathbbm{1}_{\{u<x\}}))-\mathbb{P}(v\leq X_t(x))\right).\\
&=\int_{0}^{\infty}g(v)\ddr v \left(\mathbb{P}(x+h\mathbbm{1}_{\{u\leq x\}}\geq \hat{X}_t(v))-\mathbb{P}(x\geq \hat{X}_t(v))
\right).
\end{align*}
By Lemma \ref{lem:factsInverse}(ii) and Remark \ref{rem:42}, one has $\Delta_{h,u}(x)>y$ if and only if $x>\psi_{h,u}(y)$, therefore
\begin{align*}
P_tG(x+h\mathbbm{1}_{\{u<x\}})-P_tG(x)&=\int_{0}^{\infty}g(v)\ddr v \left(\mathbb{P}(\psi_{h,u}(\hat{X}_t(v))\leq x)-\mathbb{P}(\hat{X}_t(v)\leq x)\right).
\end{align*}
Integrating with respect to $f(x)\ddr x$ we obtain
\begin{equation}
\label{eqn:firstResult}
\int_{0}^{\infty}\ddr x f(x) \left(P_tG(x+h\mathbbm{1}_{\{u<x\}})-P_tG(x)\right)=\int_{0}^{\infty}g(v)\ddr v \left(\mathbb{E}[F(\psi_{h,u}(\hat{X}_{t}(v)))]-F(\hat{X}_{t}(v))]\right).
\end{equation}

We now compute the compensated part of the discontinuous generator $\mathcal{L}^d$, by integration by part we have
\begin{equation}
\label{eqn:integrationByPart}
\int_{0}^{\infty}f(x)(P_tG)'(x)\mathbbm{1}_{\{u\leq x\}}\ddr x
  = f(\infty) P_tG(\infty) - f(u) P_tG(u) - \int_{u}^{\infty}f'(x)P_tG(x)\ddr x.
\end{equation}
Moreover, we observe that
\begin{align*}
  &\int_{u}^{\infty}f'(x)P_tG(x)\ddr x\\
 =&\int_{u}^{\infty}f'(x)\mathbb{E}\left[\int_{0}^{\infty}\mathbbm{1}_{\{X_{t}(x)>v\}}g(v)\ddr v\right]\ddr x =\int_{u}^{\infty}f'(x)\mathbb{E}\left[\int_{0}^{\infty}\mathbbm{1}_{\{x>\hat{X}_{t}(v)\}}g(v)\ddr v\right]\ddr x\\
  =&\mathbb{E}\left[\int_{0}^{\infty}g(v)\ddr v \int_{u}^{\infty}f'(x)\mathbbm{1}_{\{x>\hat{X}_{t}(v)\}}\ddr x \right]=\int_{0}^{\infty}g(v)\ddr v \left(f(\infty)-\mathbb{E}[f(\hat{X}_{t}(v)\vee u)]\right)\\
  =& G(\infty) f(\infty) - \int_0^\infty \mathbb{E}[f(\hat{X}_{t}(v)\vee u)] g(v)\ddr v.
\end{align*}
Therefore, as $P_tG(\infty) = G(\infty)$, \eqref{eqn:integrationByPart} becomes 
\begin{align*}
\int_{0}^{\infty}f(x)(P_tG)'(x)\mathbbm{1}_{\{u\leq x\}}\ddr x
&= -f(u) P_tG(u) + \int_0^\infty \mathbb{E}[f(\hat{X}_{t}(v)\vee u)] g(v)\ddr v\\
&=\int_{0}^{\infty} \left(\mathbb{E}[f(\hat{X}_{t}(v)\vee u)-f(u)\mathbb{P}(\hat{X}_t(v)\leq u)\right)g(v)\ddr v\\
&=\int_{0}^{\infty}\mathbb{E}[f(\hat{X}_{t}(v))\mathbbm{1}_{\{\hat{X}_t(v)>u\}}]g(v)\ddr v .
\end{align*}

Using the above result and \eqref{eqn:firstResult}, \eqref{eqn:intermediate} yields
\begin{align}
  &\int_{0}^{\infty}\ddr x f(x)\int_{0}^{s}\ddr t \mathcal{L}^{d}P_tG(x)\nonumber\\
  =&\int_{0}^{\infty}\! \!\ddr x f(x)\int_{0}^{s}\ddr t\int_{0}^{\infty}\! \!\pi(\ddr h)\int_{0}^{\infty}\! \!\ddr u\left[P_tG(x+h\mathbbm{1}_{\{u<x\}})-P_tG(x)-h(P_tG)'(x)\mathbbm{1}_{\{h<1\}}\mathbbm{1}_{\{u\leq x\}}\right]\nonumber\\
  =&\int_{0}^{s}\ddr t\int_{0}^{\infty}\!\!\!\!\pi(\ddr h)\int_{0}^{\infty}\ddr u\int_{0}^{\infty}\! \!\ddr x f(x)\left[P_tG(x+h\mathbbm{1}_{\{u<x\}})-P_tG(x)-h(P_tG)'(x)\mathbbm{1}_{\{h<1\}}\mathbbm{1}_{\{u\leq x\}}\right]\label{F1}\\
  =&\int_{0}^{s}\ddr t\int_{0}^{\infty}\!\!\!\!\pi(\ddr h)\int_{0}^{\infty}\ddr u \int_{0}^{\infty}g(v)\ddr v \Upsilon(t,h,u,v)\nonumber\\
  =&\int_{0}^{\infty}\!\!g(v)\ddr v\int_{0}^{s}\ddr t\int_{0}^{\infty}\! \!\pi(\ddr h)\int_{0}^{\infty}\! \!\ddr u  \Upsilon(t,h,u,v) \label{F2},
\end{align}
where 
\[\Upsilon:(h,u,t,v)\mapsto \left(\mathbb{E}[F(\psi_{h,u}(\hat{X}_{t}(v)))-F(\hat{X}_{t}(v))+hF'(\hat{X}_t(v))\mathbbm{1}_{\{\hat{X}_t(v)>u\}}\mathbbm{1}_{\{h\leq 1\}}]\right).\]
Above, \eqref{F1} and \eqref{F2} follow from applying Fubini's theorem, which we now justify. For any $t$ and $x$,
\[P_tG(x)=\frac{1-e^{-xv_t(\lambda)}}{\lambda} \text{ and } (P_tG)''(x)=-\frac{v_t(\lambda)^{2}}{\lambda}e^{-xv_t(\lambda)}.\]
Since $v_t(\lambda)e^{-xv_t(\lambda)}\leq \frac{1}{x}$ then $\sup_{[x,x+h]}|(P_tG)''(z)|\leq \frac{v_t(\lambda)}{\lambda x}$, and by Taylor's inequality
\begin{align*}
\lvert P_tG(x+h\mathbbm{1}_{\{u<x\}})-P_tG(x)-h(P_tG)'(x)\mathbbm{1}_{\{h<1\}}\mathbbm{1}_{\{u\leq x\}}\lvert&\leq \frac{v_t(\lambda)}{\lambda x}\frac{(h\wedge 1)^{2}}{2}\mathbbm{1}_{\{u\leq x\}}.
\end{align*}
Since $f$ is integrable then the upper bound is integrable with respect to $f(x)\ddr x \mathbbm{1}_{[0,s]}(t)\ddr t \pi(\ddr h)\ddr u$, which justifies the application of Fubini's theorem in \eqref{F1}.

We now explain why \eqref{F2} holds. Recall first that $f(z)=-F'(z)$. By Theorem \ref{flowinverse}, for any $q > 0$ we have $\mathbb{E}[\hat{X}_t(\mathbbm{e}_q)]=\frac{1}{v_t(q)}<\infty$ therefore $\mathbb{E}[\hat{X}_t(x)]<\infty$ for a.e. $x$. This, with the bound (\ref{bound}) allows us to conclude that $\Upsilon(h,u,t,v)$
is integrable with respect to $g(v)\ddr v \mathbbm{1}_{[0,s]}(t)\ddr t \pi(\ddr h) \ddr u$.

In a second time, we deal with the continuous part of the generator $\mathcal{L}^c$. Applying Fubini's theorem, one has
\begin{equation*}
\int_{0}^{\infty}\ddr x f(x)\int_{0}^{s}\mathcal{L}^{c}P_tG(x)\ddr t=\int_{0}^{s} \ddr t\int_{0}^{\infty}\ddr x f(x)\mathcal{L}^{c}P_tG(x). 
\end{equation*}
Set $h(x)=P_tG(x)$ and $\phi(x)=f'(x)\frac{\sigma^{2}}{2}x+f(x)\left(\frac{\sigma^{2}}{2}-\beta x\right)$. Since by assumption $F\in \mathscr{D}$, then $\underset{x\rightarrow \infty}{\lim} \phi(x)=\phi(\infty)=0$. We now compute $\int_{0}^{\infty}\ddr x f(x)\mathcal{L}^{c}P_tG(x)$. By two integration by parts
\begin{align*}
&\int_{0}^{\infty}\ddr x f(x)\mathcal{L}^{c}h(x)=\int_{0}^{\infty}\ddr x f(x)\left[\frac{\sigma^{2}}{2}xh''(x)-\beta x h'(x)\right]\\
&=\left[f(x)\frac{\sigma^{2}}{2}xh'(x)\right]_{0}^{\infty}-\int_{0}^{\infty}\ddr x \left[f'(x)\frac{\sigma^{2}}{2}x+f(x)\frac{\sigma^{2}}{2}\right]h'(x)+ \int_{0}^{\infty}\ddr x f(x)\beta xh'(x).\\
&=\left[f(x)\frac{\sigma^{2}}{2}xh'(x)\right]_{0}^{\infty}-\int_{0}^{\infty} \phi(x)h'(x)\ddr x\\
&=\left[f(x)\frac{\sigma^{2}}{2}xh'(x)\right]_{0}^{\infty}-\phi(\infty)h(\infty)+\int_{0}^{\infty} \phi'(x)h(x)\ddr x\\
&=-\phi(\infty)h(\infty)+\int_{0}^{\infty} \phi'(x)\mathbb{E}\left[\int_{0}^{\infty}g(u)\mathbbm{1}_{\{u\leq X_t(x)\}}\right]\ddr x \\
&=-\int_{0}^{\infty}\ddr u g(u)\int_{0}^{\infty}\ddr x \phi'(x)\mathbb{P}(\hat{X}_t(u)<x)\\
&=-\int_{0}^{\infty}\ddr u g(u)\mathbb{E}[\phi(\hat{X}_t(u))]=\int_{0}^{\infty}\ddr u g(u)\mathbb{E}[\hat{\mathcal{L}}^{c}F(\hat{X}_t(u))].
\end{align*}
We can now conclude as follows. One has
\begin{align}\label{laplacetrans}
\int_{0}^{\infty}\! \!\ddr u g(u)\mathbb{E}[F(\hat{X}_s(u))-F(u)]
&=\int_{0}^{\infty}\! \!\ddr x f(x)\mathbb{E}[G(X_s(x))-G(x)]  \\
&=\int_{0}^{\infty}g(v)\ddr v \int_{0}^{s}\ddr t \mathbb{E}\left[\hat{\mathcal{L}}^{d}F(\hat{X}_t(v))+\hat{\mathcal{L}}^{c}F(\hat{X}_t(v))\right]\qedhere. 
\end{align}
\end{proof}
\begin{lemma}\label{martingale} For any $F\in \mathscr{D}$ and any $y \geq 0$, $\displaystyle \left( F(\hat{X}_t(y))-\int_{0}^{t}\hat{\mathcal{L}}F(\hat{X}_{s}(y))\ddr s, t \geq 0\right)$ is a martingale. 
\end{lemma}
\begin{proof}
Recall that $g(v)=e^{-\lambda v}$. We will show that \eqref{laplacetrans} entails that for any $v$ and any $s$:
\begin{equation}\label{injectivity}
\mathbb{E}[F(\hat{X}_s(v))-F(v)]=\int_{0}^{s}\ddr t \mathbb{E}\left[\hat{\mathcal{L}}^{d}F(\hat{X}_t(v))+\hat{\mathcal{L}}^{c}F(\hat{X}_t(v))\right].
\end{equation}

From the Feller property of $(\hat{X}_t,t\geq 0)$ and the continuity of $\hat{\mathcal{L}}F$ for any function $F$ in $\mathscr{D}$, the map $v\mapsto \mathbb{E}[\hat{\mathcal{L}}F(\hat{X}_t(v))]$ is continuous. For any $a>0$, and any $v\leq a$, $\hat{X}_t(v)\leq \hat{X}_t(a)$ a.s. therefore using the bound (\ref{bound}), we see that the function defined on $[0,a]$ by
\[
  \Xi(v) = \int_{0}^{s}\ddr t\int_{0}^{\infty}\pi(\ddr h)\int_{0}^{\infty}\! \!\ddr u  \left(\mathbb{E}[F(\psi_{h,u}(\hat{X}_{t}(v)))-F(\hat{X}_{t}(v))+hF'(\hat{X}_t(v))\mathbbm{1}_{\{\hat{X}_t(v)>u\}}\mathbbm{1}_{\{h\leq 1\}}]\right)
\]
is continuous, and since $a$ is arbitrary, the mapping is continuous on $[0,\infty[$. This corresponds to the continuity of $v\mapsto \int_{0}^{s}\ddr t \mathbb{E}\left[\hat{\mathcal{L}}^{d}F(\hat{X}_t(v))\right]$. On the other hand, one can check the continuity of $v\mapsto \int_{0}^{s}\ddr t \mathbb{E}\left[\hat{\mathcal{L}}^{c}F(\hat{X}_t(v))\right]$ and by injectivity of the Laplace transform, \eqref{laplacetrans}  entails \eqref{injectivity}. This provides the martingale problem, as the following routine calculation shows. Let $t\geq 0$ and $s\geq 0$. Denote by $(\mathcal{F}_s,s\geq 0)$ the natural filtration associated to $(\hat{X}_s(x),s\geq 0,x\geq 0)$, 
\begin{align*}
&\mathbb{E}\left[F(\hat{X}_{t+s}(x))-\int_{0}^{t+s}\hat{\mathcal{L}}F(\hat{X}_u(x))\ddr u \ |\mathcal{F}_s\right]\\
&=\mathbb{E}\left[F(\hat{X}_{t+s}(x))-\int_{s}^{t+s}\hat{\mathcal{L}}F(\hat{X}_u(x))\ddr u \ |\mathcal{F}_s\right]-\int_{0}^{s}\hat{\mathcal{L}}F(\hat{X}_u(x))\ddr u\\
&=\mathbb{E}_{\hat{X}_s(x)}\left[F(\hat{X}_t)-\int_{0}^{t}\hat{\mathcal{L}}F(\hat{X}_u)\ddr u\right]-\int_{0}^{s}\hat{\mathcal{L}}F(\hat{X}_u(x))\ddr u\\
&=F(\hat{X}_{s}(x))-\int_{0}^{s}\hat{\mathcal{L}}F(\hat{X}_u(x))\ddr u. \qedhere
\end{align*}
\end{proof}
In the following Lemma, we rewrite the generator $\hat{\mathcal{L}}$ of the one-point motion in its Courr\`ege form.
We refer to Kolokoltsov \cite{MR2858558} for a general study of generators of stochastically monotone Markov processes.
\begin{lemma}\label{courregeform} For any $f\in C^{2}_b$, 
$$\hat{\mathcal{L}}f(z)=\frac{\sigma^{2}}{2}zf''(x)+\int_{0}^{z}\left[f(z-h)-f(z)+hf'(z)\right]\nu(z,\ddr h)+b(z)f'(z)$$ with $$\nu(z,\ddr h):=\mathbbm{1}_{\{h\leq z\}}\left((z-h)\pi(\ddr h)+\bar{\pi}(h)\ddr h\right)$$ and $$b(z):=\int_{0}^{\infty}h(z\mathbbm{1}_{\{h\leq 1\}}\pi(\ddr h)-\nu(z,\ddr h))-\beta z+\frac{\sigma^{2}}{2}.$$
Moreover, the martingale problem $\mathrm{(MP)}$ is well-posed.
\end{lemma}

\begin{remark}\label{finitemeandrift}
The jump measure $\nu(z,\ddr h)$ can be compared to the jumps rate of $(\hat{Z}_t(n),t\geq 0)$ obtained in Proposition \ref{numberofblocks}. Moreover, in the finite mean case, $\int^{\infty}_1h\pi(\ddr h)<\infty$, the drift $b$ can be rewritten as follows
\begin{align*}
b(z)&=\int_{0}^{\infty}h(z\pi(\ddr h)-\nu(z,\ddr h))+\Psi'(0+)z+\frac{\sigma^{2}}{2}\\
&=z\int_{z}^{\infty}\bar{\pi}(h)\ddr h +\int_{0}^{z}h\bar{\pi}(h)\ddr h+\Psi'(0+)z+\frac{\sigma^2}{2}.
\end{align*}
In particular, for any $z>0$, $b'(z)=\int_{z}^{\infty}\bar{\pi}(\ddr h)+ \Psi'(0+)$, $b''(z)=-\bar{\pi}(z)$ and $b$ is concave. 
\end{remark}

\begin{proof}
Recall $\psi_{h,u}(z):=z\mathbbm{1}_{\{z\leq u\}}+(z-h)\mathbbm{1}_{[u+h,\infty[}(z)+u\mathbbm{1}_{[u,u+h]}(z)$. Note that:
\begin{align*}
&\int_{0}^{\infty}\pi(\ddr h)\int_{0}^{\infty}\ddr u\left[f(\psi_{h,u}(z))-f(z)+h\mathbbm{1}_{\{h\leq 1\}}f'(z)\mathbbm{1}_{\{z>u\}}\right]\\
  &=\int_{0}^{\infty} \!\!\pi(\ddr h)\int_{0}^{z}\ddr u\left[(f(z-h)-f(z))\mathbbm{1}_{[u+h,\infty[}(z))+(f(u)-f(z))\mathbbm{1}_{[u,u+h]}(z)+h\mathbbm{1}_{\{h\leq 1\}}f'(z)\right].
  \end{align*}
Therefore, one has  
\begin{align*}
\hat{\mathcal{L}}f(z)&=\int_{0}^{z}\pi(\ddr h)\left[(z-h)[f(z-h)-f(z)]+\int_{0}^{z}\ddr u [f(u)-f(z)]\mathbbm{1}_{[u,u+h]}(z)+zhf'(z)\mathbbm{1}_{h<1}\right]\\
  &\qquad +\int_{z}^{\infty}\pi(\ddr h)\left[\int_{0}^{z}\ddr u\left[f(u)-f(z)\right]\mathbbm{1}_{[u,u+h]}(z)+hzf'(z)\mathbbm{1}_{h<1}\right]\\
  &=I+II.
\end{align*}

For the first integral $I$:
\begin{align*}
  I&=\int_{0}^{z}\pi(\ddr h)\left[(z-h)[f(z-h)-f(z)+hf'(z)\mathbbm{1}_{h<1}]+h^{2}f'(z)\mathbbm{1}_{h<1}\right]\\
  &\qquad +\int_{0}^{z}\pi(\ddr h)\int_{0}^{z}\ddr u \ (f(u)-f(z))\mathbbm{1}_{\{u>z-h\}},
\end{align*}
one has
\[\int_{0}^{z}\pi(\ddr h)\int_{0}^{z}\ddr u \ (f(u)-f(z))\mathbbm{1}_{\{u>z-h\}}=\int_{0}^{z}(f(z-h)-f(z))(\bar{\pi}(h)-\bar{\pi}(z))\ddr h.\]
Thus
\begin{align*}
  I&=\int_{0}^{z}[f(z-h)-f(z)+hf'(z)\mathbbm{1}_{\{h\leq 1\}}](z-h)\pi(\ddr h)+\int_{0}^{z}(f(z-h)-f(z))(\bar{\pi}(h)-\bar{\pi}(z))\ddr h\\
  &\qquad\qquad\qquad \qquad +\int_{0}^{z}h^{2}\mathbbm{1}_{\{h\leq 1\}}f'(z)\pi(\ddr h)\\
  &=\int_{0}^{z}\left[f(z-h)-f(z)+hf'(z)\mathbbm{1}_{\{h\leq 1\}}\right]((z-h)\pi(\ddr h)+(\bar{\pi}(h)-\bar{\pi}(z))\ddr h)\\
  &\qquad\qquad\qquad \qquad+\int_{0}^{z}h^{2}\mathbbm{1}_{\{h\leq 1\}}\pi(\ddr h)f'(z)-\int_{0}^{z}h\mathbbm{1}_{h<1}(\bar{\pi}(h)-\bar{\pi}(z))\ddr h f'(z).
\end{align*}

For the second integral $II$:
\begin{align*}
  II&=\int_{z}^{\infty}\pi(\ddr h)\int_{0}^{z}\ddr u(f(u)-f(z))\mathbbm{1}_{[u,u+h]}(z)+\int_{z}^{\infty}\pi(\ddr h) zhf'(z)\mathbbm{1}_{h<1}\\
  &=\bar{\pi}(z)\int_{0}^{z}\ddr v [f(z-v)-f(z)]+zf'(z)\int_{z}^{\infty}\pi(\ddr h)h\mathbbm{1}_{h<1}.
\end{align*}
Summing both expressions, $I+II$ equals to:
\begin{eqnarray*}
  &&\int_{0}^{z}\left[f(z-h)-f(z)+hf'(z)\mathbbm{1}_{\{h\leq 1\}}\right]((z-h)\pi(\ddr h)+(\bar{\pi}(h)-\bar{\pi}(z)+\bar{\pi}(z))\ddr h)\\
  &&\quad+\bigg(\int_{0}^{z}h^{2}\mathbbm{1}_{\{h\leq 1\}}\pi(\ddr h)
  -\int_{0}^{z}h\mathbbm{1}_{h<1}(\bar{\pi}(h)-\bar{\pi}(z))\ddr h\\
  &&\quad\qquad +z\int_{z}^{\infty}h\mathbbm{1}_{h\leq 1}\pi(\ddr h)-\bar{\pi}(z)\int_{0}^{\infty}h\mathbbm{1}_{\{h<1\}}\ddr h \bigg)f'(z).\\
\end{eqnarray*}
Therefore
  \[I+II=\int_{0}^{z}\left[f(z-h)-f(z)+hf'(z)\mathbbm{1}_{\{h\leq 1\}}\right]\nu(z,\ddr h)+b_1(z)f'(z)\]
with
\begin{align*}
b_1(z)&:=\int_{0}^{z}\left(h^{2}\mathbbm{1}_{h\leq 1}\pi(\ddr h)-h\bar{\pi}(h)\mathbbm{1}_{h\leq 1} \ddr h
\right)+z\int_{z}^{\infty}\pi(\ddr h)h\mathbbm{1}_{h<1}\\
&=\int_{0}^{\infty}h\mathbbm{1}_{\{h\leq 1\}}(z\pi(\ddr h)-\nu(z,\ddr h)),
\end{align*}
and we obtain
\begin{align*}
I+II&=\int_{0}^{z}\left[f(z-h)-f(z)+hf'(z)\right]\nu(z,\ddr h)-\int_{0}^{z}\mathbbm{1}_{\{1<h\leq z\}}\nu(z,\ddr h)f'(z)+b_1(z)f'(z)\\
&=\int_{0}^{z}\left[f(z-h)-f(z)+hf'(z)\right]\nu(z,\ddr h)+b(z)f'(z). 
\end{align*}
We now verify uniqueness of the solution to $\mathrm{(MP)}$ by applying Theorem 5.1 of Kolokoltsov \cite{MR2858558}. Assumptions (i) and (ii) of the theorem can be readily checked. The third assumption (iii) is that for any $z>1$, $b(z)\leq c(1+z)$ for some $c>0$. Let $z>1$, one has
\begin{align*}
b(z)&=\int_{0}^{1}h(z\pi(\ddr h)-\nu(z,\ddr h))-\int_{1}^{z}h\nu(z,\ddr h)-\beta z+\frac{\sigma^{2}}{2}\\
&=\int_{0}^{1}(h^2\pi(\ddr h)+h\bar{\pi}(h)\ddr h)-z\int_{1}^{z}h\pi(\ddr h)+\int_{1}^{z}h^2\pi(\ddr h)-\int_{1}^{z}h\bar{\pi}(h)\ddr h-\beta z+\frac{\sigma^{2}}{2}\\
&\leq \int_{0}^{1}(h^2\pi(\ddr h)+h\bar{\pi}(h)\ddr h)+\frac{\sigma^{2}}{2}-\beta z\leq c(1+z)
\end{align*}
where for the first inequality we use the fact that  $-z\int_{1}^{z}h\pi(\ddr h)+\int_{1}^{z}h^2\pi(\ddr h)\leq 0$ and we choose a large enough $c$ for the second inequality.
\end{proof}

\begin{proof}[Proof of Theorem \ref{extendedgen}.]
It follows directly by combination of Lemmas \ref{martingale} and \ref{courregeform}.
\end{proof}

\begin{remark}
Similar computations to the ones made in the proof of Lemma \ref{computation} can be done for the $p$-point motion $(\hat{X}_t(y_1),...,\hat{X}_t(y_p))$ from the duality relation
\[\mathbb{P}(\hat{X}_t(y_1)<x_1,..., \hat{X}_t(y_p)<x_p)=\mathbb{P}(X_t(x_1)>y_1,...,X_t(x_p)>y_p)\]  
Consider for example the case $\sigma=\beta=0$. For any function $F$ in $C^{2}(\mathcal{D}_p)$, where we denote by $\mathcal{D}_p:=\{\mathrm{y}:=(y_1,...,y_p)\in ]0,\infty[^p, y_1\leq y_2...\leq y_p \}$, we set  
\begin{align*}
\hat{\mathcal{L}}F(\mathrm{y})
&=\int_{0}^{\infty}\pi(\ddr h)\int_{0}^{\infty}\ddr u\left[F(\psi_{h,u}(\mathrm{y}))-F(\mathrm{y})+h\mathbbm{1}_{\{h\leq 1\}}\sum_{i=1}^{p}\frac{\partial}{\partial y_i} F(\mathrm{y})\mathbbm{1}_{\{y_i>u\}}\right]
\end{align*}
with $\psi_{h,u}(\mathrm{y})=(\psi_{h,u}(y_1),...,\psi_{h,u}(y_p))$.
Then 
$$F(\hat{X}_t(\mathrm{y}))-\int_{0}^{t}\hat{\mathcal{L}}F(\hat{X}_s(\mathrm{y}))\ddr s$$
is a martingale, where
$\hat{X}_t(\mathrm{y})=(\hat{X}_t(y_1),...,\hat{X}_t(y_p))$.
\end{remark}

\section{Examples}
\label{sec:examples}

In this section, we apply the results obtained in the previous ones to the two following important examples: the stable CSBP and the Neveu CSBP. These CSBPs arise in many different frameworks and are known for instance to be closely related to the class of exchangeable coalescents called Beta-coalescents.

\subsection{Feller and stable CSBPs}

A stable CSBP is a continuous-state branching process with branching mechanism given by $\Psi : u \mapsto c_\alpha u^\alpha - \beta u$, for some $\alpha \in (1,2]$, $c_\alpha > 0$ and $\beta \in \R$. Note in particular that the Feller flow, whose inverse flow was studied in details in Section \ref{fellerflow} is a stable CSBP with $\alpha=2$. As a direct application of Theorem \ref{completeCC}, we obtain that the Markovian coalescent $(\mathscr{C}(t),t>0)$ associated to the Feller flow has coagulation rates
\[
  \mu^{\infty}_t=\frac{2\beta}{2(1-e^{-\beta t})}\delta_2(k).
\]
In particular, in the subcritical case ($\beta < 0$), $\mathscr{C}(t)$ converges almost-surely as $t \to \infty$ towards intervals with i.i.d. exponentially distributed lengths with parameter $\hat{\rho} = 2 \beta/c_2$. This corresponds to the partition of $\R_+$ into random intervals $(]0,x^{\star}_1[,]x^{\star}_1,x^{\star}_2[,...)$ corresponding to different ancestors at time $-\infty$ found in Section \ref{fellerflow}. 

We now assume that $\Psi(u)=c_{\alpha}u^{\alpha}-\beta u$ for some $\alpha\in (1,2)$, with $c_\alpha:=\frac{\Gamma(2-\alpha)}{\alpha(\alpha-1)}$ (which corresponds to a simple time dilatation). By assumption $\alpha>1$ and Grey's condition holds $\int^{\infty}\frac{\ddr u}{\Psi(u)}<\infty$. Solving the differential equation \eqref{odev} satisfied by $v_t(\lambda)$, we have in particular that
\[
  v_t(\infty) = \begin{cases}
    c_\alpha^{-\frac{1}{\alpha-1}}\left(\frac{1-e^{-(\alpha-1)\beta t}}{\beta}\right)^{-\frac{1}{\alpha-1}} &\text{if } \beta \neq 0\\
    \left(\Gamma(2-\alpha)/\alpha\right)^{-\frac{1}{\alpha-1}}t^{-\frac{1}{\alpha-1}} &\text{if } \beta = 0.
  \end{cases}
\]
For the stable CSBP, the coagulation rates of its associated Markovian coalescent $(\mathscr{C}(t),t>0)$ are given by 
\[  \mu_t^{\infty}(k):=v_t(\infty)\mu_\alpha(k),  \] 
with $\mu_\alpha(k):=\frac{\Gamma(k-\alpha)}{k!}$. The normalized associated probability measure is 
\begin{equation}
  \label{stablerate}
  \frac{\mu^{\infty}_t(k)}{\mu^{\infty}_t(\mathbb{N})}=\frac{\alpha(2-\alpha)...(k-1-\alpha)}{k!}
\end{equation}
which is time-independent. This probability distribution corresponds to the reproduction measure of prolific individuals in supercritical stable CSBP, see Example 3 in \cite{MR2455180}. It also appears in the study of reduced $\alpha$-stable trees and $\text{Beta}(2-\alpha, \alpha)$-exchangeable coalescents, see respectively Duquesne and Le Gall \cite[page 74]{MR1954248} and Berestycki et al. \cite[Section 5]{MR2349577}.

The inhomogeneous consecutive coalescent $(C(s,t),t>s)$ representing the genealogy of any stable CSBP from time $s>0$ is obtained by a deterministic time-change of the homogeneous consecutive coalescent $(\check{C}(t),t\geq 0)$ with coagulation rates $\mu_\alpha$ via the transformation: for any $t\geq s$,
\[
  C(s,t)=\check{C}\left(\int_{s}^{t}v_u(\infty)\ddr u\right).
\]
Note that $\int_{s}^{\infty}v_u(\infty)\ddr u=\int_{0}^{v_s(\infty)}\frac{z}{\Psi(z)}\ddr z$ which is finite if and only if $\Psi$ is subcritical ($\beta<0$).
According to Theorem \ref{completeCC}, in the subcritical case $(\mathscr{C}(t),t>0)$ converges almost-surely as $t \to \infty$ towards a partition of intervals with i.i.d. lengths with law $\nu_\alpha$ such  that 
\begin{equation}
  \label{qsd}
  \int_{0}^{\infty}e^{-uz}\nu_\alpha(\ddr z)=1-e^{\beta \int_{u}^{\infty}\frac{\ddr x}{c_\alpha x^{\alpha}-\beta x}} \text{ for any } u\geq 0.
\end{equation}

We now turn to the martingale problem satisfied by the inverse flow of the stable CSBP. One easily computes the drift and the jump measure from Remark \ref{finitemeandrift}.

\begin{proposition}
The process $(\hat{X}_t,t\geq 0)$ is characterized by the martingale problem associated to $\hat{\mathcal{L}}$, acting on $C^{2}_b$, given in Lemma \ref{courregeform}, with 
$$\nu(z,\ddr h)=\left((z-h)h^{-1-\alpha}+\frac{h^{-\alpha}}{\alpha}\right)\mathbbm{1}_{[0,z]}(h)\ddr h \text{ and } b(z)=\frac{1}{\alpha(\alpha-1)(2-\alpha)}z^{2-\alpha}-\beta z.$$
\end{proposition}

In the critical case, one can identify the law of $\hat{X}$ through some random-time change.

\begin{proposition}
If $\beta=0$, the process $(\hat{X}_t,t\geq 0)$ is a positive self-similar Markov process with index $a:=\alpha-1$. Namely for any $k>0$ and any $y>0$,
\[
  (k\hat{X}_{k^{-a}t}(y),t\geq 0)\overset{\mathcal{L}}{=}(\hat{X}_{t}(ky),t\geq 0).
\]
Moreover,
\[
  \log \hat{X}_t(x)=L_{\varphi^{x}(t)}
\]
where $\varphi^{x}(t):=\inf\{s>0; \int_{0}^{s}e^{(\alpha-1)L_u}\ddr u>t\}$ and $L$ is a spectrally negative L\'evy process started from $\log x$  with Laplace exponent
\[
  \kappa(q)=-d_{\alpha}q+\int_{-\infty}^{0}\left(e^{q z}-1+q(1-e^{z})\right)\nu_\alpha(\ddr z)
\]
with $\nu_\alpha(\ddr z)=\left(e^{z}(1-e^{z})^{-1-\alpha}+\frac{1}{\alpha}(1-e^{z})^{-\alpha}\right)e^{z}\ddr z$ and $d_{\alpha}=\frac{1}{\alpha(\alpha-1)(2-\alpha)}$.
\end{proposition}

\begin{proof}
Recall that the critical CSBP $(X_t,t\geq 0)$ is itself selfsimilar with index $a:=\alpha-1$. See for instance Kyprianou and Pardo \cite{MR2484167}. For any $k>0$ and any $x>0$, $kX_{k^{-a}t}(x)\overset{\mathcal{L}}{=}X_{t}(kx)$. Thus for any $y>0$
\[
  \mathbb{P}(k\hat{X}_{k^{-a}t}(y)\leq x)=\mathbb{P}({X}_{k^{-a}t}(x/k)\geq y)=\mathbb{P}(k^{-1}X_{t}(x)>y)=\mathbb{P}(\hat{X}_{t}(ky)\leq x).
\]
By Lamperti's representation of positive self-similar Markov process, see e.g. \cite[Chapter~13]{MR3155252}, $\hat{X}_t(x)$ is of the form $\exp(L_{\varphi^{x}(t)})$ for some L\'evy process $L$ where
$t\mapsto \varphi^{x}(t)$  the time-change given in the statement. To identify the Laplace exponent $\kappa$ of $L$, note that by $(\alpha-1)$-self-similarity, one has $\kappa(q)=x^{-q+\alpha-1}\hat{\mathcal{L}}p_{q}(x)$ with $p_q(x)=x^{q}$. The result follows from simple computations.
\end{proof}

\subsection{Neveu CSBP}

We now turn in this section to the Neveu CSBP. This CSBP has branching mechanism $\Psi(q)=q\log(q)$. Recall its L\'evy-Khintchine form
\[\Psi(q)=(\gamma-1)q+\int_{0}^{\infty}\left(e^{-qh}-1+qh\mathbbm{1}_{\{h\leq 1\}}\right)\frac{\ddr h}{h^2}, \text{ for any } q\geq 0\]
where $\gamma =\int_{1}^{\infty}e^{-y}y^{-2}\ddr y$ is the Euler-Mascheroni constant. Note that Grey's condition is not verified by this process. Solving the differential equation \eqref{odev} yields $v_t(\lambda)=\lambda^{e^{-t}}$ for any $t\geq 0$ and $\lambda \in (0,\infty)$. For any fixed $t$, the subordinator $(X_t(x),x\geq 0)$ is stable with parameter $e^{-t}$. For Neveu CSBP, the consecutive coalescent process $C^\lambda$ defined in Section \ref{sec:markovianCoalescent} happens to be homogeneous in time, and not to depend on $\lambda$.

\begin{proposition}
\label{prop:lawC}
For any $\lambda >0$, the consecutive process $(C^{\lambda}(t),t\geq 0)$ is an homogeneous consecutive coalescent whose coagulation rate $\mu$ is $\mu(k)=\frac{1}{k(k-1)}$ for any $k\geq 2$. The block sizes at time $t\geq 0$ have generating function $\mathbb{E}[z^{\#C_1(t)}]=1-(1-z)^{e^{-t}}$ and for any $k\geq 1$
\[
  \mathbb{P}(\#C_1(t)=k)=\frac{e^{-t}(2-e^{-t})...(k-1-e^{-t})}{k!}.
\]
\end{proposition}

\begin{proof}
By Theorem \ref{markovcoal}, and applying the change of variable $u=v_t(\lambda)x$, we see that for any $k\geq 2$,
\[
  \mu^{\lambda}_t(k)=v_t(\lambda)^{k-1}\int_{0}^{\infty}\frac{x^{k}}{k!}e^{-v_t(\lambda)x}\frac{\ddr x}{x^2}=\frac{1}{k(k-1)}
\]
which does not depend on $\lambda$ nor on $t$. Since $C(t)$ is a $(\lambda, v_t)$-Poisson box with $v_t(q)=q^{e^{-t}}$, the other statements can be obtained by a direct application of Theorem \ref{markovcoal}. See also the calculations around Lemma 7 in Pitman \cite{MR1466546}.
\end{proof}

\begin{lemma}
\label{convergence}
Consider a consecutive coalescent $(C(t),t\geq 0)$ with coagulation rate $\mu(k)=\frac{1}{k(k-1)}$ for any $k\geq 2$ then, as $n$ goes to $\infty$
\[
  \left(\frac{\#C_{|[nx]}(t)}{n^{e^{-t}}}, t\geq 0,x\geq 0\right)\Longrightarrow (\hat{X}_t(x),t\geq 0, x\geq 0)
\]
in finite-dimensional sense in time and in the Skorokhod topology in $x$. 
\end{lemma}

\begin{proof}
We simply prove the convergence in law of $\left(\frac{\#C_{|[nx]}(t)}{n^{e^{-t}}},x\geq 0\right)$ toward $\hat{X}_t$ for a fixed value of $t$, with the Skorokhod topology. Then, the finite-dimensional convergence is deduce from the cocycle property of $\hat{X}$ (Proposition~\ref{flowA}) and $C$ (Proposition~\ref{flowpartitions}).
For any $t > 0$ and $n\in \mathbb{N}$, set $Z_{-t,0}(n)=\sum_{j=1}^{n}\#C_{j}(t)$. The process $(Z_{-t,0}(n),n\geq 0)$ is a random walk and from Proposition~\ref{prop:lawC} we see that
\begin{align*}
  \P(\#C_1(t) = k) &= \frac{e^{-t}(2-e^{-t})...(k-1-e^{-t})}{k!}\\
   &= e^{-t} \frac{\Gamma(k-e^{-t})}{\Gamma(2-e^{-t}) \Gamma(k+1)} \sim_{k \to \infty} \frac{e^{-t}}{\Gamma(2-e^{-t})} k^{-1-e^{-t}}.
\end{align*}
Therefore, the law of $\#C_1$ is in the domain of attraction of a stable random variable with parameter $e^{-t}$. 
Using an extension of Donsker's theorem to stable distributions, due to Prokhorov \cite{MR0084896}, we obtain that
\[
\left( \frac{Z_{-t,0}\left(\floor{n^{e^{-t}} x}\right)}{n} , x \geq 0 \right) \underset{n \to \infty}{\Longrightarrow} (\tilde{X}_t(x),x\geq 0),
\]
where $\tilde{X}_t$ is a stable subordinator with Laplace exponent $\lambda \mapsto \lambda^{e^{-t}}$.

To conclude, we observe that $(\hat{X}_t(x), x \geq 0)$ has the same law as $\tilde{X}_t^{-1}$ the right-continuous inverse of $\tilde{X}_t$ and that $(\#C_{|[n]}(t), n \geq 0)$ is the right continuous inverse of $(Z_{-t,0}(n),n\geq 0)$. Hence, as $f \mapsto f^{-1}$ is continuous for the Skorokhod topology, we have convergence in law of $\left(\frac{\#C_{|[nx]}(t)}{n^{e^{-t}}},x\geq 0\right)$ toward $\hat{X}_t$.
\end{proof}

\begin{lemma}[M\"ohle \cite{MR3333734}, Mittag-Leffler process]
The process $(\hat{X}_t,t\geq 0)$ has for generator
\[
  \hat{\mathcal{L}}f(z)=z\int_{0}^{z}\left(f(z-h)-f(z)+hf'(z)\right)\frac{\ddr h}{h^2}+((1-\gamma)z-z\log(z))f'(z).
\]
\end{lemma}

\begin{proof}
By Lemma \ref{courregeform}, we see that $\nu(z,\ddr h)=\mathbbm{1}_{\{h\leq z\}}\left((z-h)\pi(\ddr h)+\bar{\pi}(h)\ddr h\right)=\mathbbm{1}_{\{h\leq z\}}\frac{z}{h^{2}}$. For any $z\geq 0$,
\begin{align*}
  b(z)&=(1-\gamma)z+\int_{0}^{\infty}\frac{z}{h}\left(1_{\{h\leq 1\}}-1_{\{h\leq z\}}\right)\ddr h\\
  &=z\int \frac{1}{h}\left(1_{\{h\leq 1\}}-1_{\{h\leq z\}}\right) 1_{\{z\leq 1\}}\ddr h-z\int \frac{1}{h}\left(1_{\{h\leq z\}}-1_{\{h\leq 1\}}\right) 1_{\{z>1\}}\ddr h\\
  &=(1-\gamma)z+z\int_{z}^{1}\frac{\ddr h}{h}1_{\{z\leq 1\}}-z\int_{1}^{z}\frac{\ddr h}{h}1_{\{z>1\}}=(1-\gamma)z-z\log(z).\qedhere
\end{align*}
\end{proof}

\begin{proposition}[Bertoin and Baur \cite{MR3399834}]
The process $(\log \hat{X}_t, t\geq 0)$ is a generalized Ornstein-Uhlenbeck process: 
\begin{equation}\label{OU}
  \log \hat{X}_t=\log(x)+L_t-\int_{0}^{t}\log \hat{X}_s\ddr s
\end{equation}
where $(L_t,t\geq 0)$ is a spectrally negative L\'evy process with Laplace exponent
\[
  \kappa(q)=-\gamma q +\int_{-\infty}^{0}\left(e^{qu}-1-qu\right)\frac{e^{u}}{(1-e^{u})^2}\ddr u.
\]
\end{proposition}

\begin{proof}
By injectivity of $g:x\mapsto \log(x)$, the generator of $(Y_t,t\geq 0):=(\log \hat{X}_t,t\geq 0)$ is given by $\mathcal{A}f(y)=\hat{\mathcal{L}}(f\circ g)(g^{-1}(y))$ and a computation provides
\[
  \mathcal{A}f(y)=\int_{-\infty}^{0}\left( f(y+u)-f(y)-uf'(y)\right)\nu(\ddr u)-\gamma f'(y),
\]
with $\nu(\ddr u)=\frac{e^{u}}{(1-e^{u})^2}\ddr u$. It is well-known that the process with generator $\mathcal{A}$ is a weak solution of the equation (\ref{OU}). See for instance,  Sato and Yamazato \cite[Theorem 3.1]{MR738769}.
\end{proof}

The last two statements already appear in the study of the Bolthausen-Sznitman coalescent. We explain now some connections between the Neveu consecutive coalescent and the Bolthausen-Sznitman exchangeable coalescent. The following is a rephrasing of an observation made by Hénard \cite{MR3375893} and M\"ohle \cite{MR3333734}. Denote by $(N_t^{(n)},t\geq 0)$ the number of blocks in a Bolthausen-Sznitman coalescent started from $n$ blocks.  Recall that $(N_t^{(n)},t\geq 0)$ jumps from $n$ to $n-k+1$ at rate $\frac{n}{k(k+1)}$ for any $k\in [|2,n|]$. By Proposition 
\ref{numberofblocks}, one can check that   
$(\#C_{|[n]}(t),t\geq 0)$ loses $k$ blocks at the same rate as $(N_t^{(n)},t\geq 0)$. Therefore $(N_t^{(n)},t\geq 0)$ and  $(\#C_{|[n]}(t),t\geq 0)$ have the same law and by Lemma \ref{convergence}, as $n$ goes to $\infty$
\[\left(\frac{N_t^{(n)}}{n^{e^{-t}}},t\geq 0\right)\Longrightarrow (\hat{X}_t(1),t\geq 0)\]
in the Skorohod topology. Such result was shown by M\"ohle in \cite[Theorem 1.1]{MR3333734}, Kukla and M\"ohle in \cite[Theorem 2.1-(a)]{MR3758343} with different techniques. We refer also to 
Bertoin and Baur \cite[Theorem 3.1-(i)]{MR3399834} for an almost-sure convergence. The connections between Neveu's consecutive coalescent and Bolthausen-Sznitman exchangeable one are not surprising since it is known that for any initial size $x$, the genealogy of i.i.d random variables sampled in $[0,x]$ is described by a Bolthausen-Sznitman coalescent, see Bertoin and Le Gall \cite[Theorem 4]{MR1771663}.

\bigskip

Several natural questions on the inverse flow and its consecutive coalescent have not been addressed here and are left for possible future works. It might be interesting for instance to look for a complete description of the two-parameter flow $(\hat{X}_t(x),t\geq 0,x\geq 0)$ in the general case, as given for the Feller flow in Section \ref{fellerflow}. Moreover, the genealogy of the total population has only been characterized under the Grey's condition. When this condition is not fulfilled the process $(\mathscr{C}(t),t\geq 0)$ cannot be described by a single consecutive coalescent on $\mathbb{N}$. We recall that Duquesne and Winkel \cite{MR2322700} have described the genealogy forward in time of a CSBP (including those without Grey's condition) through a collection of  continuous-time Galton-Watson processes. A natural question is thus to see if in a dual way, one can represent the backward genealogy of the total population with a collection of consecutive coalescents on $\mathbb{N}$.

\bigskip

\textbf{Acknowledgement.} C.F and B.M are partially supported by the French National Research Agency (ANR): ANR GRAAL, ANR MALIN and LABEX MME-DII. C.M is supported by the NSFC of China
(11671216).  C.F would like to thank T. Duquesne and V. Rivero for fruitful discussions at early stages of this work.

\appendix

\section{Intermediary results}
\label{sec:appendix}

\subsection{Right-continuous inverse}

In this section, we recall and compile some elementary properties on right continuous inverse of càdlàg non-decreasing functions. As multiple competing definitions of generalized inverse coexist, it can be challenging to find a single reference for the results we need. Therefore we give a short proof of these well-known facts, in order to be self-contained. Let $f$ be a càdlàg non-decreasing function on $\R_+$, we denote by
\begin{equation}
  \label{eqn:rightContinuousInverse}
  f^{-1} : y \in [0,\infty) \mapsto \inf\{ x \geq 0 : f(x) > y \}
\end{equation}
its right continuous inverse.

\begin{lemma}
\label{lem:factsInverse}
Let $f,(f_n,n \geq  1), g$ be càdlàg non-decreasing functions on $\R_+$.
\begin{enumerate}
  \item The function $f^{-1}$ is non-decreasing and càdlàg.
  \item For every $x,y \geq 0$, we have $f(x) > y \iff f^{-1}(y) < x$.
  \item We have $(f \circ g)^{-1} = g^{-1} \circ f^{-1}$.
  \item If $\lim_{n \to \infty} f_n = \mathrm{Id}$ pointwise, then $\lim_{n \to \infty} f_n^{-1} = \mathrm{Id}$ pointwise, with $\mathrm{Id}$ being the identity function on $[0,\infty)$.
\end{enumerate}
\end{lemma}

\begin{remark}
Dini's theorems imply that both convergences in (iv) hold uniformly on compact sets.
\end{remark}

\begin{proof}
Let $f$ be a càdlàg non-decreasing function, note that for all $y < z$, we have
\[\{x \geq 0 : f(x) > z\} \subset \{x \geq 0 : f(x) > y\}.\]
Therefore $f^{-1}(y) \leq f^{-1}(z)$, which proves that $f^{-1}$ is increasing. In particular, it has left limits at each point. We now observe that for all $y \geq 0$, as $f$ is non-decreasing,
\[
  \inf_{z > y} f^{-1}(z) = \inf\{  \inf\{ x \geq 0 : f(x) > z \}, z > y \} = \inf\{ x \geq 0 : f(x) > y \} = f^{-1}(y),
\]
proving that $f^{-1}$ is right continuous at point $y$, which proves (i).

Let $x,y \geq 0$, we first assume that $f^{-1}(y) < x$. Then by definition of $f^{-1}$, there exists $u \in [f^{-1}(y), x)$ such that $f(u)>y$. As $f$ is non-decreasing, we deduce that $f(x) \geq f(u) > y$.

We now assume that $f^{-1}(y) \leq x$. As $f$ is non-decreasing, we observe that $f(x) \geq f(f^{-1}(y))$. Therefore, the only thing left to prove is that 
\begin{equation}
  \label{eqn:fcircfmoins}
  \forall y \geq 0, f(f^{-1}(y)) \geq y
\end{equation} 
We write $z = f^{-1}(y)$. By definition of $f^{-1}(y)$, for all $\epsilon > 0$, there exists $u < z + \epsilon$ such that $f(u) > y$. Then, as $f$ is right-continuous, we have $f(z) = \inf_{u > z} f(u)$, thus for all $\eta > 0$, there exists $\epsilon > 0$ such that if $u < z + \epsilon$, then $f(u) < f(z) + \eta$. As a result, for all $\eta > 0$, there exists $u < z + \epsilon$ such that $y < f(u) < f(z) + \eta$. This inequality being true for all $\eta > 0$, we therefore conclude that $f(z) \geq y$, completing the proof \eqref{eqn:fcircfmoins}. We thus deduce that $f(x) \geq y$, completing the proof of (ii).

In a third time, we note that given $f$ and $g$ two càdlàg non-decreasing functions on $\R_+$, we have for all $y \geq 0$,
\[
  (f \circ g)^{-1}(y) = \inf \{ z \geq 0 : (f \circ g) (z) > y\} = \inf\{ z \geq 0 : f(g(z)) > y\}.
\]
By point (ii), this can therefore be rewritten as
\[
  (f \circ g)^{-1}(y) = \inf\{ z \geq 0 : g(z) > f^{-1}(y)\} = g^{-1} \circ f^{-1}(y),
\]
proving point (iii).

We finally prove the last point. Let $(f_n)$ be a sequence of non-decreasing càdlàg functions such that $\lim_{n \to \infty} f_n = \mathrm{Id}$ pointwise. We prove that for all $y\geq 0$, $\lim_{n \to \infty} f^{-1}_n(y) =  y$. Let $\epsilon > 0$, by point (ii), we have that
\[
  f^{-1}_n(y) < y + \epsilon \iff f_n(y+\epsilon) > y.
\]
As $\lim_{n \to \infty} f_n(y+\epsilon) = y+\epsilon$, we conclude that for all $n$ large enough, $f^{-1}_n(y) < y + \epsilon$. Similarly, we also have
\[
  f^{-1}_n(y) \geq y - \epsilon \iff f_n(y-\epsilon) \geq y.
\]
therefore $f^{-1}_n(y) \geq y - \epsilon$ for all $n$ large enough by pointwise convergence of $f_n$ at point $y-\epsilon$. This concludes the proof of (iv).
\end{proof}

\subsection{Discretization of subordinators}
\label{sec:keyPPA}

In this section, we introduce the key lemma allowing to study the genealogical structure of CSBPs. Namely, we prove that the pullback measure of a Poisson process by a subordinator is a Poisson process decorated by i.i.d. integer-valued random variables.

\begin{lemma}
\label{lem:keyPPA}
Let $\lambda \geq 0$ and $(X(x), x \geq 0)$ be a subordinator with Lévy-Khinchine exponent
\[
  \phi : \mu \mapsto d \mu + \int \left(1 - e^{-\mu x}\right) \ell(\ddr x).
\]
We denote by $N$ an independent Poisson point process with intensity $\lambda$, and we write $(J_j, j \geq 1)$ the positions of the atoms of $N$, ranked in the decreasing order. Then, setting $M = \sum_{j = 1}^\infty \delta_{X^{-1}(J_j)}$ the image measure of $N$ by $X^{-1}$, we have
\[
  M = \sum_{j=1}^\infty Z_j \delta_{J_j'},
\]
where $(J_j', j \geq 1)$ are the atoms of a Poisson point process with intensity $\phi(\lambda)$ and $(Z_j, j \geq 1)$ are i.i.d. random variables, independent of $(J'_j, j \geq 1)$, such that
\[
  \P(Z_1=k)=\frac{1}{\phi(\lambda)}\int_{0}^{\infty}\frac{(\lambda x)^{k}}{k!}e^{-\lambda x} \ell(\ddr x) + d \ind{k=1}=(-1)^{k-1}\frac{\lambda^k}{k!}\frac{\phi^{(k)}(\lambda)}{\phi(\lambda)},
\]
i.e. $\E(z^{Z_1}) = 1 - \frac{\phi(\lambda(1-z))}{\phi(\lambda)}$ for all $z \in [0,1]$.
\end{lemma}

\begin{proof}
The proof is based on a joint construction by the same ``master'' Poisson point process of the subordinator $X$ and the Poisson point process $N$, in such a way that $M$ becomes a simple functional of that master point process. To see why such a construction is possible, we write
\[
  \phi(\lambda) = d \lambda + \int (1 - e^{-\lambda x}) \ell(\ddr x),
\]
with $d \geq 0$ the drift and $\ell$ the Lévy measure of $X$ on $\R_+$. By the Lévy-Itô décomposition, one can write
\[
  \forall t \geq 0, \quad  X_t = d t +  \sum_{0 \leq s \leq t} x_t,
\]
with $(t,x_t)_{t \geq 0}$ being the atoms of a Poisson point process with intensity $\ddr t \otimes \ell(\ddr x)$. The proof being slightly simpler for $d=0$, we focus here on the case $d>0$.

Recalling that $\mathcal{D}$ denote the set of càdlàg non-decreasing functions on $\R_+$, we introduce the point process $R = \sum_{i \in I} \delta_{(t_i,x_i,N^{(i)})}$ on $\R_+ \times \R_+ \times \mathcal{D}$ with intensity $\ddr t \otimes \ddr x \otimes \mathcal{P}^\lambda(\ddr N)$, with $\mathcal{P}^\lambda$ begin the law of a Poisson point process with intensity $\lambda$ on $\R_+$. We also set $N^{(0)}$ an independent Poisson point process with intensity $\lambda$. We then define
\[
  \bar{X}_t = d t + \sum_{i \in I} x_i \ind{t_i \leq t},
\]
which is a subordinator with Lévy-Khinchine exponent $\phi$. Then, denoting $(J^{(i)}_j, j \geq 1)$ the atoms of the Poisson point process $N^{(i)}$, we set
\[
  \bar{N} = \sum_{j =1}^\infty \delta_{\bar{X}_{J^{(0)}_j/d}} + \sum_{i \in I} \sum_{j=1}^\infty \delta_{\bar{X}_{t_i-} + J^{(i)}_j} \ind{J^{(i)}_j < x_i}.
\]
Heuristically, the point process $\bar{N}$ can be thought of as follows: $\R_+$ is divided in intervals $\cup_{i \in I} [\bar{X}_{t_i-},\bar{X}_{t_i}]$ corresponding to jumps in the subordinator $\bar{X}$ and the remaining space $I$ corresponding to points with an antecedent by $X$. Atoms are added to the interval $[\bar{X}_{t_i-},\bar{X}_{t_i}]$ according to the point process $N^{(i)}$, and to the set $I$ with the point process $N^0$. It should then be heuristically clear that $\bar{N}$ is a Poisson point process with intensity $\lambda$ independent of $\bar{X}$. To verify it, we compute its conditional Laplace transform against a smooth locally compact test function $f$. By construction, $(N^{(i)}, i \in I \cup\{0\})$ are i.i.d. Poisson point process with intensity $\lambda$, which are further independent from $X$, thus
\begin{align*}
  &\E\left(\exp\left(- \crochet{N,f} \right) \middle| X \right)\\
  =&\E\left( \exp\left( - \sum_{j \geq 0} f(X_{J^{(0)}_j}/d)\right) \middle| X \right) \prod_{i \in I} \E\left( \exp\left( - \sum_{j \geq 0} f(J^{(i)}_j + X_{t_i-}) \ind{J^{(i)}_j < x_i} \right) \middle| X\right)\\
  = &\exp\left( -\lambda \int_0^\infty 1 - e^{-f(X_{s/d})} \ddr s - \lambda \sum_{i \in I} \int_{X_{t_i-}}^{X_{t_i}} 1 - e^{-f(X_s)} \ddr s  \right) \quad \text{a.s.}\\
  = &\exp\left( \lambda \int 1 - e^{-f(x)} \ddr x \right),
\end{align*}
by change of variable, using that $X_t' = d$ at all continuity points $t$ of $X$.

As a result, the couple $(\bar{X},\bar{N})$ has same law as $(X,N)$ given in the lemma. Moreover, we have immediately by construction that
\[
  \bar{M} := \bar{X}^{-1} \ast \bar{N} = \sum_{j =1}^\infty \delta_{J^{(0)}_j/d} + \sum_{i \in I} N^{(i)}([0,x_i])\delta_{t_i}.
\]
and computing the law of that point process is straightforward by the definition of $R$. Indeed, by independence, for any continuous function $f$ with compact support, we have
\[
  \E\left( \exp\left( - \crochet{N,f} \right)\right) = \E\left( \exp\left( - \sum_{j \geq 1} f(J^{(0)}_j/d) \right) \right) \E\left( \exp\left( - \sum_{i \in I} N^{(i)}([0,x_i]) f(t_i) \right)\right).
\]
Then, using Campbell's formula, we have both
\[
\E\left( \exp\left( - \sum_{j \geq 1} f(J^{(0)}_j/d) \right) \right)  = \exp\left( \lambda d \int 1 - e^{-f(x)} \ddr x\right)
\]
\[
\E\left( \exp\left( - \sum_{i \in I} N^{(i)}([0,x_i]) f(t_i) \right)\right) = \exp\left( \int 1 - e^{-N([0,x])f(t)} \ddr t \ell(\ddr x) \mathcal{P}^\lambda(\ddr N) \right).
\]
But as under law $\mathcal{P}^\lambda$, $N([0,x])$ is a Poisson random variable with parameter $\lambda x$, the last inequality can be written, by Fubini theorem
\begin{align*}
  \E\left( \exp\left( - \sum_{i \in I} N^{(i)}([0,x_i]) f(t_i) \right)\right) &=\exp\left( \int 1 - \exp\left( \lambda x (e^{-f(t)} - 1) \right) \ddr t \ell(\ddr x) \right)\\
  &= \exp\left( - \int \frac{\phi\left(\lambda (1 - e^{-f(t)})\right)}{\phi(\lambda)} \ddr t \right).
\end{align*}
We deduce that the Laplace transform of $\bar{M}$ is the same as the one of $M$ given in the lemma, which concludes the proof.
\end{proof}

This result can be straightforwardly extended to killed subordinators, by constructing it as a limit of non-killed subordinators. For the sake of completeness, we add a proof of the following result.
\begin{corollary}
\label{corkilled}
Let $\lambda \geq 0$ and $(X(x), x \geq 0)$ be a subordinator with Laplace exponent
\[
  \phi : \mu \mapsto \kappa + d \mu + \int\left( 1 - e^{-\mu x} \right)\ell(\ddr x).
\]
With the same notation as in the previous lemma, we have $M = \sum_{j=1}^\infty Z'_j \delta_{J_j'}$, where $(J_j', j \geq 1)$ are the atoms of a Poisson point process with intensity $\phi(\lambda)$, $(Z_j, j \geq 1)$ are i.i.d. random variables, independent of $(J'_j, j \geq 1)$, such that
\[
  \P(Z_1=k)=\frac{1}{\phi(\lambda)}\int_{0}^{\infty}\frac{(\lambda x)^{k}}{k!}e^{-\lambda x} \ell(\ddr x) + d \ind{k=1}=(-1)^{k-1}\frac{\lambda^k}{k!}\frac{\phi^{(k)}(\lambda)}{\phi(\lambda)},
\]
\[
  \text{and} \qquad Z'_j = \begin{cases} Z_j & \text{ if } \sup_{i < j} Z_i < \infty, \\ 0 &\text{ otherwise.}\end{cases}
\]
\end{corollary}

\begin{proof}
Let $Y$ be a subordinator with Laplace exponent $\lambda \mapsto d \lambda + \int 1 - e^{-\lambda x} \ell(\ddr x)$, and $R$ an independent Poisson process with intensity $\kappa$. Observe that for all $r > 0$, the process defined by
\[
  Y^r(t) = Y(t) + r R(t), \quad t \geq 0,
\]
is a Lévy process, and that $X = \lim_{r \to \infty} Y^r$ is a Lévy process with Laplace exponent $\phi$. We set $(J_j, j \geq 1)$, $({J'}^r_j, j \geq 1)$ and $(Z^r_j, j \geq 1)$ the quantities obtained by applying Lemma \ref{lem:keyPPA}, and
\[T = \inf\{ t > 0 : R_t = 1\}.\]
Observe that for all $j$ such that ${J'}^r_j < T$, the quantities ${J'}^r_j$ and $Z^r_j$ do not depend on $r$. On the contrary, for all $j$ such that $J_j > T$, as $r \to \infty$, all value $(Y^r)^{-1}(J_j)$ converge toward $T$, and the associated value of $Z^r$ to the atom at position $T$ converges toward $\infty$.

Explicit formulas for the law of $Z^\infty$ are straightforward Poisson computations.
\end{proof}

\bibliographystyle{amsalpha}

\begin{thebibliography}{M{\relax\"o}{}h15}

\bibitem[Ald93]{MR1207226}
David Aldous, \emph{The continuum random tree. {III}}, Ann. Probab. \textbf{21}
  (1993), no.~1, 248--289. \MR{1207226}

\bibitem[AN04]{MR2047480}
K.~B. Athreya and P.~E. Ney, \emph{Branching processes}, Dover Publications,
  Inc., Mineola, NY, 2004, Reprint of the 1972 original [Springer, New York;
  MR0373040]. \MR{2047480}

\bibitem[AP05]{MR2193998}
David Aldous and Lea Popovic, \emph{A critical branching process model for
  biodiversity}, Adv. in Appl. Probab. \textbf{37} (2005), no.~4, 1094--1115.
  \MR{2193998}

\bibitem[Ath12]{MR3012088}
K.~B. Athreya, \emph{Coalescence in critical and subcritical {G}alton-{W}atson
  branching processes}, J. Appl. Probab. \textbf{49} (2012), no.~3, 627--638.
  \MR{3012088}

\bibitem[BB15]{MR3399834}
Erich Baur and Jean Bertoin, \emph{The fragmentation process of an infinite
  recursive tree and {O}rnstein-{U}hlenbeck type processes}, Electron. J.
  Probab. \textbf{20} (2015), no. 98, 20. \MR{3399834}

\bibitem[BBC{\etalchar{+}}05]{MR2120246}
Matthias Birkner, Jochen Blath, Marcella Capaldo, Alison Etheridge, Martin
  M\"{o}hle, Jason Schweinsberg, and Anton Wakolbinger, \emph{Alpha-stable
  branching and beta-coalescents}, Electron. J. Probab. \textbf{10} (2005), no.
  9, 303--325. \MR{2120246}

\bibitem[BBS07]{MR2349577}
Julien Berestycki, Nathana\"el Berestycki, and Jason Schweinsberg,
  \emph{Beta-coalescents and continuous stable random trees}, Ann. Probab.
  \textbf{35} (2007), no.~5, 1835--1887. \MR{2349577}

\bibitem[BD16]{MR3531711}
Hongwei Bi and Jean-Fran\c{c}ois Delmas, \emph{Total length of the genealogical
  tree for quadratic stationary continuous-state branching processes}, Ann.
  Inst. Henri Poincar\'e Probab. Stat. \textbf{52} (2016), no.~3, 1321--1350.
  \MR{3531711}

\bibitem[Ber06]{MR2253162}
Jean Bertoin, \emph{Random fragmentation and coagulation processes}, Cambridge
  Studies in Advanced Mathematics, vol. 102, Cambridge University Press,
  Cambridge, 2006. \MR{2253162}

\bibitem[BFM08]{MR2455180}
Jean Bertoin, Joaquin Fontbona, and Servet Mart{\'{\i}}nez, \emph{On prolific
  individuals in a supercritical continuous-state branching process}, J. Appl.
  Probab. \textbf{45} (2008), no.~3, 714--726. \MR{2455180}

\bibitem[BLG00]{MR1771663}
Jean Bertoin and Jean-Fran\c{c}ois Le~Gall, \emph{The {B}olthausen-{S}znitman
  coalescent and the genealogy of continuous-state branching processes},
  Probab. Theory Related Fields \textbf{117} (2000), no.~2, 249--266.
  \MR{1771663}

\bibitem[BLG03]{MR1990057}
\bysame, \emph{Stochastic flows associated to coalescent processes}, Probab.
  Theory Related Fields \textbf{126} (2003), no.~2, 261--288. \MR{1990057}

\bibitem[BLG05]{LGB2}
Jean Bertoin and Jean-Fran{\c{c}}ois Le~Gall, \emph{Stochastic flows associated
  to coalescent processes. {II}. {S}tochastic differential equations}, Ann.
  Inst. H. Poincar\'e Probab. Statist. \textbf{41} (2005), no.~3, 307--333.
  \MR{2139022 (2005m:60067)}

\bibitem[BLG06a]{LGB3}
\bysame, \emph{Stochastic flows associated to coalescent processes. {III}.
  {L}imit theorems}, Illinois J. Math. \textbf{50} (2006), no.~1-4, 147--181
  (electronic). \MR{2247827 (2008c:60032)}

\bibitem[BLG06b]{MR2247827}
\bysame, \emph{Stochastic flows associated to coalescent processes. {III}.
  {L}imit theorems}, Illinois J. Math. \textbf{50} (2006), no.~1-4, 147--181.
  \MR{2247827}

\bibitem[CD12]{MR3025710}
Yu-Ting Chen and Jean-Fran\c{c}ois Delmas, \emph{Smaller population size at the
  {MRCA} time for stationary branching processes}, Ann. Probab. \textbf{40}
  (2012), no.~5, 2034--2068. \MR{3025710}

\bibitem[CS85]{MR781422}
Peter Clifford and Aidan Sudbury, \emph{A sample path proof of the duality for
  stochastically monotone {M}arkov processes}, Ann. Probab. \textbf{13} (1985),
  no.~2, 558--565. \MR{781422}

\bibitem[DFM14]{MR3264444}
Xan Duhalde, Cl\'ement Foucart, and Chunhua Ma, \emph{On the hitting times of
  continuous-state branching processes with immigration}, Stochastic Process.
  Appl. \textbf{124} (2014), no.~12, 4182--4201. \MR{3264444}

\bibitem[DK99]{DonnKurtz}
Peter Donnelly and Thomas~G. Kurtz, \emph{Particle representations for
  measure-valued population models}, Ann. Probab. \textbf{27} (1999), no.~1,
  166--205. \MR{1681126 (2000f:60108)}

\bibitem[DL14]{MR3164759}
Thomas Duquesne and Cyril Labb\'{e}, \emph{On the {E}ve property for {CSBP}},
  Electron. J. Probab. \textbf{19} (2014), no. 6, 31. \MR{3164759}

\bibitem[DLG02]{MR1954248}
Thomas Duquesne and Jean-Fran\c{c}ois Le~Gall, \emph{Random trees, {L}\'evy
  processes and spatial branching processes}, Ast\'erisque (2002), no.~281,
  vi+147. \MR{1954248}

\bibitem[DW07]{MR2322700}
Thomas Duquesne and Matthias Winkel, \emph{Growth of {L}\'{e}vy trees}, Probab.
  Theory Related Fields \textbf{139} (2007), no.~3-4, 313--371. \MR{2322700}

\bibitem[ER10]{MR2582640}
Steven~N. Evans and Peter~L. Ralph, \emph{Dynamics of the time to the most
  recent common ancestor in a large branching population}, Ann. Appl. Probab.
  \textbf{20} (2010), no.~1, 1--25. \MR{2582640}

\bibitem[Fel51]{MR0046022}
William Feller, \emph{Diffusion processes in genetics}, Proceedings of the
  {S}econd {B}erkeley {S}ymposium on {M}athematical {S}tatistics and
  {P}robability, 1950, University of California Press, Berkeley and Los
  Angeles, 1951, pp.~227--246. \MR{0046022}

\bibitem[FFK17]{2017arXiv170203533F}
D.~Fekete, J.~Fontbona, and A.~E. Kyprianou, \emph{{Skeletal stochastic
  differential equations for continuous-state branching process}}, arXiv
  e-prints, February 2017.

\bibitem[FH13]{MR3035751}
Cl\'{e}ment Foucart and Olivier H\'{e}nard, \emph{Stable continuous-state
  branching processes with immigration and {B}eta-{F}leming-{V}iot processes
  with immigration}, Electron. J. Probab. \textbf{18} (2013), no. 23, 21.
  \MR{3035751}

\bibitem[FM16]{2016arXiv161106178F}
C.~{Foucart} and C.~{Ma}, \emph{Continuous-state branching processes, extremal
  processes and super-individuals}, Ann. Inst. H. Poincaré Probab. Statist.
  (2016), 27pp., To appear.

\bibitem[Fou12]{MR3069373}
Cl\'ement Foucart, \emph{Generalized {F}leming-{V}iot processes with
  immigration via stochastic flows of partitions}, ALEA Lat. Am. J. Probab.
  Math. Stat. \textbf{9} (2012), no.~2, 451--472. \MR{3069373}

\bibitem[Get80]{MR580144}
R.~K. Getoor, \emph{Transience and recurrence of {M}arkov processes}, Seminar
  on {P}robability, {XIV} ({P}aris, 1978/1979) ({F}rench), Lecture Notes in
  Math., vol. 784, Springer, Berlin, 1980, pp.~397--409. \MR{580144}

\bibitem[GH16]{MR3581248}
Nicolas Grosjean and Thierry Huillet, \emph{On a coalescence process and its
  branching genealogy}, J. Appl. Probab. \textbf{53} (2016), no.~4, 1156--1165.
  \MR{3581248}

\bibitem[Gre74]{MR0408016}
D.~R. Grey, \emph{Asymptotic behaviour of continuous time, continuous
  state-space branching processes}, J. Appl. Probability \textbf{11} (1974),
  669--677. \MR{0408016}

\bibitem[Gri74]{MR0362529}
Anders Grimvall, \emph{On the convergence of sequences of branching processes},
  Ann. Probability \textbf{2} (1974), 1027--1045. \MR{0362529}

\bibitem[HJR17]{2017arXiv170300299H}
S.~C. {Harris}, S.~G.~G. {Johnston}, and M.~I. {Roberts}, \emph{{The coalescent
  structure of continuous-time Galton-Watson trees}}, arXiv e-prints, March
  2017.

\bibitem[H{\relaxé}{}n15]{MR3375893}
Olivier H{\relaxé}{}nard, \emph{The fixation line in the
  {$\Lambda$}-coalescent}, Ann. Appl. Probab. \textbf{25} (2015), no.~5,
  3007--3032. \MR{3375893}

\bibitem[ILP15]{MR3313753}
Gautam Iyer, Nicholas Leger, and Robert~L. Pego, \emph{Limit theorems for
  {S}moluchowski dynamics associated with critical continuous-state branching
  processes}, Ann. Appl. Probab. \textbf{25} (2015), no.~2, 675--713.
  \MR{3313753}

\bibitem[ILP18]{iyer_leger_pego_2018}
\bysame, \emph{Coagulation and universal scaling limits for critical
  galton–watson processes}, Advances in Applied Probability \textbf{50}
  (2018), no.~2, 504–542.

\bibitem[Ji{\v{r}}58]{MR0101554}
Miloslav Ji{\v{r}}ina, \emph{Stochastic branching processes with continuous
  state space}, Czechoslovak Math. J. \textbf{8 (83)} (1958), 292--313.
  \MR{0101554}

\bibitem[Ji{\v{r}}69]{MR0247676}
\bysame, \emph{On {F}eller's branching diffusion processes}, \v{C}asopis
  P\v{e}st. Mat. \textbf{94} (1969), 84--90, 107. \MR{0247676}

\bibitem[JL18]{JohnstonLambert}
S.G.G. {Johnston} and A.~{Lambert}, \emph{The coalescent structure of samples
  from branching processes: a unifying poissonisation approach}, In
  preparation, 2018+.

\bibitem[{Joh}17]{2017arXiv170908500J}
S.~G.~G. {Johnston}, \emph{{Coalescence in supercritical and subcritical
  continuous-time Galton-Watson trees}}, arXiv e-prints, September 2017.

\bibitem[Kal02]{Kallenberg}
Olav Kallenberg, \emph{Foundations of modern probability}, second ed.,
  Probability and its Applications (New York), Springer-Verlag, New York, 2002.
  \MR{1876169 (2002m:60002)}

\bibitem[KM18]{MR3758343}
Jonas Kukla and Martin M{\relax\"o}{}hle, \emph{On the block counting process
  and the fixation line of the {B}olthausen-{S}znitman coalescent}, Stochastic
  Process. Appl. \textbf{128} (2018), no.~3, 939--962. \MR{3758343}

\bibitem[Kol11]{MR2858558}
V.~N. Kolokol'tsov, \emph{Stochastic monotonicity and duality of
  one-dimensional {M}arkov processes}, Mat. Zametki \textbf{89} (2011), no.~5,
  694--704. \MR{2858558}

\bibitem[KP08]{MR2484167}
A.~E. Kyprianou and J.~C. Pardo, \emph{Continuous-state branching processes and
  self-similarity}, J. Appl. Probab. \textbf{45} (2008), no.~4, 1140--1160.
  \MR{2484167}

\bibitem[Kyp14]{MR3155252}
Andreas~E. Kyprianou, \emph{Fluctuations of {L}\'{e}vy processes with
  applications}, second ed., Universitext, Springer, Heidelberg, 2014,
  Introductory lectures. \MR{3155252}

\bibitem[Lab14a]{MR3227064}
Cyril Labb\'{e}, \emph{From flows of {$\Lambda$}-{F}leming-{V}iot processes to
  lookdown processes via flows of partitions}, Electron. J. Probab. \textbf{19}
  (2014), no. 55, 49. \MR{3227064}

\bibitem[Lab14b]{MR3224288}
\bysame, \emph{Genealogy of flows of continuous-state branching processes via
  flows of partitions and the {E}ve property}, Ann. Inst. Henri Poincar\'{e}
  Probab. Stat. \textbf{50} (2014), no.~3, 732--769. \MR{3224288}

\bibitem[Lam67a]{MR0208685}
John Lamperti, \emph{Continuous state branching processes}, Bull. Amer. Math.
  Soc. \textbf{73} (1967), 382--386. \MR{0208685}

\bibitem[Lam67b]{MR0217893}
\bysame, \emph{The limit of a sequence of branching processes}, Z.
  Wahrscheinlichkeitstheorie und Verw. Gebiete \textbf{7} (1967), 271--288.
  \MR{0217893}

\bibitem[Lam03]{MR2014270}
Amaury Lambert, \emph{Coalescence times for the branching process}, Adv. in
  Appl. Probab. \textbf{35} (2003), no.~4, 1071--1089. \MR{2014270}

\bibitem[Lam07]{MR2299923}
\bysame, \emph{Quasi-stationary distributions and the continuous-state
  branching process conditioned to be never extinct}, Electron. J. Probab.
  \textbf{12} (2007), no. 14, 420--446. \MR{2299923}

\bibitem[Le14]{MR3189452}
V.~Le, \emph{Coalescence times for the {B}ienaym\'{e}-{G}alton-{W}atson
  process}, J. Appl. Probab. \textbf{51} (2014), no.~1, 209--218. \MR{3189452}

\bibitem[LGLJ98]{MR1617047}
Jean-Francois Le~Gall and Yves Le~Jan, \emph{Branching processes in {L}\'evy
  processes: the exploration process}, Ann. Probab. \textbf{26} (1998), no.~1,
  213--252. \MR{1617047}

\bibitem[Li00]{MR1727226}
Zeng-Hu Li, \emph{Asymptotic behaviour of continuous time and state branching
  processes}, J. Austral. Math. Soc. Ser. A \textbf{68} (2000), no.~1, 68--84.
  \MR{1727226}

\bibitem[Li11]{MR2760602}
Zenghu Li, \emph{Measure-valued branching {M}arkov processes}, Probability and
  its Applications (New York), Springer, Heidelberg, 2011. \MR{2760602}

\bibitem[LP13]{MR3059232}
Amaury Lambert and Lea Popovic, \emph{The coalescent point process of branching
  trees}, Ann. Appl. Probab. \textbf{23} (2013), no.~1, 99--144. \MR{3059232}

\bibitem[LPLG08]{MR2409319}
Yangrong Li, Anthony~G. Pakes, Jia Li, and Anhui Gu, \emph{The limit behavior
  of dual {M}arkov branching processes}, J. Appl. Probab. \textbf{45} (2008),
  no.~1, 176--189. \MR{2409319}

\bibitem[LUB17]{LambertUribe}
A.~Lambert and G.~Uribe~Bravo, \emph{The comb representation of compact
  ultrametric spaces}, p-Adic Numbers, Ultrametric Analysis and Applications
  \textbf{9} (2017), no.~1, 22--38.

\bibitem[M{\relax\"o}{}h15]{MR3333734}
Martin M{\relax\"o}{}hle, \emph{The {M}ittag-{L}effler process and a scaling
  limit for the block counting process of the {B}olthausen-{S}znitman
  coalescent}, ALEA Lat. Am. J. Probab. Math. Stat. \textbf{12} (2015), no.~1,
  35--53. \MR{3333734}

\bibitem[Mor58]{MR0127989}
P.~A.~P. Moran, \emph{Random processes in genetics}, Proc. Cambridge Philos.
  Soc. \textbf{54} (1958), 60--71. \MR{0127989}

\bibitem[MS01]{MR1880231}
Martin M{\relax\"o}{}hle and Serik Sagitov, \emph{A classification of
  coalescent processes for haploid exchangeable population models}, Ann.
  Probab. \textbf{29} (2001), no.~4, 1547--1562. \MR{1880231}

\bibitem[Par08]{MR2507503}
\'{E}tienne Pardoux, \emph{Continuous branching processes: the discrete hidden
  in the continuous}, ARIMA Rev. Afr. Rech. Inform. Math. Appl. \textbf{9}
  (2008), 211--229. \MR{2507503}

\bibitem[Pit97]{MR1466546}
Jim Pitman, \emph{Partition structures derived from {B}rownian motion and
  stable subordinators}, Bernoulli \textbf{3} (1997), no.~1, 79--96.
  \MR{1466546}

\bibitem[Pit99]{Pitman}
\bysame, \emph{Coalescents with multiple collisions}, Ann. Probab. \textbf{27}
  (1999), no.~4, 1870--1902. \MR{1742892 (2001h:60016)}

\bibitem[Pop04]{MR2100386}
Lea Popovic, \emph{Asymptotic genealogy of a critical branching process}, Ann.
  Appl. Probab. \textbf{14} (2004), no.~4, 2120--2148. \MR{2100386}

\bibitem[Pro56]{MR0084896}
Yu.~V. Prokhorov, \emph{Convergence of random processes and limit theorems in
  probability theory}, Teor. Veroyatnost. i Primenen. \textbf{1} (1956),
  177--238. \MR{0084896}

\bibitem[PY82]{MR656509}
Jim Pitman and Marc Yor, \emph{A decomposition of {B}essel bridges}, Z.
  Wahrsch. Verw. Gebiete \textbf{59} (1982), no.~4, 425--457. \MR{656509}

\bibitem[Sag99]{MR1742154}
Serik Sagitov, \emph{The general coalescent with asynchronous mergers of
  ancestral lines}, J. Appl. Probab. \textbf{36} (1999), no.~4, 1116--1125.
  \MR{1742154}

\bibitem[Sch00]{MR1781024}
Jason Schweinsberg, \emph{Coalescents with simultaneous multiple collisions},
  Electron. J. Probab. \textbf{5} (2000), Paper no. 12, 50. \MR{1781024}

\bibitem[Sch03]{MR1983046}
\bysame, \emph{Coalescent processes obtained from supercritical
  {G}alton-{W}atson processes}, Stochastic Process. Appl. \textbf{106} (2003),
  no.~1, 107--139. \MR{1983046}

\bibitem[Sie76]{MR0431386}
D.~Siegmund, \emph{The equivalence of absorbing and reflecting barrier problems
  for stochastically monotone {M}arkov processes}, Ann. Probability \textbf{4}
  (1976), no.~6, 914--924. \MR{0431386}

\bibitem[Sil68]{MR0226734}
M.~L. Silverstein, \emph{A new approach to local times}, J. Math. Mech.
  \textbf{17} (1967/1968), 1023--1054. \MR{0226734}

\bibitem[SY84]{MR738769}
Ken-iti Sato and Makoto Yamazato, \emph{Operator-self-decomposable
  distributions as limit distributions of processes of {O}rnstein-{U}hlenbeck
  type}, Stochastic Process. Appl. \textbf{17} (1984), no.~1, 73--100.
  \MR{738769}

\end{thebibliography}
\newcommand{\etalchar}[1]{$^{#1}$}
\providecommand{\bysame}{\leavevmode\hbox to3em{\hrulefill}\thinspace}
\providecommand{\MR}{\relax\ifhmode\unskip\space\fi MR }
\providecommand{\MRhref}[2]{%
  \href{http://www.ams.org/mathscinet-getitem?mr=#1}{#2}
}
\providecommand{\href}[2]{#2}

\end{document}